\documentclass{article}

\usepackage[utf8]{inputenc}
\usepackage[T1]{fontenc}
\usepackage[french,ngerman,english]{babel}

\usepackage{geometry}
\usepackage[normalem]{ulem} 
\usepackage{amsmath}
\usepackage{amssymb}
\usepackage{amsthm}

\usepackage{enumitem} 

\usepackage[colorlinks, pdfpagelabels, pdfstartview = FitH, bookmarksopen=true, bookmarksnumbered = true, linkcolor = black,plainpages = false, hypertexnames = false, citecolor = black]{hyperref} 

\usepackage{tabularx} 
\usepackage{graphicx}
\usepackage{tikz,tikz-3dplot}
\usetikzlibrary{patterns}
\usepackage{pgffor}
\usepackage{pgfplots}

\usepackage{float}

\usepackage{algorithm}	
\usepackage{algpseudocode} 

\usepackage{caption}
\usepackage{subcaption} 

\usepackage{pdfpages} 

\usepackage{cite}
\usepackage{authblk} 

\usepackage[color, arrow, matrix, curve]{xy}

\newcommand{\im}{\mathrm{im}}
\newcommand{\dx}[1]{\; \mathrm{d} #1}

\newcommand{\argmin}[1]{\underset{ #1 }{\mathrm{arg}\,\mathrm{min}}}
\newcommand{\sd}{\: : \:}
\newcommand{\supp}{\mathrm{supp}\:}

\newcommand{\defi}{\: \mathrel{\mathop{\raisebox{1pt}{\scriptsize$:$}}}= \:}
\newcommand{\restr}{\mathbin{\vrule height 1.6ex depth 0pt width 0.13ex\vrule height 0.13ex depth 0pt width 1.3ex}}

\def\Xint#1{\mathchoice
{\XXint\displaystyle\textstyle{#1}}%
{\XXint\textstyle\scriptstyle{#1}}%
{\XXint\scriptstyle\scriptscriptstyle{#1}}%
{\XXint\scriptscriptstyle\scriptscriptstyle{#1}}%
\!\int}
\def\XXint#1#2#3{{\setbox 0=\hbox{$#1{#2#3}{\int}$}
\vcenter{\hbox{$#2#3$}}\kern-.5\wd0}}
\def\dashint{\Xint-}

\newcommand{\dof}{\mathrm{dof}}

\newcommand{\dif}{\mathrm{div}}

\newcommand{\curl}{\mathrm{curl}}
\newcommand{\dist}{\mathrm{dist}}

\newcommand{\id}{\mathrm{Id}}
\newcommand{\bv}{\mathrm{BV}}

\newcommand{\N}{\mathcal{N}}
\newcommand{\E}{\mathcal{E}}

\newcommand{\C}{\mathcal{C}}

\newcommand{\M}{\mathcal{M}}

\newcommand{\T}{\mathcal{T}}

\newcommand{\V}{\mathcal{V}}
\newcommand{\D}{\mathcal{D}}
\renewcommand{\P}{\mathcal{P}}
\renewcommand{\L}{\mathcal{L}}
\newcommand{\F}{\mathcal{F}}

\newcommand{\MM}{\mathbb{M}}

\newcommand{\HH}{\mathbb{H}}

\newcommand{\cst}{\mathsf{C}} 

\newcommand{\intc}{\mathbf{C}}
\newcommand{\rca}{\mathbf{M}}

\newcommand{\ee}{\mathbf{e}}

\newcommand{\fePP}[1]{\mathbb{P}^{#1}}
\newcommand{\feP}[1]{P^{#1}}
\newcommand{\feNed}{\mathrm{Ned}}
\newcommand{\feRT}{\mathrm{RT}}

\algnewcommand\algorithmicinit{\textbf{Initialize:}}
\algnewcommand\Init{\item[\algorithmicinit]}
\algnewcommand\algorithmicparam{\textbf{Parameter:}}
\algnewcommand\Param{\item[\algorithmicparam]}
\algnewcommand\algorithmicout{\textbf{Output:}}
\algnewcommand\Out{\item[\algorithmicout]}

\newtheorem{theorem}{Theorem}[section]
\newtheorem{definition}[theorem]{Definition}
\newtheorem{proposition}[theorem]{Proposition}
\newtheorem{lemma}[theorem]{Lemma}
\newtheorem{remark}[theorem]{Remark}

\newtheoremstyle{name_and_space}{11pt}{20pt}{\itshape}{}{\bfseries}{}{.5em}{#1 #2 #3}
\theoremstyle{name_and_space}

\newif\ifJournal
\Journalfalse

\begin{document}

\title{A finite element approach for minimizing \\ line and surface energies arising in the study of \\ singularities in liquid crystals}

\date\today
\author[1]{Dominik Stantejsky}
\affil[1]{Universit\'e de Lorraine, Institut \'Elie Cartan de Lorraine, UMR 7502 CNRS,  54506 Vand\oe uvre-l\`es-Nancy Cedex, France}


\parskip 6pt

\date\today 

\parskip 6pt

\maketitle

\begin{abstract}
Motivated by a problem originating in the study of defect structures in nematic liquid crystals, we 
describe and study a numerical algorithm for the resolution of a Plateau-like problem. The
energy contains the area of a two-dimensional surface $T$ and the length of its boundary $\partial T$ reduced by a prescribed curve to make our problem non-trivial. 
We additionally include an obstacle $E$ for $T$ and pose a surface energy on $E$.
We present an algorithm based on the Alternating Direction Method of Multipliers that minimizes a discretized version of the energy using finite elements, generalizing existing TV-minimization methods.
We study different inclusion shapes demonstrating the rich structure of minimizing configurations and provide physical interpretation of our findings for colloidal particles in nematic liquid crystal .
\linebreak
\end{abstract}

\tableofcontents

\section{Introduction}
\label{sec:intro}

In this paper we propose a numerical algorithm for the solution of a geometric problem consisting in finding a two-dimensional surface $T$ in $\mathbb{R}^3$ which minimizes the energy $\E_0$ given by
\begin{align} \label{intro:eq:E0}
\mathcal{E}_0(T) &= \mathbb{M}(T\restr\Omega) + \int_\mathcal{M} |\nu_3| \:\mathrm{d}\mu_{T\restr\mathcal{M}} + \beta\: \mathbb{M}(\partial T + \Gamma)  \, .
\end{align} 
The energy $\E_0$ consists of the following three parts:
The first contribution to the energy is given by the (two-dimensional) mass of $T$ (i.e.\ the surface area) outside an inclusion $E\subset\mathbb{R}^3$ and $\Omega$ being $\Omega=\mathbb{R}^3\setminus\Omega$.
The second part is the integral of a density integrated over the part of $T$ on the inclusion surface $\M\defi\partial E$.
In our case the density is given by $|\nu_3|$, the absolute value of the $3-$component of the normal vector field on $\M$.
The last term is the (one-dimensional) mass (i.e.\ the length) of the boundary $\partial T$ reduced by a given prescribed curve $\Gamma$ and weighted by a parameter $\beta\in (0,\infty)$.
The latter are essential since for $\Gamma=\emptyset$ or $\beta=0$ the minimizer is trivial. 
A natural domain for this energy is the space of so called \emph{currents} or \emph{flat chains} of dimension $k$, denoted $\F^k$, i.e.\ $T\in \F^2$ and $\partial T,\Gamma\in\F^1$.
We refer the reader to \cite{Federer1960,Simon1983,Morgan2016} 
for details on currents and \cite{Fleming1966,White1999a} for flat chains.

Our motivation for considering the energy in \eqref{intro:eq:E0} comes from \cite{ACS2021,ACS2024} in which it has been shown that $\E_0$ describes the asymptotic behaviour of point and line singularities in the Landau-de Gennes model for nematic liquid crystals around an inclusion when a homogeneous external magnetic field is applied in $\ee_3-$direction.
More precisely, point defects are located on $T$ which also incorporates surface effects coming from the inclusion $E$ through the integral over $\M$.
The geometric objects $T,\partial T, \Gamma$ are flat chains (of dimension $2$ or $1$ respectively) with values in the coefficient group $\mathbb{Z}_2$, in particular the addition of two objects of the same dimension is well defined.
This is the coefficient group we're focusing on in this article.
As shown in \cite{ACS2024}, the singular set of lines corresponds to $S := \partial T + \Gamma$, where $\Gamma=\{\nu\cdot\ee_3=0\}\subset\M$, so that $\MM(\partial T + \Gamma)$ measures the length of the line singularitiy $S$. 
The parameter $\beta$ acts as a weight for penalizing deviations of $\partial T$ from $\Gamma$ and originates in the coupling between magnetic, bulk and elastic forces in the asymptotic liquid crystal model.
Note that compared to the limit energy in \cite{ACS2024}, the energy in \eqref{intro:eq:E0} is missing the constant $C_\M$ which only depends on the shape $\M$ relative to the direction of the magnetic field.
As this constant does not influence our analysis but only need to be added when comparing energies of different orientations, we do not account for it until Section~\ref{sec:results} when non-symmetric particles are discussed.

While the interpretation in terms of point and line singularities is our main interest in studying minimizers of \eqref{intro:eq:E0}, it is worth noting that our problem is closely related to the \emph{Plateau problem} and the \emph{obstacle problem} and can be seen as a generalization of both. 
The obstacle problem consists in finding a surface $T$ with minimal area outside a given obstacle $E$ such that $\partial T$ is given by a fixed curve $\Gamma$.
Setting the integrated density on $T\restr\M$ equal to one and taking $\beta\gg 1$ to force $\partial T = \Gamma$, one can see minimizers of \eqref{intro:eq:E0} as solutions to the obstacle problem.
If again $\beta\gg 1$ to ensure $\partial T = \Gamma$ and taking $E=\emptyset$, the energy $\E_0$ reduces to the classical Plateau problem which asks for the surface $T$ of minimal area spanning a given curve $\Gamma$.
This problem dates back to Lagrange \cite{Lagrange1761} and solutions, so called \emph{minimal surfaces} have been studied extensively since then \cite{Douglas1931,Rado1933,Struwe2014}.
In some particular cases, analytic tools allow for explicit solutions or characterizations, see \cite{Hoffman1993,Karcher1989,Courant1977,Jost1991,Nitsche2011}.

Whenever the boundary and ambient space do not disclose an exact solution, the question of numerical approximation arises. 
Different approaches have been developed to represent surfaces and how to ensure their minimality. 
Most famous is the \emph{mean curvature flow} \cite{Dziuk1990,Brakke1992}, based on the vanishing mean curvature optimality condition for minimal surfaces. 
Although originally developed for hypersurfaces, the concept has been rapidly generalized to arbitrary codimension \cite{Ambrosio1996}.
While the standard approach consists in representing the surface via a level set, more recently varifolds constituted by point clouds are also used \cite{Buet2019}.
In addition to classical algorithms, it is also possible to employ machine learning techniques such as training neural networks to simulate a mean curvature flow \cite{Masnou2021}.
Another very successful ansatz is to model the surface as jump set of a $\bv-$function and therefore minimizing the total variation leads to minimal surfaces. 
The total variation is then discretized via finite difference or finite elements \cite{Condat2017,Herrmann2018,ChambollePock2021,Wang2021}, although other choices also prove to be useful \cite{Abergel2017}.

Compared to all of the aforementioned methods, the treatment of our limit energy $\E_0$ exhibits the major challenge of optimizing a surface \emph{and} its boundary simultaneously. 
While mean curvature flows of both objects independently are known and implemented, their joint minimization might cause conflict due to possible contradicting movements close to the boundary. 
In the following, we will use an approach related to the minimization of the total variation generalizing the method in \cite{Wang2021}.

\section{Theoretical background}
\label{sec:theory}

As a first step, we describe a way to compute the mass $\MM$ of two- and one-dimensional objects that will allow as easy discretization later on.
Let $K$ be an open Lipschitz domain of finite measure in which we will state our problem.
The main simplification we are going to apply is to take the objects $T,\partial T,\Gamma$ in the space of currents rather than flat chains resulting in a convex optimization problem.
Then, we reformulate the problem 
in terms of vector fields.
If we assume that $T$ and $\partial T$ are regular enough, let $\nu_T$ be a unit normal vector field on $T$ and $\tau_{\partial T}$ a normal tangent vector field on $\partial T$ (with induced orientation).
Then, it holds that
\begin{align*}
\MM(T)
\ &= \ \int_T 1\dx\mathcal{H}^2
\ = \ \sup_{\substack{p\in L^\infty(K,\mathbb{R}^3) \\ \Vert p\Vert_{L^\infty}\leq 1}} \int_T p\cdot \nu_T \dx\mathcal{H}^2\, ,
\end{align*}
and 
\begin{align*}
\MM(\partial T)
\ &= \ \sup_{\substack{q\in L^\infty(K,\mathbb{R}^3) \\ \Vert q\Vert_{L^\infty}\leq 1}} \int_{\partial T} q\cdot \tau_{\partial T} \dx\mathcal{H}^1\, .
\end{align*}
By Stokes' Theorem it holds that (at least formally, if $q$ is smooth enough)
\begin{align*}
\int_{\partial T} q\cdot \tau_{\partial T} \dx x
\ &= \ \int_T \curl(q)\cdot \nu_T \dx x \, .
\end{align*}
If we think of $u$, a $\mathbb{R}^3-$valued measure compactly supported on $K$, i.e.\ $u\in\rca(K)^3$, as representation of $T$, the above reasoning justifies the definition  of
\begin{align}\label{def:C_func_meas}
\intc(u,K)
\ \defi \
\sup\Big\{ \int_K \curl(\phi)\cdot\dx u \sd \phi\in C_c^\infty(K,\mathbb{R}^3),\, \Vert \phi\Vert_\infty\leq 1 \Big\}\, ,
\end{align}
as proxy for the mass of the boundary $\partial T$.
In order to incorporate $\Gamma$, we need to replace $\partial T$ in the computation by $\partial T + \Gamma$ and to add the tangent vector field of $\Gamma$.
If formally $\curl(u_0)$ is the vector-valued measure $\tau_\Gamma \mathcal{H}^1\restr\Gamma$, for some $u_0\in\rca(K)^3$, then we can 
write
\begin{align*}
\MM(\partial T + \Gamma)
\ = \
\intc(u+u_0,K) 
\, .
\end{align*}
It remains to include the particle $E\subset K$ and the energy coming from its surface $\M=\partial E$.
Let $\Omega\defi K\setminus \overline{E}$.
We then write
\begin{align} \label{pb:lim_currents_Gam_M}
\inf_{\substack{u\in \rca(K)^3}} \int_\Omega \dx\Vert u\Vert 
\ + \ \int_\M |\nu_3| \dx\Vert u\Vert
\ + \ \beta \ \intc(u+u_0,K) \, .
\end{align}

It is not imminent that the minimization problem in \eqref{intro:eq:E0} (or \eqref{pb:lim_currents_Gam_M}) is well posed. 
This is due to the fact that $|\nu_3|=0$ on $\Gamma$, and hence a minimizing sequence of \eqref{intro:eq:E0} might accumulate mass on $\Gamma$ and the limiting object is not a finite mass current (or measure) any more despite being of finite energy.

\begin{proposition}[Minimizers have finite mass]
\label{prop:well_posed_pb}
A minimizers of \eqref{intro:eq:E0} exist in the class of finite mass currents.
\end{proposition}

\begin{proof}
In order to show existence it is enough to show that for a minimizing sequence $T_k$ of $\E_0$ the mass is bounded uniformly in $k$.
Indeed, since the mass of the boundary of $T_k$ is uniformly bounded in $k$ via the energy, we can then apply the compactness theorem for currents \cite[4.2.17]{Federer1996} and pass to the limit in $k$.

The only location where potentially the mass of $T_k$ could blow up is on the line $\Gamma$, since everywhere else the mass is controlled by the energy.
If there is no danger of confusion, we suppress the $k$ in the notation to improve readability.

Let $\varepsilon>0$ be arbitrary, but small.
Since the total energy $\E_0(T)$ is bounded, so is the energy in the annular region $\{\dist(\cdot,\M)\in (\varepsilon,2\varepsilon)\}$ and via slicing from \cite[Corollary~3.10]{Federer1960} we can choose a radius $\varepsilon_*\in (\varepsilon,2\varepsilon)$ such that 
\begin{align*}
\MM(\partial(T\cap\{\dist(\cdot,\M)\leq\varepsilon_*\}))
+
\MM(\partial(\partial T\cap\{\dist(\cdot,\M)\leq\varepsilon_*\}))
\ &\leq \
\frac{C}{\varepsilon_*}\E_0(T)
\, .
\end{align*}
We deform $T$ into $\widetilde{T}$ by replacing $T\restr\{\dist(\cdot,\M)\leq\varepsilon_*\}$ by its projection onto $\M$, called $\Pi_\M(T\restr\{\dist(\cdot,\M)\leq\varepsilon_*\})$, and adding the surface connecting $T\restr\{\dist(\cdot,\M)=\varepsilon_*\}$ along $\nabla\dist(\cdot,\M)$ with $\Pi_\M(T\restr\{\dist(\cdot,\M)=\varepsilon_*\})$.
Thus, 
\begin{align*}
\MM(\Pi_\M(\widetilde{T}))
\ &\leq \
C\MM(T)
\, ,
\quad
\MM(\Pi_\M(\partial\widetilde{T}))
\ \leq \
C\MM(\partial T)
\, ,
\end{align*}
which implies that
\begin{align*}
\MM(\widetilde{T})
\ &\leq \
C\E_0(T)
+
\frac{C}{\varepsilon_*} \varepsilon_* \E_0(T) 
\, ,
\quad
\E_0(\widetilde{T})
\ \leq \
C\E_0(T)
+
\frac{C}{\varepsilon_*} \varepsilon_* \E_0(T)
\, .
\end{align*}
Now we repeat the slicing argument for $\widetilde{T}\restr\M$ to find a distance $\varepsilon_{**}\in(\varepsilon,\varepsilon_*)$ such that 
\begin{align*}
\MM(\partial(\widetilde{T}\restr\{\dist_\M(\cdot,\Gamma)\leq\varepsilon_{**}\}))
+
\MM(\partial(\partial \widetilde{T}\restr\{\dist_\M(\cdot,\Gamma)\leq\varepsilon_{**}\}))
\ &\leq \
\frac{C}{\varepsilon_{**}}\E_0(\widetilde{T}\restr\M)
\, .
\end{align*}
With the same argument as in \cite[5.1]{Pauw2022}, one can conclude that 
\begin{align*}
\MM(\widetilde{T}\restr\{\dist_\M(\cdot,\Gamma)\leq\varepsilon_{**}\})
\ &\leq \
C \MM(\partial(\widetilde{T}\restr\{\dist_\M(\cdot,\Gamma)\leq\varepsilon_{**}\})) 
\, ,
\end{align*}
and thus the mass of $\widetilde{T}\restr\{\dist_\M(\cdot,\Gamma)\leq\varepsilon_{**}\}$ is bounded, so no accumulation of mass on $\Gamma$ is possible.
\end{proof}

\begin{remark}[Relation to total variation minimization]
\label{rem:total_var_min}
The fundamental idea in what follows is to represent $T$ by a vector field $u$ and its boundary by $\curl(u)$.
Inside $T$, $\curl(u)$ vanishes and $u$ can be seen as a (local) gradient of some scalar potential $\phi$.
Minimizing the mass of $T$ (represented by the $L^1-$norm of $u$) and hence the $L^1-$norm of $\nabla\phi$ can be interpreted as minimizing the total variation of $\phi$.
\end{remark}

The problem of minimizing a total variation has attracted a lot of attention in recent years \cite{Condat2017, Herrmann2018, Abergel2017, ChambollePock2021}.
Our problem differs since we consider a general vector field $u$ (in contrast to the classical total variation problems where $u=\nabla\phi$) and we have an additional $L^1-$term in our minimization that depends on the derivative of $u$.
The latter fact contrasts with problems e.g. from image reconstruction where terms like $\Vert \phi - \phi_0 \Vert_{L^2}^2$ are considered.
We point out that the subject of \cite{Wang2021} is the case when $\curl(u) = \curl(u_0)$ is prescribed. 
This can be seen as a total variation minimization for solving the classical Plateau problem.

\begin{remark}[Differential forms]
\label{rem:diff_forms}
In the smooth case, problem \eqref{pb:lim_currents_Gam_M} can also be stated via differential forms.
Let $\Omega^k(K)$ denote the space of differential $k-$forms. 
We can identify $\Omega^1(K)$ with vector fields in $K$ and \eqref{pb:lim_currents_Gam_M} can be reformulated as finding a $1-$form $\omega_T$ with $\int_K \omega_T\wedge \eta = \int_T\eta$ for all $\eta\in \Omega^2(K)$ minimizing the mass norm 
\begin{align*}
\Vert \omega_T\Vert_{\mathrm{Mass}} 
\ = \ 
\sup_{\substack{\eta\in\Omega^2(K) \\ \Vert\eta\Vert_\infty\leq 1}} \int_K \eta\wedge\omega_T \, .
\end{align*}
The boundary $\partial T$ corresponds to a $2-$form $\omega_{\partial T}$ with $\omega_{\partial T} = \dx \omega_T$. Indeed, by Stokes' Theorem it holds that for any $\eta\in \Omega^1(K)$ which vanishes on $\partial K$
\begin{align*}
\int_K \eta\wedge \omega_{\partial T}
\ = \ \int_{\partial T} \eta
\ &= \ \int_T \dx\eta
\ = \ \int_K \dx\eta\wedge \omega_T
\ = \ \int_{K} \eta\wedge \dx\omega_T \, .
\end{align*}
Hence, \eqref{pb:lim_currents_Gam_M} is essentially the same as to minimize $\Vert \omega\Vert_{\mathrm{Mass}} + \Vert|\nu_3| \, \iota_\M^{*}\omega\Vert_{\mathrm{Mass}} + \beta \Vert \hspace{-1.5mm}\dx\omega + \dx\omega_0\Vert_{\mathrm{Mass}} $ among all $\omega\in\Omega^1(K)$, where $\omega_0$ is a $1-$form corresponding to the boundary $\Gamma$ and $\iota_\M:\M\to K$ is an embedding, $\iota_\M^{*}\omega$ denoting the pull back of $\omega$ to a $1-$form on $\M$.
\end{remark}

\section{Numerical simulation}
\label{sec:numerics}

In order to solve Problem \eqref{pb:lim_currents_Gam_M} we use a finite element discretization and then employ an Alternating Direction Method of Multipliers (ADMM) algorithm for the optimization.

\subsection{Finite element discretization}
\label{subsec:fe}

We start by introducing the finite element spaces that we will use in the further course of this chapter.
The definition we give here is standard and can be found in many books on finite elements, for example \cite{Ern2021} or \cite[Ch. 3]{Fenics2012}.

\begin{definition}[Finite element spaces]
Let $\T_h$ be a tetrahedral mesh of $K$ consisting of a set of nodes $\N_h$, edges $\E_h$, facets $\F_h$ and tetrahedra $\T_h$.
We call a \emph{finite element space} a finite dimensional function space $\V$ (in our case a subspace of polynomials) defined on a domain $\D$ (here the tetrahedra of $\T_h$) and a set of degrees of freedom $\L$ which form a basis of the dual space of $\V$ (unisolvence).
For $T\in\T_h$, let $\P_q(T)$ be the space of polynomials on $T$ with degree smaller or equal to $q\geq 0$ and $\overline{\P}_q(T)$ the polynomials on $T$ with degree equal to $q$.
\begin{enumerate}
\item The \emph{Lagrange-element} of order $1$ is given by
\begin{align*}
\D = T\, , \qquad \V = \P_1(T)\, ,
\end{align*}
where $T\in \T_h$ and $\L$ consists of function evaluations at the nodes of $T$.
We call $\feP 1$ the corresponding finite element space defined on $\T_h$.
\item The \emph{discontinuous Lagrange-element} of order $0$ is given by
\begin{align*}
\D = T\, , \qquad \V = \P_0(T) \, ,
\end{align*}
where $T\in \T_h$ and $\L$ consists of taking the average of a function over $T$.
The corresponding finite element space defined on $\T_h$ is denoted by $\feP 0$ and the vector valued version $\fePP{0}\defi(\feP 0)^3$.
\item Next, the \emph{Nédélec-element} of the first kind of order $1$ is given by
\begin{align*}
\D = T\, , \qquad \V = (\P_0(T))^3 + x\times\overline{\P}_0(T) \, ,
\end{align*}
for $T\in \T_h$. 
The degrees of freedom $\L$ for a function $v\in\V$ are given by the integrals of $v\cdot t_e$ over the edges $e\in\E_h$ that are included in the boundary $\partial T$, $t_e$ denoting the tangent vector of $e$.
The finite element space consisting of Nédélec-elements of order $0$ is denoted by $\feNed$.
\item We define the \emph{Raviart-Thomas-element} of order $0$ as
\begin{align*}
\D = T\, , \qquad \V = (\P_0(T))^3 + x\overline{\P}_0(T) \, ,
\end{align*}
for $T\in \T_h$. 
The degrees of freedom $\L$ for a function $v\in\V$ are given by the integrals of $v\cdot n_F$ over the facets $F\in\F_h$ that are included in the boundary $\partial T$, $n_F$ denoting the normal vector of $F$.
The finite element space consisting of Raviart-Thomas-elements of order $0$ is denoted by $\feRT$.
\end{enumerate}
Each finite element space $V$ comes with a projection operator $\Pi_V$ which allows us to pass from functions defined on $\Omega$ to finite element approximations over $\T_h$.
\end{definition}

We note that $\nabla v\in\feNed$ if $v\in\feP 1$, $\curl(v)\in\feRT$ if $v\in\feNed$ and $\dif(v)\in\feP 0$ if $v\in\feRT$.
More precisely, we have the following proposition:

\begin{proposition}\label{prop:exact_seq_FE}
Let $K$ be an open, simply connected Lipschitz domain in $\mathbb{R}^3$.
Then the diagram 
  \[
\begin{xy}
  \xymatrix{
\mathbb{R} \ar[r] & H^1 \ar[d]^{\Pi_{\feP 1}} \ar[r]^{\nabla} & H(\curl)\ar[d]^{\Pi_{\feNed}} \ar[r]^{\curl} & H(\dif) \ar[d]^{\Pi_\feRT} \ar[r]^{\dif} & L^2 \ar[d]^{\Pi_{\feP 0}} \ar[r] & 0 \\
\mathbb{R} \ar[r] & \feP{1} \ar[r]^{\nabla} & \feNed \ar[r]^{\curl} & \feRT \ar[r]^{\dif} & \feP{0} \ar[r] & 0
    }
\end{xy}
\]
commutes and the rows are short exact sequences.
\end{proposition}

\begin{proof}
The upper row is just the usual de Rham-complex for a simply connected Lipschitz domain in $\mathbb{R}^3$.
The lower row is its discrete analogue, see Prop.\ 16.15 in \cite{Ern2021}.
For the commutation properties, we refer to Chapter 16.1.2 and Chapter 16.2.2 and Lemma 16.16 in \cite{Ern2021}.
\end{proof}

Since $T$ and $\partial T$ are measures, we cannot assume much regularity on these objects and need to discretize them with finite elements of lowest possible order, e.g.\ $\fePP 0-$elements. 
We therefore choose to represent $u$ as a Nédélec function, noting that its curl belongs to $\feRT$ and can be seen as a subset of $\fePP 0$.
We thus rewrite 
$\MM(T)$ and $\MM(\partial T + \Gamma)$ as
\begin{align}\label{pb:fe_Ned_PP0}
\int_K |\Pi_{\feP 0}(u)| \dx x 
\quad\text{and}\quad
\int_K |\curl(u) + \curl(u_0)| \dx x
\, ,
\end{align}
where $u_0\in \feNed$ is an approximation for the tangential vector field $\tau_\Gamma$.

Note that the choice of Nédélec elements is not canonic and other choices are possible.
For example in \cite{Chambolle2020}, Crouzeix-Raviart elements are used for a total variation minimization problem.
The approximation of the solutions obtained there is sharper compared to the Nédélec/$P^1$ elements.
Also other finite elements could be used as long as they verify a variant of Proposition~\ref{prop:exact_seq_FE} and in particular the image and kernel of the discrete $\curl-$operator are known, see \cite[Ch. 7]{Arnold2018}.
In order to decide which elements to use one must then balance the quality of the approximations for $T$ and $S$ given by the size of $\ker(\curl)$ and $\im(\curl)$, as well as the computational cost required by the finite elements.

\begin{remark}[Relation to discrete total variation minimization]
\label{rem:total_var_min_discr}
In view of Proposition~\ref{prop:exact_seq_FE}, Remark~\ref{rem:total_var_min} can be translated into the finite element setting as follows:
If $u\in\feNed$ such that $\curl(u)=0\in\feRT$, then by exactness of the sequence there exists $\phi\in\feP 1$ such that $u=\nabla\phi$ and $\Vert u\Vert_{L^1}=TV(\phi)$.
\end{remark}

In order to find a finite element representation of \eqref{pb:lim_currents_Gam_M}, we introduce a density function $\rho_\delta\in C^0(\overline{\Omega})$, defined for $0<\delta\ll 1$ and $x\in\overline{\Omega}$ as $\rho_\delta(x)=|\nu_3(\Pi_\M(x))|$ if $\dist(x,\M)\leq \delta$ and $\rho_\delta(x)=1$ if $\dist(x,\M)>2\delta$.
The parameter $\delta$ determines the size of the boundary layer $\M_\delta \defi \{x\in\overline{\Omega}\sd x=\omega+r\nu(\omega)\text{ for }\omega\in\M,r\in [0,\delta]\}$ which will serve as proxy for $\M$.
We choose the approximation $u_{0,\delta}$ of $u_0$ to be centred at distance $\frac{\delta}{2}$ from $\M$ and supported in $\M_\delta\setminus\M_{\delta/4}$ such that $u_{0,\delta,h}\stackrel{*}{\rightharpoonup} u_0$ in $\rca(K)^3$.
We furthermore use the notation $\Omega_\delta\defi \Omega\setminus \M_\delta$.
For $u_h\in\feNed$ and $u_{0,\delta,h}\defi\Pi_{\feNed}(u_{0,\delta})$ we then define
\begin{align}\label{def:func_Ehd}
\E_{h,\delta}(u_h)
\ &= \
\begin{cases}
\int_K \rho_\delta|\Pi_{\feP 0}(u_h)| \dx x \ + \ \beta \int_K |\curl(u_h) + \curl(u_{0,\delta,h})| \dx x & \text{if } u_h\equiv 0 \text{ in }E\, , \\
+\infty & \text{otherwise.}
\end{cases}
\end{align}

We next show that if the boundary layer $\M_\delta$ is exactly one layer of cells in the mesh $\mathcal{T}_h$ (in particular $\delta=h$), then the functional $\E_{h,\delta}$ approaches \eqref{pb:lim_currents_Gam_M} in a variational sense when $h\rightarrow 0$.

The fact that $\M_\delta$ has a thickness of only one cell is crucial to ensure compactness as stated in the following lemma:

\begin{lemma}\label{lem:num_isoperim_ineq}
Let $\mathcal{T}_h$ be a mesh of $\Omega$ such that $\M_\delta$ has a thickness of one cell, i.e.\ all cells in $\M_\delta$ share at least one face (facet, edge or vertex) with $\M$.
Then, there exists a constant $\cst=\cst(\M)>0$ such that for all $0<h\ll 1$ and all functions $u_h\in\feNed$ with $\supp(u_h)\subset\M_\delta$ it holds
\begin{align}\label{lem:num_isoperim_ineq:eq}
\int_{\M_\delta} |u_h| \dx x
\ &\leq \ 
\cst \left(\int_{\M_\delta} |\curl(u_h)| \dx x
\ + \
\int_{A_\delta} |u_h| \dx x \right)
\, ,
\end{align}
where $A_\delta\subset \M_\delta$ is the boundary layer corresponding to an open set $A\subset \M$.
\end{lemma}

Note that both sides of \eqref{lem:num_isoperim_ineq:eq} could be infinite.

\begin{proof}
We argue by contradiction and assume that for all $n\in\mathbb{N}$ there exists $h_n>0$ and functions $u_{n}\in\feNed$ such that
\begin{align*}
\int_{\M_\delta} |u_{n}| \dx x
\ &> \ 
n 
\left(\int_{\M_\delta} |\curl(u_{n})| \dx x
\ + \
\int_{A_h} |u_{n}| \dx x \right)
\, .
\end{align*}
We start with the following observations:
\begin{enumerate}[label=(\roman*)]
\item The mass of $u_n$ on $A_h$ as well as the $L^1-$norm of the curl of $u_n$ in $\M_\delta$ is bounded uniformly in $n$, otherwise \eqref{lem:num_isoperim_ineq:eq} holds trivially with the RHS being infinite.
\item It holds $\int_{\M_\delta} |\curl(u_{n})| \dx x +\int_{A_\delta} |u_{n}| \dx x > 0$, otherwise $u_{n}\equiv 0$ on $\M_\delta$ and thus \eqref{lem:num_isoperim_ineq:eq} holds trivially with the LHS being zero.
Indeed, assume that $u_n\not\equiv 0$ and $\int_{\M_\delta} |\curl(u_{n})| \dx x=0$. 
Then $u_n=\nabla\phi$ for some $\phi\in P^1$ and  since the thickness of $\M_\delta$ is of one cell and $\supp(u_n)\subset\M_\delta$, such a $\phi$ must have trivial gradient on all edges included in $\partial\M_\delta$.
Thus all connected components of $\partial\M_\delta$ are level sets of $\phi$ and hence, since $u_n$ is non-trivial, $\int_{A_h} |u_{n}| \dx x>0$ in contradiction to the assumption.
\item Furthermore, we can assume that there exists a constant $c>0$ such that 
$$
\limsup_{n\to\infty}\int_{\M_\delta} |\curl(u_{n})| \dx x +\int_{A_\delta} |u_{n}| \dx x 
\geq 
c > 0 
\, .
$$
Otherwise, arguing similarly to the previous case, one can show that $u_{n}\rightarrow 0$ and we can find a constant $\cst>0$ such that \eqref{lem:num_isoperim_ineq:eq} holds, in contradiction to the assumptions.
\end{enumerate}
It therefore holds that there exists a subsequence of $u_n$ (not relabelled) such that the mass of $u_{n}$ in $\M_\delta$ diverges to $\infty$ as $n\to\infty$. 
We can therefore find cells $T_{n}\in \T_{h_n}\restr\M_\delta$ such that
\begin{align*}
\frac{1}{h^2}\int_{T_{n}} |u_n|\dx x \to \infty
\end{align*}
as $n\to \infty$.
Indeed, if for all cells in $\M_\delta$ the expression $\frac{1}{h^2}\int_{T_{n}} |u_n|\dx x$ was uniformly bounded, then, using that the number of cells in $\M_\delta$ is of order $\frac{1}{h^2}$, 
\begin{align*}
\int_{\M_\delta} |u_{n}| \dx x
\ &= \
\sum_{T\in\T_{h_n}\cap\M_\delta} \int_T |u_n| \dx x
\ \lesssim \
\sum_{T\in\T_{h_n}\cap\M_\delta} h^2
\ \lesssim \
1
\, ,
\end{align*} 
which contradicts the diverging mass of $u_n$.

If the mass of $u_n$ on $T_n$ is diverging, there exists an edge $e_n\in\partial T_n$, $e_n\nsubseteq\partial \M_\delta$ with tangent $\tau_{e_n}$ such that $|u_n\cdot\tau_{e_n}|\to\infty$ as $n\to\infty$.
We identify $u_n\cdot\tau_{e_n}$ with the degree of freedom of the Nédéléc space and write $|\dof(u_n,e_n)|\to\infty$.

Since the curl of $u_n$ is uniformly bounded and $e_n$ is part of at least three facets of the mesh, there exist edges $\tilde{e}_{n,1},\tilde{e}_{n,2}\in \partial T_n$, adjacent, but not equal to $e_n$, such that $|\dof(u_n,\tilde{e}_{n,i})|\to\infty$ as $n\to\infty$ for $i=1,2$.
Furthermore, the edges $\tilde{e}_{n,i}$ must lie in the interior of $\M_\delta$ because of the assumption on the support of $u_n$.

With these two new edges each being part of at least three facets (two facets who do not contain $e_n$), we can find new edges $\tilde{e}_{n,i}$, $i=3,4,5,6$ in the interior of $\M_\delta$ such that $\tilde{e}_{n,i}\notin \{e_n,\tilde{e}_{n,1},\tilde{e}_{n,2}\}$ and $|\dof(u_n,\tilde{e}_{n,i})|\to\infty$ as $n\to\infty$, see Figure~\ref{fig:num_isoperim_ineq-1layer}.
We observe that since the thickness of $\M_\delta$ is only one cell, there cannot be any regions in which the edges that were selected form a ``bubble'' e.g.\ by being all connected to the same vertex in $\M_\delta$. 
This would mean that the facets of diverging $u_n$ form a polyhedron (with at least one interior vertex) but our assumption of a one-layer thickness excludes this possibility since a ``bubble'' requires a thickness of at least two layers, see also the illustration in Figure~\ref{fig:num_isoperim_ineq-2layer}.

Repeating this procedure, it follows that $|\dof(u_n,e)|\rightarrow \infty$ for all edges $e$ in $\M_\delta$ as $n\to\infty$, a contradiction to the finiteness of $\int_{A_\delta}|u_n|\dx x$. 
\end{proof}

\begin{figure}[H]
	\begin{subfigure}[c]{0.49\textwidth}
	\centering
	\includegraphics[scale=1.2]{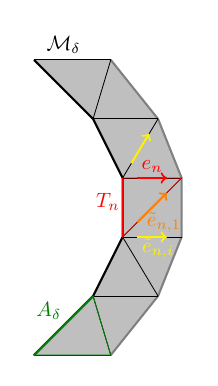}
    \caption{}
    \label{fig:num_isoperim_ineq-1layer}
    \end{subfigure}
	\begin{subfigure}[c]{0.49\textwidth}
	\centering
	\includegraphics[scale=1.2]{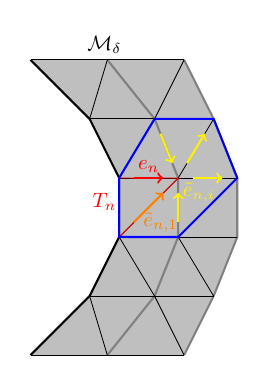}
    \caption{}
    \label{fig:num_isoperim_ineq-2layer}
    \end{subfigure}
\caption{\eqref{fig:num_isoperim_ineq-1layer} Visualization of choices of $e_n$ and $\tilde{e}_{n,i}$ in the proof of Lemma~\ref{lem:num_isoperim_ineq}. After finitely many steps an edge in $A_\delta$ is reached.
\eqref{fig:num_isoperim_ineq-2layer} Possible ``bubble'' in $\M_\delta$ having a thickness of \emph{two layers} which allows for diverging mass while edges stay localized in one part of $\M_\delta\setminus A_\delta$. Such configurations are not possible in our situation due to the assumption of a one-layer thickness of $\M_\delta$.}
\label{fig:num_isoperim_ineq_mesh_bubble}
\end{figure}

\begin{remark}\label{rem:num_isoperim_ineq}
\begin{enumerate}
\item The compactness lemma can be seen as a Poincaré-type inequality. Note that the term $\int_{A_h} |u_h|\dx x$ in the RHS of \eqref{lem:num_isoperim_ineq:eq} is only used once to deduce for the implication that if the RHS was zero, so was $u_h$. In this regard, one could replace $\int_{A_h} |u_h|\dx x$ by other terms, e.g.\ the sum of the degrees of freedom on selected edges. 
\item Note that an inequality such as \eqref{lem:num_isoperim_ineq:eq} cannot hold on general domains, see e.g.\ \cite{Ern2024}. 
This inequality is only possible due to the special geometry (i.e.\ thickness of one layer) of $\M_\delta$, allowing to rule out the possibility of $u_h$ being a gradient of a function supported outside of $A_\delta$.
\item Another way to see this lemma is as a finite element analogue of the continuous isoperimetric inequality in \cite{Pauw2022}. In fact, we could have also modified the statement of the lemma so that for $u_h\in\feNed$ there exists $v_h\in\feNed$ with $|\curl(u_h)|=|\curl(v_h)|$ and 
\begin{align*}
\int_{\M_h} |v_h| \dx x
\ &\leq \ 
\cst \int_{\M_h} |\curl(u_h)| \dx x
\, .
\end{align*}
\end{enumerate}
\end{remark}

\begin{proposition}\label{prop:num_G_conv}
Let $\mathcal{T}_h$ be a mesh of $\Omega$ such that $\M_\delta$ has a thickness of one cell. 
Then, the energy $\E_{h,\delta}$ $\Gamma-$converges to the energy in \eqref{pb:lim_currents_Gam_M} with respect to the weak*-topology of measures.
\end{proposition}

\begin{proof}
The proof is divided into three part: Compactness, lower bound and construction of a recovery sequence to prove the upper bound.

\bigskip

\noindent\textit{\underline{Compactness.}}
Let $u_h\in \feNed$ and $u_{0,\delta,h}\in\feNed$ be the previously described approximation for the tangential vector field $\tau_\Gamma$ with respect to the weak*-topology of vector valued measures such that $\E_{h,\delta}\leq C$.
We want to show that a subsequence of $u_h$ converges weakly* to a limit measure $u\in\rca(K)^3$.
We note that a function $v$ of the Nédélec space on a tetrahedron $T\in\mathcal{T}_h$ can be written as 
\begin{align}\label{prop:num_G_conv:Ned_on_tetra} 
v|_T(x) = v(c_T) + \frac12\curl(v|_T)\times (x-c_T) \, ,
\end{align}
where $c_T$ is the center of mass of the cell $T$, $\curl(v|_T)$ is a constant vector on each cell and $v(c_T) = \dashint_T v\dx x$.
We can hence estimate $|u_h|_T - \Pi_{P^0}(u_h)|_T| \leq C h |\curl(u_h)|_T|$ which gives
\begin{align*}
\int_K |u_h - \Pi_{P^0}(u_h)| \dx x
\ &= \ 
\sum_{T\in\mathcal{T}_h} \int_T |u_h|_T - \Pi_{P^0}(u_h)|_T| \dx x \\
\ &\leq \ 
C h \sum_{T\in\mathcal{T}_h} \int_T |\curl(u_h)|_T| \dx x 
\ = \ 
C h \int_K |\curl(u_h)| \dx x \, .
\end{align*}
This allows to conclude that $u_h - \Pi_{P^0}(u_h)\overset{*}{\rightharpoonup} 0$ in the space of measures as $h\rightarrow 0$.
Now, for some $\varepsilon>0$, we decompose the integral
\begin{align*}
\int_K |u_h| \dx x
\ &= \
\int_{K\setminus \M_\delta} |u_h| \dx x
+ \int_{\M_\delta\setminus B_\varepsilon(\Gamma)} |u_h| \dx x
+ \int_{\M_\delta\cap B_\varepsilon(\Gamma)} |u_h| \dx x
\ =: \
I_1 + I_2 + I_3
\, ,
\end{align*}
and estimate each term individually.
The first term is bounded via
\begin{align*}
I_1
\ &\leq \
\int_{K\setminus \M_\delta} |\Pi_{P^0}(u_h)| \dx x \ + \ \int_{K\setminus \M_\delta} |u_h - \Pi_{P^0}(u_h)| \dx x \\
\ &\leq \
\E_{h,\delta}(u_h) + C h \int_K |\curl(u_h)| \dx x 
\ \leq \
(1+C h)\E_{h,\delta}(u_h) + C h \mathcal{H}^1(\Gamma)
\, .
\end{align*}
Following the proof in the continuous setting (Proposition~\ref{prop:well_posed_pb}), and slightly deforming $u_h$ is necessary, we choose $\varepsilon>0$ such that the mass of $\curl(u_h\restr\M_\delta)$ is bounded by $\frac{C}{\varepsilon}\E_{h,\delta}(u_h)$.
For $I_2$ we note that $|\nu_3(\Pi_\M(x))|\gtrsim \varepsilon$ if $\dist(\Pi_\M(x),\Gamma)>\varepsilon$ and so
\begin{align*}
I_2
\ &\leq \
\int_{\M_\delta\setminus B_\varepsilon(\Gamma)} |\Pi_{P^0}(u_h)| \dx x 
\ + \ \int_{\M_\delta\setminus B_\varepsilon(\Gamma)} |u_h - \Pi_{P^0}(u_h)| \dx x \\
\ &\leq \
\frac{C}{\varepsilon}\E_{h,\delta}(u_h) + C h \int_{\M_\delta} |\curl(u_h)| \dx x 
\ \leq \
\Big(\frac{C}{\varepsilon} + C h\Big)\E_{h,\delta}(u_h) + C h \mathcal{H}^1(\Gamma)
\, .
\end{align*}
Finally, we use Lemma~\ref{lem:num_isoperim_ineq} and the bound for $I_2$ to get
\begin{align*}
I_3
\ &\leq \
\cst \left(\int_{\M_\delta} |\curl(u_h)| \dx x
\ + \
\int_{\M_\delta\setminus B_\varepsilon(\Gamma)} |u_h| \dx x \right)
\ \leq \
\cst\Big(1 + \frac{C}{\varepsilon} + C h\Big)\E_{h,\delta}(u_h) + C h \mathcal{H}^1(\Gamma)
\, .
\end{align*}
This implies that the mass of $u_h$ on the whole set $K$ is bounded by
\begin{align*}
\int_K |u_h| \dx x
\ &\lesssim \
\Big(1 + \frac{1}{\varepsilon} + h\Big) \E_{h,\delta}(u_h) + h \mathcal{H}^1(\Gamma)
\, .
\end{align*} 
We can therefore assume that (up to a subsequence) $u_h \overset{*}{\rightharpoonup} u$ as measures for some $u\in\rca(K)^3$ with support outside of $E$.

\bigskip

\noindent\textit{\underline{Lower bound.}}
For the liminf inequality, take $\phi\in C_c^\infty(K)^3$ with $|\phi|_\infty\leq 1$ to get
\begin{align*} 
\int_K \phi\cdot\curl(u_h + u_{0,\delta,h})\dx x
\ &= \  
\sum_{T\in\mathcal{T}_h}  \int_T \phi\cdot\curl(u_h + u_{0,\delta,h})\dx x \, .
\end{align*}
Integration by parts on each mesh cell $T\in\mathcal{T}_h$ yields
\begin{align*}
\sum_{T\in\mathcal{T}_h}  \int_T \phi\cdot\curl(u_h + u_{0,\delta,h})\dx x
\ &= \ 
\sum_{T\in\mathcal{T}_h} \int_T \curl(\phi)\cdot (u_h + u_{0,\delta,h})\dx x \\
\qquad &+ \ \sum_{T\in\mathcal{T}_h} \int_{\partial T} \phi\cdot ((u_h + u_{0,h})\times\nu_{\partial T})\dx x \, .
\end{align*} 
We note that on boundary facets it holds $u_h = u_{0,\delta,h} = 0$.
Furthermore, in the sum over all cells boundaries of the above calculation each interior facets appears exactly twice but with opposite signs for the normal vector. 
Since the tangential components of Nédélec-functions are continuous across facets, it follows that the second sum on the right hand side vanishes.
Thus,
\begin{align}\label{prop:num_G_conv:IPP_Ned}
\int_K \phi\cdot\curl(u_h + u_{0,\delta,h})\dx x
\ &= \
\int_K \curl(\phi)\cdot(u_h + u_{0,\delta,h})\dx x \, .
\end{align}
We can use the weak*-convergence of $u_h$ and $u_{0,\delta,h}$ and take the supremum in $\phi$ to get
\begin{align*}
\liminf_{h,\delta\rightarrow 0} \int_K |\curl(u_h) + \curl(u_{0,\delta,h})| \dx x
\ &\geq \
\intc(u+u_{0},K)\, .
\end{align*}
The decomposition of the mass 
\begin{align*}
\int_K \rho_\delta|\Pi_{\feP 0}(u)| \dx x 
\ &= \
\int_{\M_\delta} \rho_\delta|\Pi_{\feP 0}(u)| \dx x 
\ + \ 
\int_{\Omega\setminus\M_{\delta}} |\Pi_{\feP 0}(u)| \dx x 
\, .
\end{align*}
implies, together with the weak* convergence of $u_h$ 
as $h\to 0$ that
\begin{align*}
\liminf_{h,\delta\to 0} \int_K \rho_\delta|\Pi_{\feP 0}(u)| \dx x 
\ &\geq \
\int_{\M} |\nu_3| \dx \Vert u\Vert
\ + \
\int_\Omega \dx\Vert u\Vert
\, ,
\end{align*}
which concludes the proof of the lower bound.

\bigskip

\noindent\textit{\underline{Recovery sequence.}}
For the upper bound, let $u$ be a vector valued measure of finite mass in $K$.
We assume that $u\equiv 0$ in $E$, otherwise the result is trivial.
We regularize $u$ in a manner similar to \cite{Ern2015}.
The idea is to push $u$ away from the surface $\M$ and then apply a convolution.

In order to illustrate the approach, let us assume that $u$ is a function, $\M$ a sphere and $K$ the exterior domain.
Then we could simply define $\widetilde{u}(x) := u(x + \kappa\ee_r)$ and then pose $u_\kappa(x) := \int_{B_1(0)}\zeta(y)\widetilde{u}(x+\kappa y)\dx y$, where $\zeta\in C_c^\infty(B_1(0),[0,1])$ is the standard radially symmetric convolution kernel with $\int_{B_1(0)}\zeta(x)\dx x=1$.
In the above, $\kappa\in (0,1]$ is the regularization parameter.
This definition encodes a convolution at scale $\kappa$ and is well-defined since for $x\in\Omega$ and $x\in B_1(0)$ it holds $(x+\kappa y) + \kappa\ee_r \in \Omega$.

In the general case, one has to be more careful and take into account the shape of $\M$ and the fact that $u$ is only a measure.
We introduce as in \cite{Ern2015} the function $\varphi_\kappa(x) = x-\kappa n(x)$, where $n$ is a smooth unit vector field, transversal to $\partial\Omega$ that is used to ``push'' the position $x$ into the domain $\Omega$.
A radius $r>0$ is chosen based on the uniform cone property such that $\varphi_\kappa(\Omega)+ B_{\kappa r}(0)\subset\Omega$ for all $\kappa\in [0,1]$.
We then define 
\begin{align*}
u_\kappa(x)
\ := \
(\kappa r)^{-n}\int_{B_{\kappa r}(\varphi_\kappa(x))}\zeta\Big(\frac{z-\varphi_\kappa(x)}{\kappa r}\Big)\dx u(z)
\, .
\end{align*}
Note that this definition agrees with the previous idea if $u$ has a density with respect to the Lebesgue measure.
As $\kappa\to 0$, this function $u_\kappa$ weakly* converges to $u$ as a measure.
%

We define the discretized function $u_{\kappa,h}$ via the $L^2-$projection $P_h$ (and not the interpolation operator $\Pi_{\feNed}$) 
\begin{align*}
u_{h,\kappa}
\ &:= \
P_h(u_\kappa)
\ := \
\argmin{v_h\in\feNed} \Vert v_h - u_\kappa\Vert_{L^2}^2
\, .
\end{align*}
We will use later that $\int_\Omega (P_h(u_\kappa) - u_\kappa)v_h \dx x = 0$ for any $v_h\in\feNed$.

In order to show weak* convergence of $u_{\kappa,h}$ to $u$, let $\phi\in C^\infty(\Omega)$.
We compute
\begin{align}
\int_\Omega & u_{h,\kappa}(x) \cdot\phi(x) \dx x
\ - \
\int_\Omega \phi(x)\cdot \dx u(x) \nonumber \\
\ &= \
\int_\Omega (P_h(u_\kappa)(x) - u_\kappa(x))\cdot \phi(x) \dx x
+
\int_\Omega u_\kappa(x)\cdot\phi(x) \dx x - \int_\Omega \phi(x) \cdot\dx u(x)\label{prop:num_G_conv:upper_cpt}
\, .
\end{align} 
For the first term on the RHS of \eqref{prop:num_G_conv:upper_cpt} we note that, using the interpolation of $\phi$ in $\feNed$, that
\begin{align*}
\int_\Omega (P_h(u_\kappa) - u_\kappa)\cdot \phi \dx x
\ &= \
\int_\Omega (P_h(u_\kappa) - u_\kappa)\cdot \Pi_{\feNed}(\phi) \dx x
+
\int_\Omega (P_h(u_\kappa) - u_\kappa)\cdot (\phi - \Pi_{\feNed}(\phi)) \dx x
\, .
\end{align*}
The first term is zero by orthogonality of the projection $P_h$ since the interpolation $\Pi_{\feNed}(\phi)\in\feNed$ and for the second term we estimate using Hölder's inequality
\begin{align}\label{prop:num_G_conv:estim_Phuk_uk_phi__Piphi}
\Big|\int_\Omega (P_h(u_\kappa) - u_\kappa)\cdot (\phi - \Pi_{\feNed}(\phi)) \dx x\Big|
\ &\leq \
\Vert P_h(u_\kappa) - u_\kappa \Vert_{L^1} \Vert \phi - \Pi_{\feNed}(\phi) \Vert_{L^\infty}
\, ,
\end{align}
which vanishes in the limit $h\to 0$ since the mass of $u$ is finite and $\Pi_{\feNed}(\phi)\to\phi$ uniformly as $h\to 0$.

The remaining two terms on the RHS of \eqref{prop:num_G_conv:upper_cpt} converge to zero by weak* convergence $u_\kappa\to u$.

The convergence of the mass can be shown as follows:
%
Let $\phi\in C^\infty(\Omega)$, $\Vert\phi\Vert_\infty\leq 1$.
Using the continuity of $\phi$ and estimating as in \eqref{prop:num_G_conv:upper_cpt} and \eqref{prop:num_G_conv:estim_Phuk_uk_phi__Piphi}, we get that
\begin{align*}
\int_{\Omega_\delta} \Pi_{\feP0}(u_{h,\kappa})(x) \phi(x)\dx x
\ &= \
\sum_{T\in\T_h\cap{\Omega_\delta}} \int_T \Big(\dashint_T u_{h,\kappa} \dx y\Big) \phi(x) \dx x
\ = \
\sum_{T\in\T_h\cap{\Omega_\delta}} \int_T \Big(\dashint_T \phi(x) \dx x\Big) u_{h,\kappa}(y) \dx y \\
\ &= \
\int_{\Omega_\delta} \phi(y)u_{h,\kappa}(y) \dx y + o(1)
\ = \
\int_{\Omega_\delta} \phi(y)u_{\kappa}(y) \dx y + o(1)
\, .
\end{align*}
Using the notation $\zeta_\kappa = (r\kappa)^{-n}\zeta(\frac{\cdot}{\kappa r})$, we find that
\begin{align*}
\int_{\Omega_\delta} \phi(y)u_{\kappa}(y) \dx y
\ &= \
\int_{\Omega_\delta} \phi(y)\int_{\mathbb{R}^3} \zeta_\kappa(z-\varphi_\kappa(y)) \dx u(z) \dx y
\ = \
\int_{\mathbb{R}^3} \int_{\Omega_\delta} \phi(y)\zeta_\kappa(z - \varphi_\kappa(y)) \dx y \dx u(z) \\
\ &= \
\int_{\mathbb{R}^3} \phi\tilde{*}\zeta_\kappa \dx u(z)
\, ,
\end{align*}
where $\phi\tilde{*}\zeta_\kappa$ denotes the convolution-type operation $\int_{\Omega_\delta} \phi(y)\zeta_\kappa(\cdot - \varphi_\kappa(y)) \dx y$.
Since $\Vert\phi\Vert_\infty\leq 1$, $\int\zeta_\kappa=1$ and $D\varphi_\kappa = \id + O(\kappa)$, we also have that $\Vert\phi\tilde{*}\zeta_\kappa\Vert_\infty\leq 1 + C\kappa$ and hence $\frac{1}{1 + C\kappa}\phi\tilde{*}\zeta_\kappa$ is an admissible test function for the mass of $u$, thus
\begin{align*}
\int_{\Omega_\delta} \phi(y)u_{\kappa}(y) \dx y
\ \leq \
(1+C\kappa)\int_\Omega\dx\Vert u\Vert
\, .
\end{align*}

For the integral $\int_{\M_\delta} \rho_\delta |\Pi_{\feP 0}(u_{h,\kappa})| \dx x$, we note that by continuity of $\nu$ and using the fact that $u_{h,\kappa}$ was constructed using a convolution it holds
\begin{align*}
\limsup_{h\to 0} \int_{\M_\delta} \rho_\delta |\Pi_{\feP 0}(u_{h,\kappa})| \dx x
\ &\leq \
\int_\M |\nu_3| \dx\Vert u\Vert\, .
\end{align*}
It remains the estimate of $\int_K |\curl(u_{h,\kappa}+u_{0,\delta,h})|\dx x$.
The idea is the same as in the lower bound.
We rewrite the absolute values as a supremum using functions $\phi\in C_c^\infty(K)^3$ with $|\phi|_\infty\leq 1$ and use the formula in \eqref{prop:num_G_conv:IPP_Ned}.
We then get 
\begin{align*}
\limsup_{h,\kappa\to 0} \int_K |\curl(u_{h,\kappa}+u_{0,\delta,h})|\dx x
\ &= \ 
\sup_{\substack{\psi\in\C_c^\infty(K) \\ \Vert\phi\Vert_\infty\leq 1}}\int_K \curl(\phi)\cdot (u+u_{0})\dx x 
\ = \ \intc(u+u_0,K) \, ,
\end{align*}
where we used 
the weak* convergences $u_\kappa \stackrel{*}{\rightharpoonup} u$ and $u_{0,\delta} \stackrel{*}{\rightharpoonup} u_0$ as $h\rightarrow 0$.

\end{proof}

\subsection{ADMM-algorithm}
\label{subsec:admm}


To minimize the energy in \eqref{def:func_Ehd} numerically, we employ the Alternating Direction Method of Multipliers (ADMM) \cite{Gabay1976}.
This algorithm requires our problem to be of the form \cite[Sec. 3]{Boyd2010}
\begin{align}\label{pb:admm}
\min_{x\in\mathbb{R}^{n_x}\, , \, y\in\mathbb{R}^{n_y}} F(x) + G(y) \quad\text{subject to } Ax + By = c\in\mathbb{R}^{n_c}\, .
\end{align}
Then, one performs alternating minimizations in $x$ and $y$ as well as updates of the dual variable $z$ of the augmented Lagrangian
\begin{align*}
L(x,y,z) 
\ &= \ F(x) + G(y) + z^\top(Ax+By-c) + \frac{\varpi}{2}\Vert Ax+By-c\Vert^2\, .
\end{align*}

We first reformulate \eqref{def:func_Ehd} to fit into this scheme.
With a small abuse of notation, we identify the finite element functions and their vectors of degrees of freedom.
We set $x:= u$, $y\defi(p,q)$, $z\defi(\lambda,\mu)$ and define the density functions
\begin{align*}
p_\mathrm{max}(x)
\defi
\begin{cases}
1 & \text{if } x\in \Omega_h\, , \\
|\nu_3(x)| & \text{if } x\in \M_h\, , \\ 
+\infty & \text{otherwise,}
\end{cases}
\quad\text{ and }\quad
q_\mathrm{max}(x) 
\defi
\begin{cases}
\beta & \text{if } x\in \Omega\, , \\
+\infty & \text{otherwise.}
\end{cases}
\end{align*}
Defining the objective functions $F$ and $G$ 
\begin{align*}
F(u) \defi 0  \qquad G(p,q)\defi \Vert p_\mathrm{max} p \Vert_{L^1} \ + \ \Vert q_\mathrm{max} q\Vert_{L^1}\, ,
\end{align*}
subject to the constraints 
\begin{align*}
\begin{pmatrix} \id \\ \curl \end{pmatrix}
u
\ + \ 
\begin{pmatrix} -\id & 0 \\ 0 & -\id \end{pmatrix} 
\begin{pmatrix} p \\ q \end{pmatrix}
\ = \ 
\begin{pmatrix} 0 \\ \curl(u_0) \end{pmatrix}\, ,
\end{align*}
we transformed the minimization of \eqref{def:func_Ehd} into a problem suitable for ADMM.
We furthermore write $\varpi=(\gamma_M,\gamma_C)$ for the stepsizes used inside the algorithm.
Our ADMM scheme is summarized in Algorithm~\ref{alg:admm_E0}.

The minimization procedures in Algorithm~\ref{alg:admm_E0} are carried out using the optimality conditions and solving the associated linear systems.
More precisely, for the minimization in line 2 of Algorithm~\ref{alg:admm_E0} we solve the weak formulation
\begin{align} \label{eq:alg_sol_lin_eq}
0 
\ &= \ 
\gamma_C \langle \curl(u), \curl(v)\rangle + \gamma_M\langle u, v\rangle - \langle (\lambda+\gamma_C p), \curl(v)\rangle - \langle (\mu+\gamma_M q), v\rangle \quad \forall v\in\feNed \, .
\end{align}
Line 3 and 4 are computed by 
\begin{align*}
\begin{cases}
p \ = \ \frac{1}{\gamma_M}\left( \Big(\max\{ |\overline{p}|/p_\mathrm{max} , 1 \}\Big)^{-1} - 1 \right)\overline{p}\, , & \qquad\overline{p} = \lambda - \gamma_M \Pi_{\fePP{0}}(u)\, , \\
q \ =\  \frac{1}{\gamma_C}\left( \Big(\max\{ |\overline{q}|/q_\mathrm{max} , 1 \}\Big)^{-1}  - 1\right)\overline{q}\, , & \qquad\overline{q} = \mu - \gamma_C \curl(u) \, .
\end{cases}
\end{align*}

Numerous modifications for speeding up ADMM are known \cite{Kim2021}.
We tested the "Accelerated Alternating Direction Method of Multipliers" from \cite{Kim2021}, the Nesterov acceleration as described in \cite{Goldstein2014} as well as a simple over-relaxation scheme, finding that the latter leads to a faster convergence, while the others two had no noticeable influence on the speed.

\begin{algorithm}
\caption{ADMM algorithm for minimizing $\E_0$}
\label{alg:admm_E0}
\begin{algorithmic}[1] 
\Param $\gamma_M,\gamma_C>0$ (step sizes), $h>0$ (mesh size), $\beta\geq 0$ (weight for $\MM(S)$), $\phi\in [0,\pi]$ (angle of rotation), $w_E\gg 1$ (penalization weight inside $E$), $\varepsilon>0$ (to avoid zero divisions)
\Init $p,q,\lambda,\mu \gets 0\in \fePP 0$
\Init $u \gets 0\in\feNed$
\Init $q_\mathrm{max} \gets \Pi_{\fePP 0}(\beta\chi_\Omega + w_E \chi_E) $
\Init $p_\mathrm{max} \gets \Pi_{\fePP 0}(1\, \chi_{\Omega\setminus\M_h} + \max\{|\nu\cdot\HH|, \varepsilon\}\chi_{\M_h} + w_E \chi_E) $
\For{$k = 1,2,...$}
	\State $u \gets \argmin{u}\left\{ -\langle\lambda, u\rangle - \langle\mu, \curl(u)\rangle + \frac{\gamma_C}{2}\Vert q-\curl(u)\Vert^2_{L^2} + \frac{\gamma_M}{2}\Vert p-u\Vert^2_{L^2}\right\} $
	\State $q \gets \argmin{q}\left\{ \Vert q_\mathrm{max} q \Vert_{L^1} + \langle\lambda,q\rangle + \frac{\gamma_C}{2}\Vert q-\curl(u)\Vert^2_{L^2}\right\} $
	\State $p \gets \argmin{p}\left\{ \Vert p_\mathrm{max} p \Vert_{L^1} + \langle\mu,p\rangle + \frac{\gamma_M}{2}\Vert p-u\Vert^2_{L^2}\right\} $
	\State $\lambda \gets \lambda + \gamma_C (q - \curl(u) - \curl(u_0))$
	\State $\mu \gets \mu + \gamma_M ( p - \Pi_{\fePP 0}(u))$
\EndFor
\Out $E_h = \Vert p_\mathrm{max} p \Vert_{L^1(\Omega)} + \beta\Vert q\Vert_{L^1(\Omega)}$
\end{algorithmic}
\end{algorithm}

\subsection{Implementation}

The finite element discretization and implementation of Algorithm~\ref{alg:admm_E0} has been carried out using \emph{FEniCS} \cite{Fenics2012,Fenics2015}, see also \cite{FenicsTutorial2016} for an introduction. 
For generating the meshes we use the program GMSH \cite{gmsh2020}. 
The visualization of the results is realised using ParaView \cite{Paraview2015}.

\paragraph{Solving equations and projections.} 
We perform the projections in Algorithm~\ref{alg:admm_E0} by solving the linear system associated to the $L^2-$orthogonality relation verified by the projection.
All matrices for the linear systems, including \eqref{eq:alg_sol_lin_eq} are independent of the iteration step, and hence we assemble them only once before the iteration starts.
For solving the system, we apply a solver using a $LU-$decomposition also calculated once in the beginning and assembling only the right hand side in each iteration.

\paragraph{Mesh and representation of $\Gamma$.}
From \cite{ACS2024} we know that minimal configurations of $T$ and $S$ are concentrated inside the convex envelope of the particle. 
This is why we decided to create a mesh adapted to this situation, i.e.\ a fine mesh close to the surface $\M$ and inside the convex envelope, while far from $\M$ the mesh can be coarser, see Figure~\ref{fig:mesh_M_P}.
In our case we choose the mesh size $h$ to be between $0.03$ on $\M$ and $0.3$ at the boundary of the box.
The simulations in which we investigate the influence of the angle between particle orientation and external field $\mathbf{H}$ (previously chosen to be $\ee_3$) are conducted by changing the vector $\mathbf{H}$, thus allowing to use the same mesh for all simulations of a particle and only adapting the energy by using the surface energy density $|\nu\cdot\mathbb{H}|$ and $\Gamma=\{\nu\cdot\mathbb{H}=0\}\subset\M$.
In order to obtain an accurate approximation, we also include the boundary layer $\M_h$ into the mesh generation.
We choose the thickness of $\M_h$ to be equal to one cell, which gives satisfactory results from our experiences in accordence with Proposition~\ref{prop:num_G_conv}, see Figure~\ref{fig:mesh_M_P-b}.
We can furthermore perform a cut-out of our mesh, i.e.\ remove the cells inside $E$ from our simulation, which significantly increases the speed of our code. 
In the case of the peanut mesh, this reduces the number of cells by more than $40 \%$, see Figure~\ref{fig:mesh_M_P-a}. 
We maintain a small layer inside $E$ which is used to prevent the surface and its boundary to enter the particle and to create the initial condition: 
Using the level-set function of the shape, we calculate the normal field $\nu$ which allows us to define the function $\nu\cdot\mathbf{H}$.
Then, we set $u_0\defi \chi_E\nabla(\nu\cdot\mathbf{H})$.
The vector field $\curl(u_0)$ serves as approximation for $\tau_\Gamma$ as seen in Section~\ref{sec:theory}.

\begin{figure}[H]
	\begin{subfigure}[c]{0.66\textwidth}
	\centering
	\includegraphics[scale=0.19]{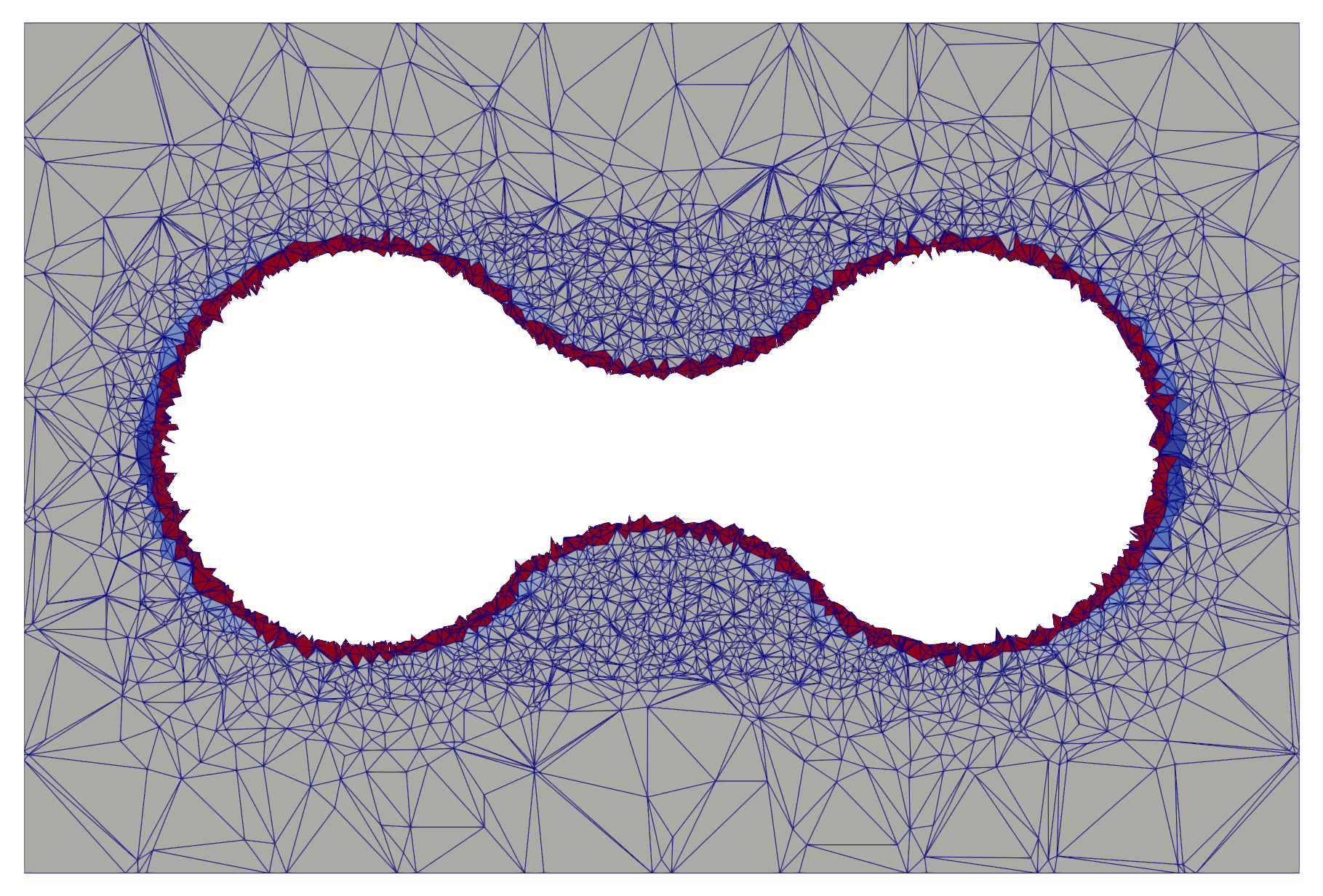}
    \caption{}
    \label{fig:mesh_M_P-a}
    \end{subfigure}
	\begin{subfigure}[c]{0.33\textwidth}
	\centering
	\includegraphics[scale=0.32]{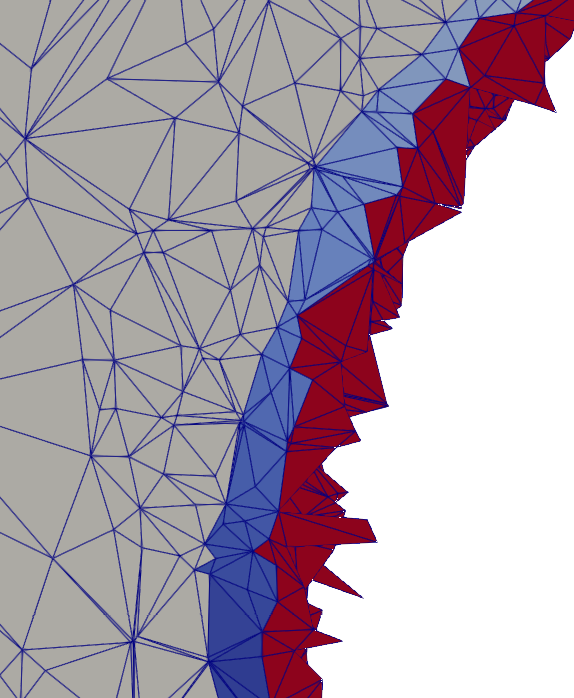}
    \caption{}
    \label{fig:mesh_M_P-b}
    \end{subfigure}
\caption{Left: Mesh after removing the interior cells. The colors represent the values of $p_\mathrm{max}$: $10^5$ inside the particle (red), $|\nu_3|$ in the boundary layer (shades of blue), $1$ otherwise (grey). Right: Magnification of a part of the mesh.}
\label{fig:mesh_M_P}
\end{figure}

\paragraph{Breaking symmetry: the parameter $d_\Gamma$.}
In the first simulations, we observed the behaviour that the algorithm converges, but the solution surface has holes or covers the whole of $\M$, see Figure~\ref{fig:d_Gamma_P_non-uniq}.
This phenomenon seems to be related to the non-uniqueness of the solution.
By mirror symmetry of the particle, there are two distinct configurations which have the same energy and seem to superpose each other.
Note that the curl we obtain in this situation is of order $10^{-2}$, i.e.\ close to zero.
To overcome this issue, we introduce the parameter $d_\Gamma$ which acts as a shift for the initial condition in direction $\mathbf{H}$ and breaks the symmetry.
Choosing $d_\Gamma\defi h$ is enough to eliminate this phenomenon, compare Figure~\ref{fig:P_min_config-c} to Figure~\ref{fig:d_Gamma_P_non-uniq}.

\begin{figure}[H]
\begin{center}
\includegraphics[scale=0.3]{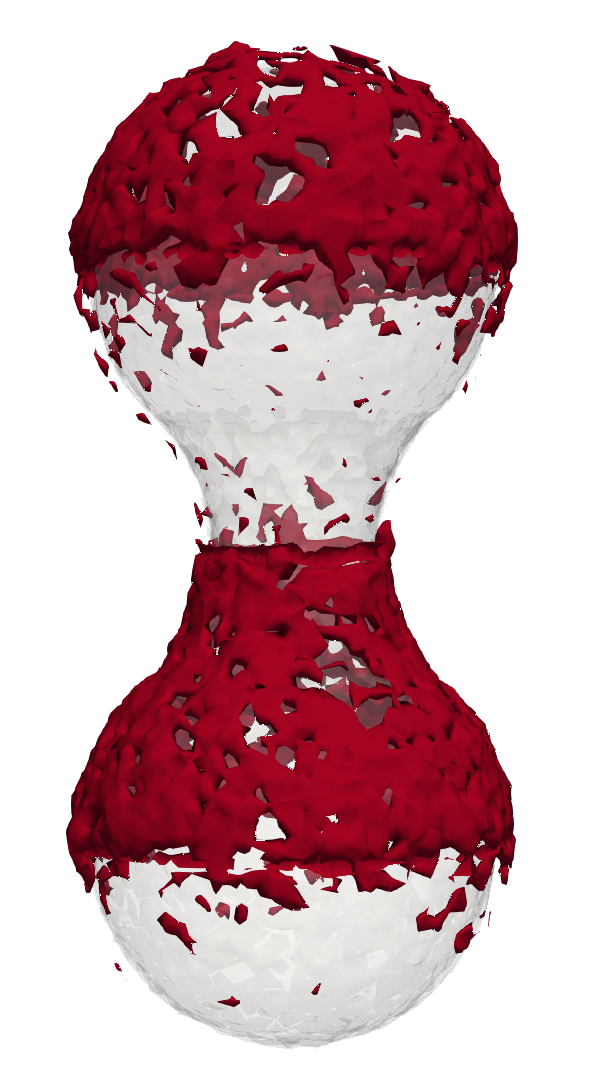}
\hspace*{0.5cm}
\includegraphics[scale=0.3]{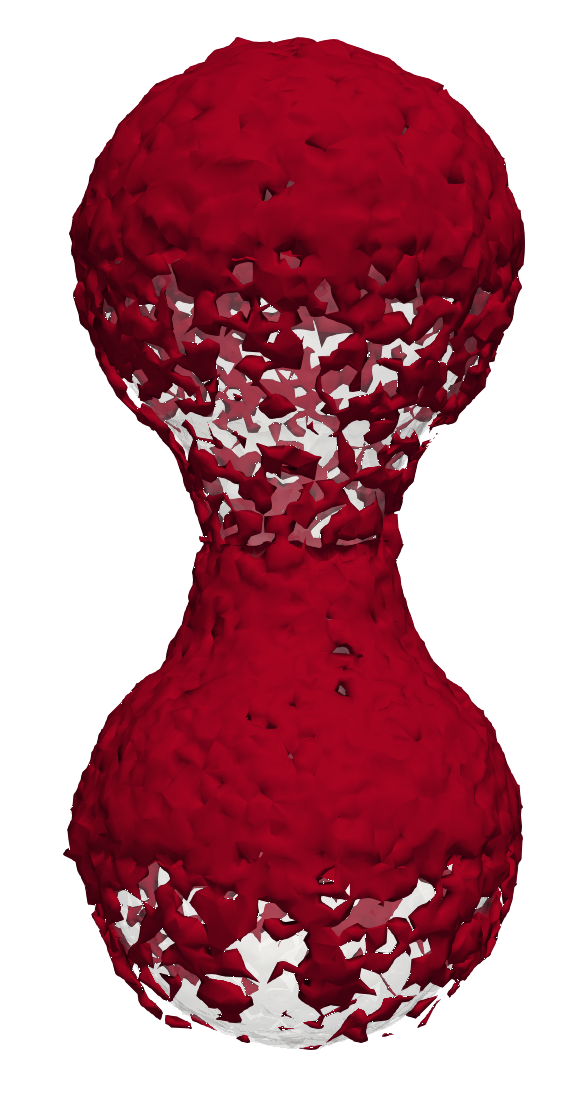}
\end{center}
\caption{Different isosurfaces of a terminal configuration obtained for $\beta=1$ after $4000$ iterations with $d_\Gamma=0$.}
\label{fig:d_Gamma_P_non-uniq}
\end{figure}

\section{Results}
\label{sec:results}

In this section we detail and comment on the numerical simulations we conducted using four different particle geometries: Sphere, Peanut, Donut and Croissant.
In the case of the sphere, the minimizers of $\E_0$ are known to be the Saturn ring around the equator and the dipole, as described in\cite[Ch. 6]{ACS2021}.
This case can thus serve as validation of our algorithm and the numerical implementation.
The peanut-shape has been chosen because the rotational symmetry is broken along an axis.
It is therefore possible to study the defect structures and the energy as function of a single angle $\phi$ between this axis and the external field $\mathbf{H}$.
Since the peanut is also non-convex, one could hope to see a non-trivial $T\restr\Omega$.
However, this is not observed in our simulations.
The donut shape is an example of a shape with a different topology than sphere and peanut and the non-convexity allows to observe a non-trivial $T\restr\Omega$.
Finally, the croissant is another non-convex particle for which certain parameter choices do result in a non-vanishing $T\restr\Omega$ and $S\restr\Omega$ with non-zero curvature.

\subsection{Spherical particle}

The simulation of the spherical particle serves mainly as validation case for our algorithm.
Running simulations for values of $\beta$ between $0.01$ and $1.1$, we observe only a Saturn ring at the equator or a dipole, as expected from the theoretical analysis, see Figure~\ref{fig:S_SR_DP}.

\begin{figure}
	\begin{subfigure}[c]{0.49\textwidth}
	\centering
	\includegraphics[scale=0.3]{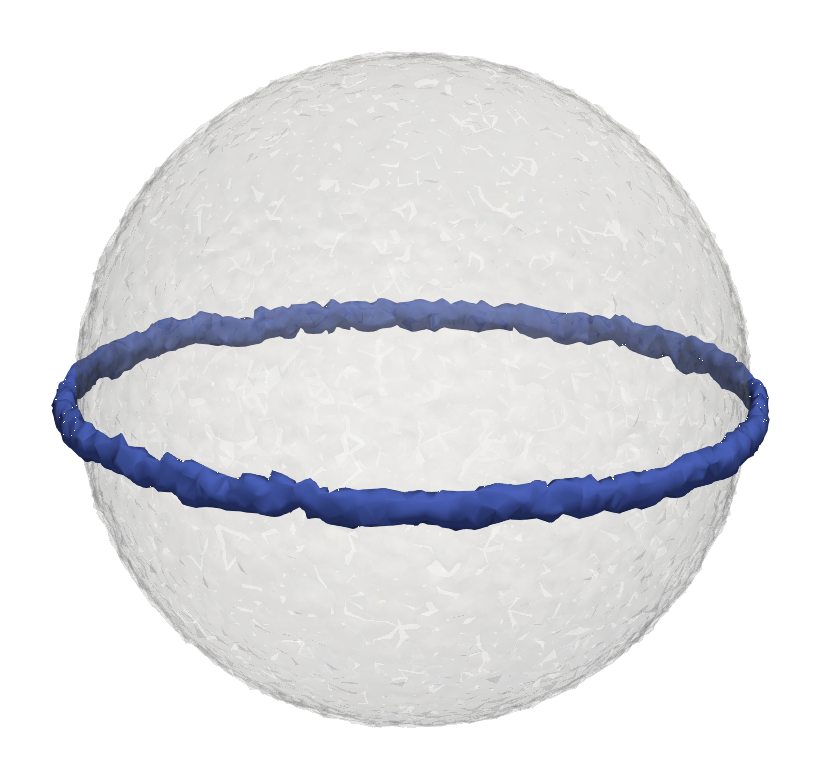}
    \caption{}
    \label{fig:S_SR_DP-a}
    \end{subfigure}
	\begin{subfigure}[c]{0.49\textwidth}
	\centering
	\includegraphics[scale=0.3]{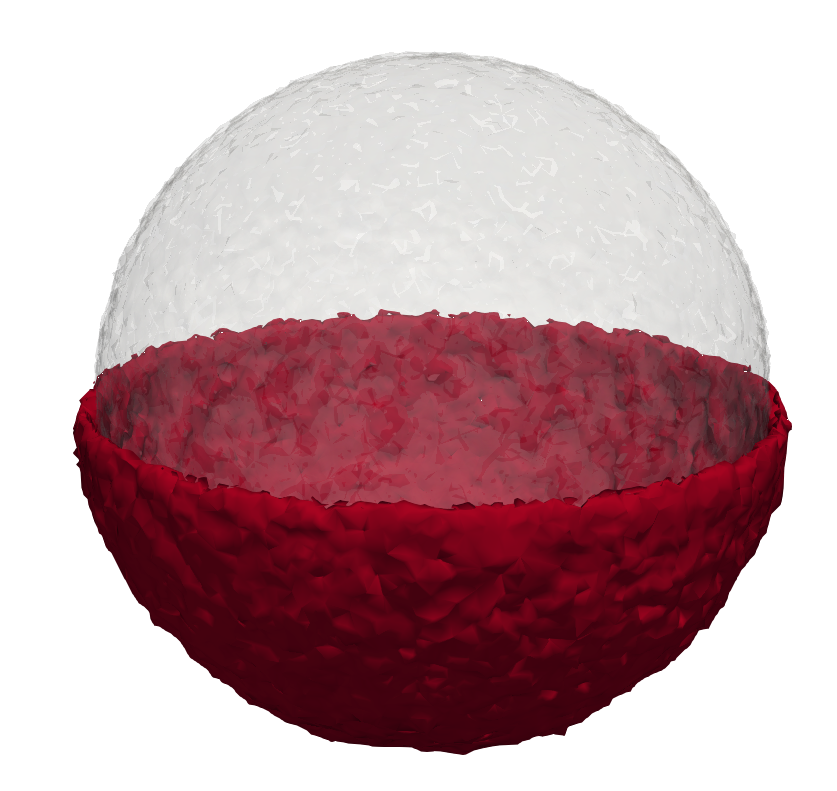}
    \caption{}
    \label{fig:S_SR_DP-b}
    \end{subfigure}
\caption{Observed defect configurations: Saturn ring (a) for small values of $\beta$ and dipole (b) for large $\beta$. The line $S$ is indicated in red and $T$ in blue.}
\label{fig:S_SR_DP}
\end{figure}

Furthermore, the numerically calculated energy as function of the parameter $\beta$ is linearly increasing for a Saturn ring configuration and constant (independent of $\beta$) for a dipole, see Figure~\ref{fig:diagr_S2_E_beta}.
These findings are consistent with the behaviour calculated in \cite[Ch. 6]{ACS2021}, compare with Figure 6 therein.
Note that since we are calculating the globally energy minimizing configuration, it is not possible to reproduce the hysteresis phenomenon.

\begin{figure}
\begin{center}
\includegraphics[scale=1.0]{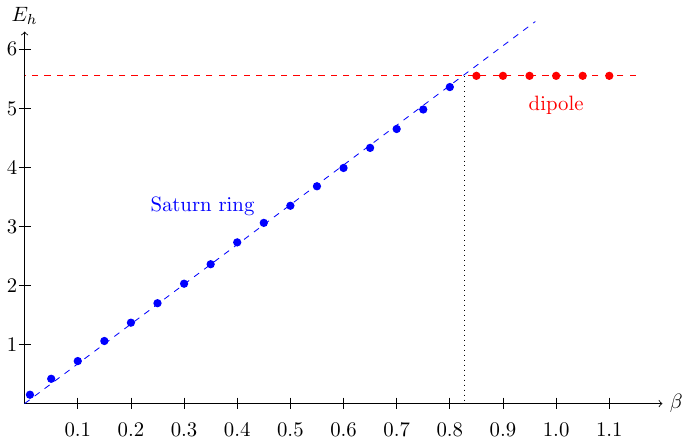}
\end{center}
\caption{Energy of minimizers for different values of $\beta$ around a sphere of radius $1$. 
The mesh consists of around $449\, 000$ cells of size $h=0.03$ around the particle surface.}
\label{fig:diagr_S2_E_beta}
\end{figure}

\subsection{Peanut-shaped particle}

In this subsection we present our findings about the first non-spherical particle that we consider in this article.
It is constructed by combining three circle arcs (two bending outwards which are mirror images of each other, and one inwards) and then rotating the resulting curve to obtain a surface of revolution.
The resulting object is therefore rotationally symmetric around an axis and exhibits additionally a mirror symmetry.
We are interested how the angle $\phi$ between the symmetry axis of the peanut and the external field $\mathbf{H}$ influences the defect structures and energy.
Note that since the peanut is not symmetric with respect to all rotations, we need to calculate the constant $C_\M$ depending on the angle $\phi$ to find the energy minimizing configurations.
The constant $C_\M$ is given by $C_\M = \frac12\int_\M (1-|\nu\cdot\mathbf{H}|) \dx\mathcal{H}^2$ (see \cite[Section 3]{ACS2024}) and can easily be approximated using the finite element discreization of the surface layer $\M_\delta$ by computing $\frac12\int_{\M_\delta} (1-|(\nu\circ\Pi_\M)\cdot\mathbf{H}|) \dx x$.

For $\phi=\frac{\pi}{2}$, the results resemble the spherical case:
There are only two observed configurations, Saturn ring and dipole, see Figure~\ref{fig:P_min_config} (d) and (e).
The plot of the energy as function of $\beta$ in Figure~\ref{fig:diagr_P_E_beta_phi} has the same qualitative behaviour as Figure~\ref{fig:diagr_S2_E_beta}.
If $\phi=0$ the situation is quite different:
since $\Gamma$ consists of three disjoint circle arcs, for $\beta$ small, we see three Saturn rings surrounding the peanut, see Figure~\ref{fig:P_min_config} (a).
Note that the upper and lower Saturn rings are of type $-\frac{1}{2}$, while the inner one has to be a $+\frac12-$defect for orientability reasons.
In our program the difference of a $\pm\frac12-$defect is represented by the fact that the vector field $\curl(u)$ (and $q$) have different orientations.
Before attaining a dipole configuration for large $\beta$ which consists of two disjoint components of $T\restr\M$ (as in Figure~\ref{fig:P_min_config} (c)), there exists a regime for $\beta$ in which two of the components of $\Gamma$ are joined together by $T\restr\M$ and the third component still appears as a Saturn ring, see Figure~\ref{fig:P_min_config} (b).

\begin{figure}
	\begin{subfigure}[c]{0.33\textwidth}
	\centering
	\includegraphics[scale=0.27]{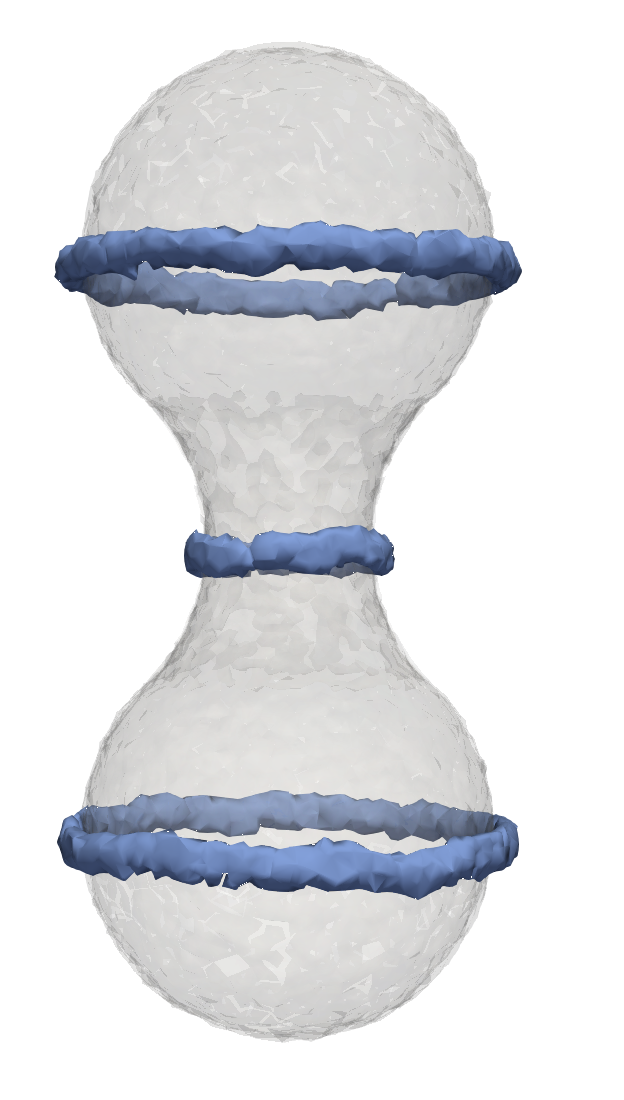}
    \caption{} 
    \label{fig:P_min_config-a}
    \end{subfigure}
    \begin{subfigure}[c]{0.33\textwidth}
	\centering
	\includegraphics[scale=0.27]{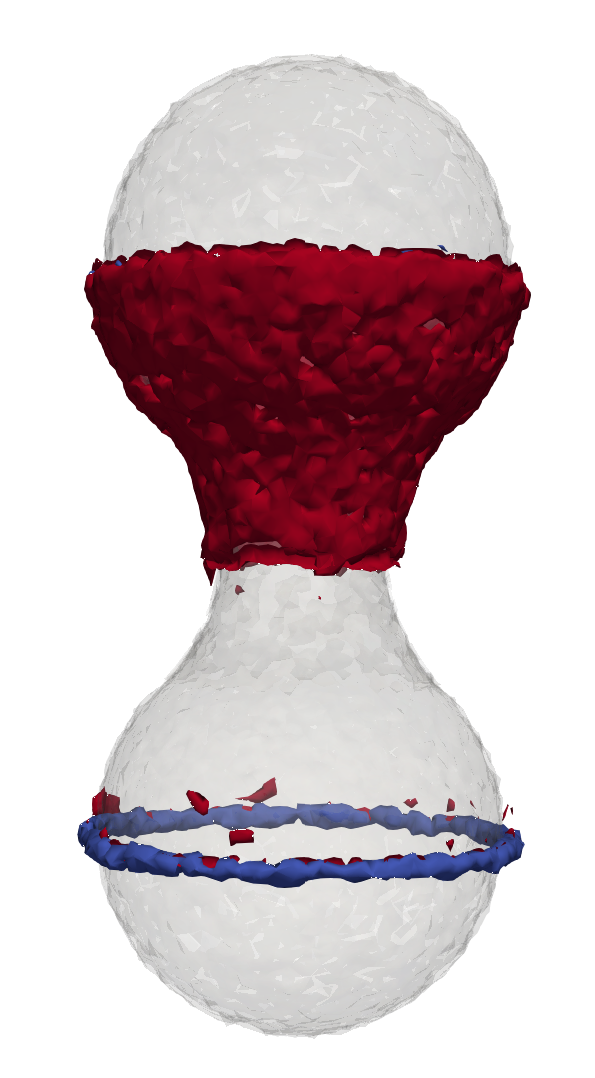}
    \caption{} 
    \label{fig:P_min_config-b}
    \end{subfigure}
    \begin{subfigure}[c]{0.33\textwidth}
	\centering
	\includegraphics[scale=0.27]{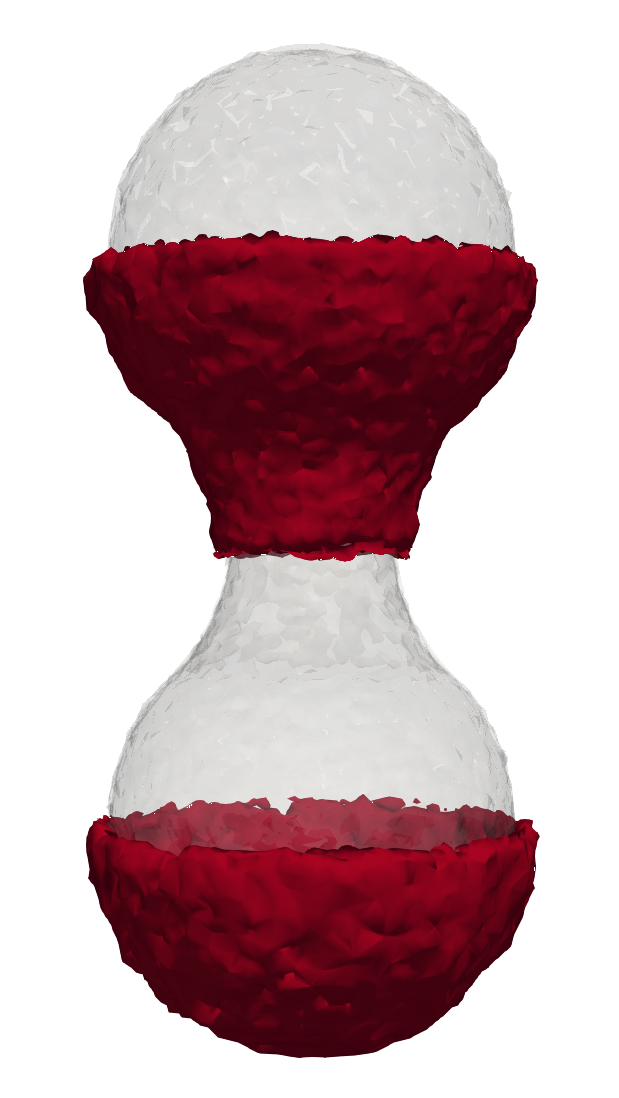}
    \caption{} 
    \label{fig:P_min_config-c}
    \end{subfigure}

	\begin{subfigure}[c]{0.49\textwidth}
	\centering
	\includegraphics[scale=0.18]{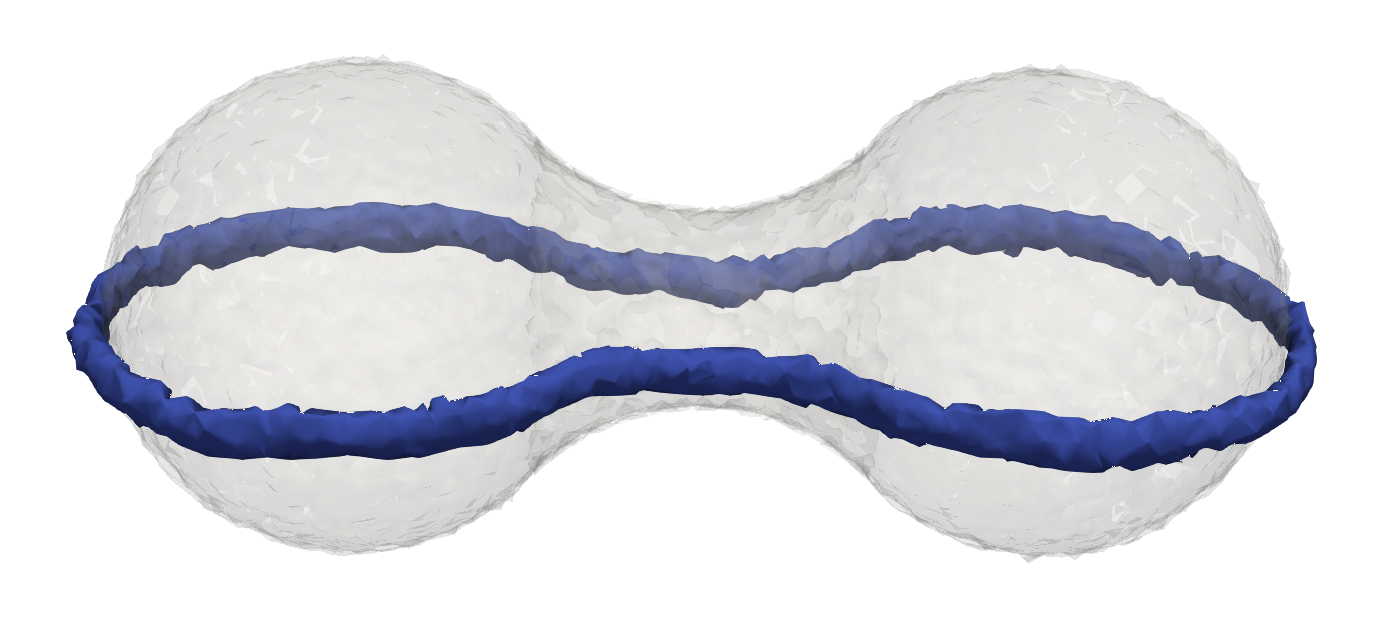}
    \caption{} 
    \label{fig:P_min_config-d}
    \end{subfigure}
    \begin{subfigure}[c]{0.49\textwidth}
	\centering
	\includegraphics[scale=0.18]{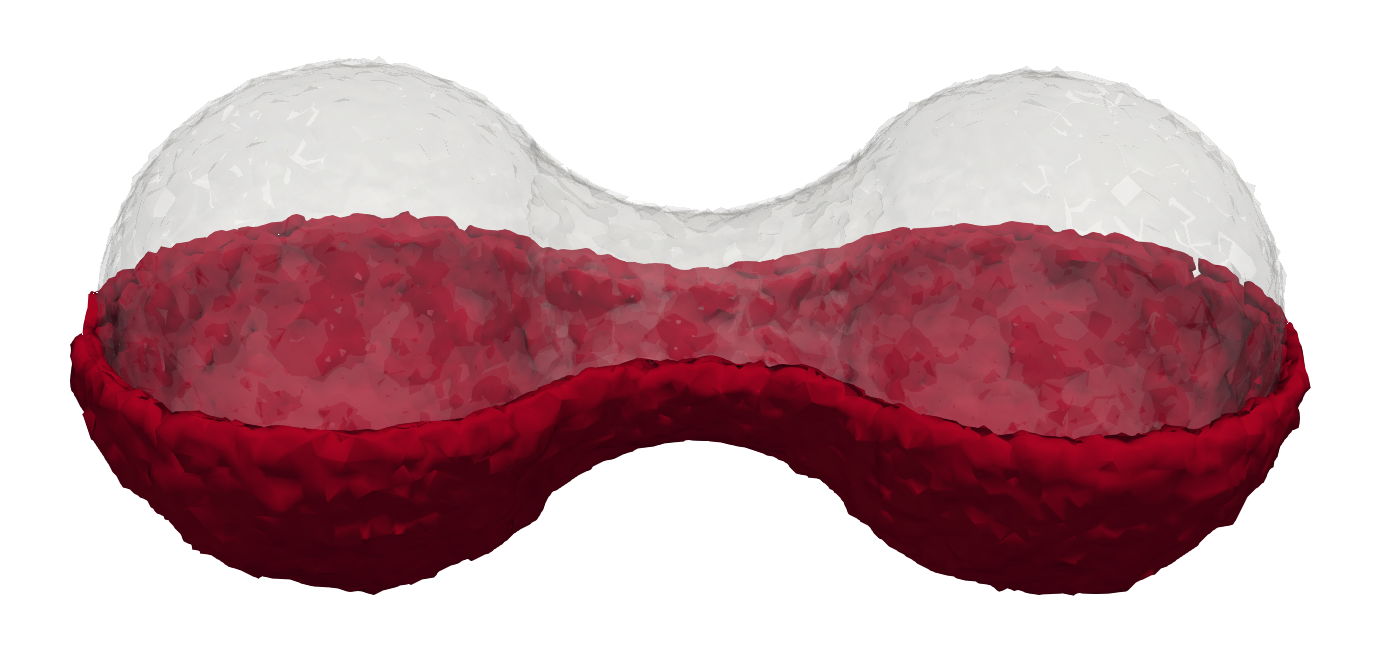}
    \caption{} 
    \label{fig:P_min_config-e}
    \end{subfigure}
\caption{The observed defect configurations for $\phi=0$ (a)-(c) are three Saturn rings (a), two components of $\Gamma$ joined together leaving one Saturn ring (b) and dipole (c). 
For $\phi=\frac{\pi}{2}$ (d)-(e), we find a Saturn ring (d) and dipole (e). }
\label{fig:P_min_config}
\end{figure}

Comparing the found defect structures, we find that depending on the number of connected components of $\Gamma$, we observe $1$ or $3$ Saturn rings for small $\beta$.

\begin{figure}
\begin{center}
\includegraphics[scale=0.9]{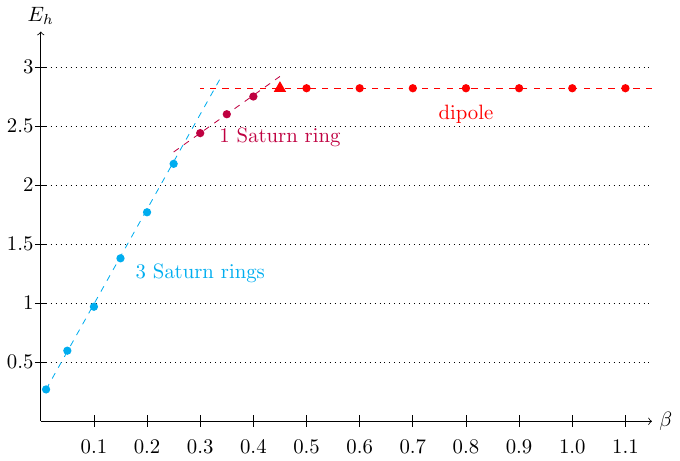}
\end{center}
\caption{Energy of minimizers for different values of $\beta$ around the peanut shape for $\phi=0$. Dashed lines indicate the regression line for each of the three observed configurations. See Figures~\ref{fig:P_min_config} (a)-(c) for images of the three configurations.}
\label{fig:diagr_P_E_beta_phi0}
\end{figure}

As the total length of $\Gamma$ for different angles $\phi$ is not equal (in fact the length decreases when $\phi$ increases), it is natural to expect the slope of the energy as function of $\beta$ to be monotonically decreasing as function of $\phi$, being minimal for $\phi=\frac\pi 2$.
This is indeed the observed behaviour in Figure~\ref{fig:diagr_P_E_beta_phi}.
Also the dipoles have different energy even though they always cover half of the particle.
This can be explained through the weight $|\nu\cdot\mathbf{H}|$ in the integration over the surface, since for $\phi=0$ a larger part of $\M$ is oriented perpendicular to $\mathbf{H}$, while for $\phi=\frac\pi 2$ the normal vector $\nu$ is close to parallel to $\mathbf{H}$ on a larger portion of the particle surface.
We observe again a monotone behaviour of the energy with respect to $\phi$, the minimal energy being given for $\phi=\frac{\pi}{2}$.
We conclude that $\phi=\frac{\pi}{2}$ is the energetically preferred orientation, see Figure~\ref{fig:P_min_config} (d) and (e).
Notice that in the experiments conducted in \cite[Fig. 1]{Sahu2019} the configurations with $\phi=\frac{\pi}{2}$ (or $\phi$ close to $\frac{\pi}{2}$) are by far the most frequently observed orientations.

\begin{figure}
\begin{center}
\includegraphics[scale=0.66]{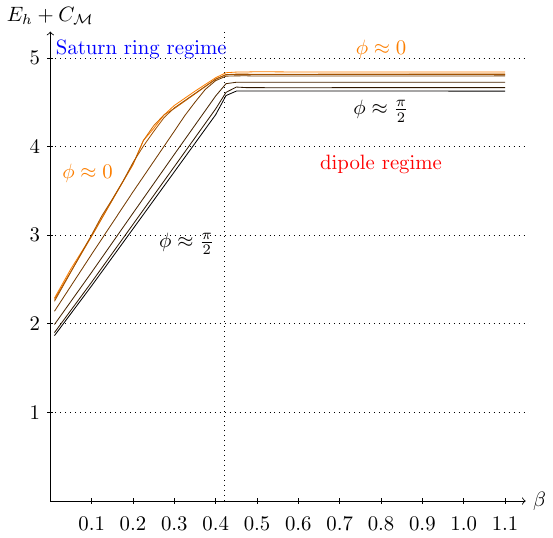}
\hspace*{0.25cm}
\includegraphics[scale=0.88]{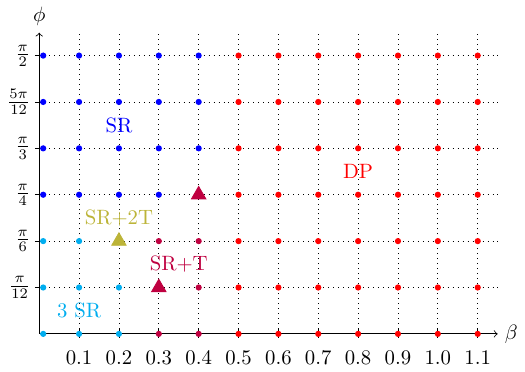}
\end{center}
\caption{Left: Energy of minimizers as function of $\beta$ around the peanut shape for different values of $\phi$ between $0$ (solid line) and $\frac{\pi}{2}$ (dotted line). The mesh consists of around $485\, 000$ cells of size $h=0.03$ around the particle surface.
Right: Defect configuration corresponding to the minimal energy for given $\phi$ and $\beta$. Dots indicate simulations with $2\ 000$, triangles with $4\ 000$ iterations. 
We observe dipoles (DP), Saturn rings with one (SR) or three components (3 SR) and Saturn rings with non-trivial surface $T$ of one (SR+T) or two components (SR+2T).
See Figures~\ref{fig:P_min_config}, \ref{fig:image_P_special_phipi4} and \ref{fig:image_P_special_phipi6} for images of these configurations.}
\label{fig:diagr_P_E_beta_phi}
\end{figure}

\begin{figure}
\begin{center}
\begin{subfigure}[c]{0.32\textwidth}
\centering
\includegraphics[scale=0.2]{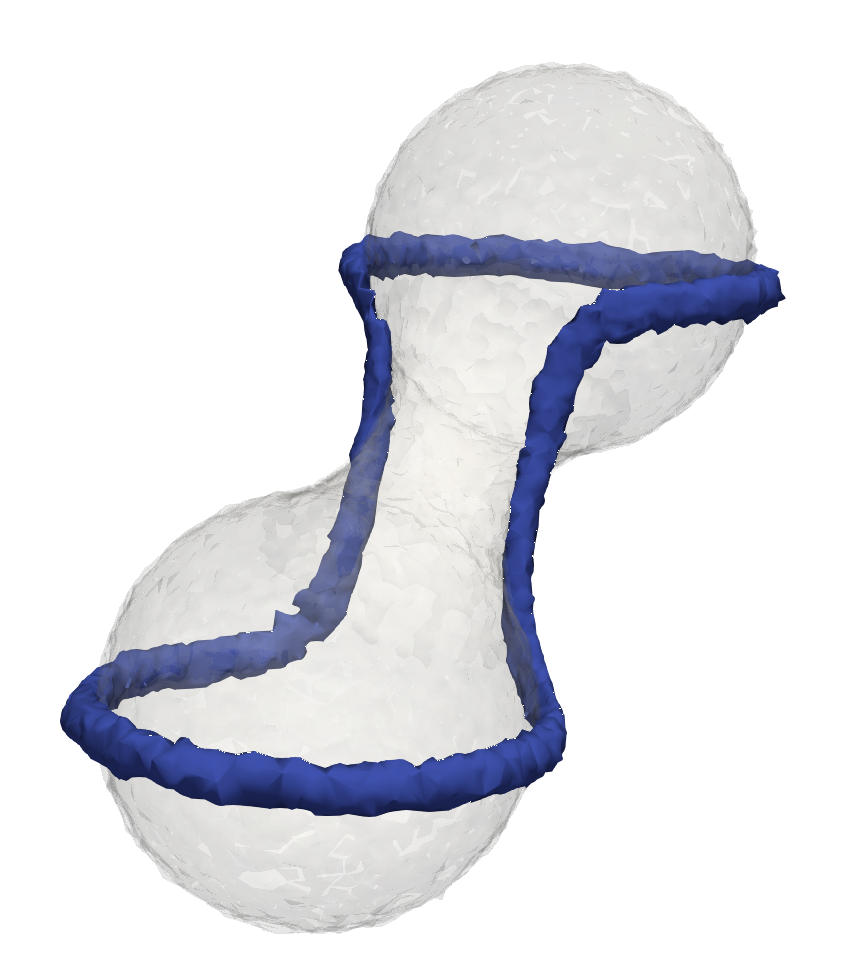}
\caption{} 
\label{fig:image_P_special_phipi4-a}
\end{subfigure}
\begin{subfigure}[c]{0.32\textwidth}
\centering
\includegraphics[scale=0.2]{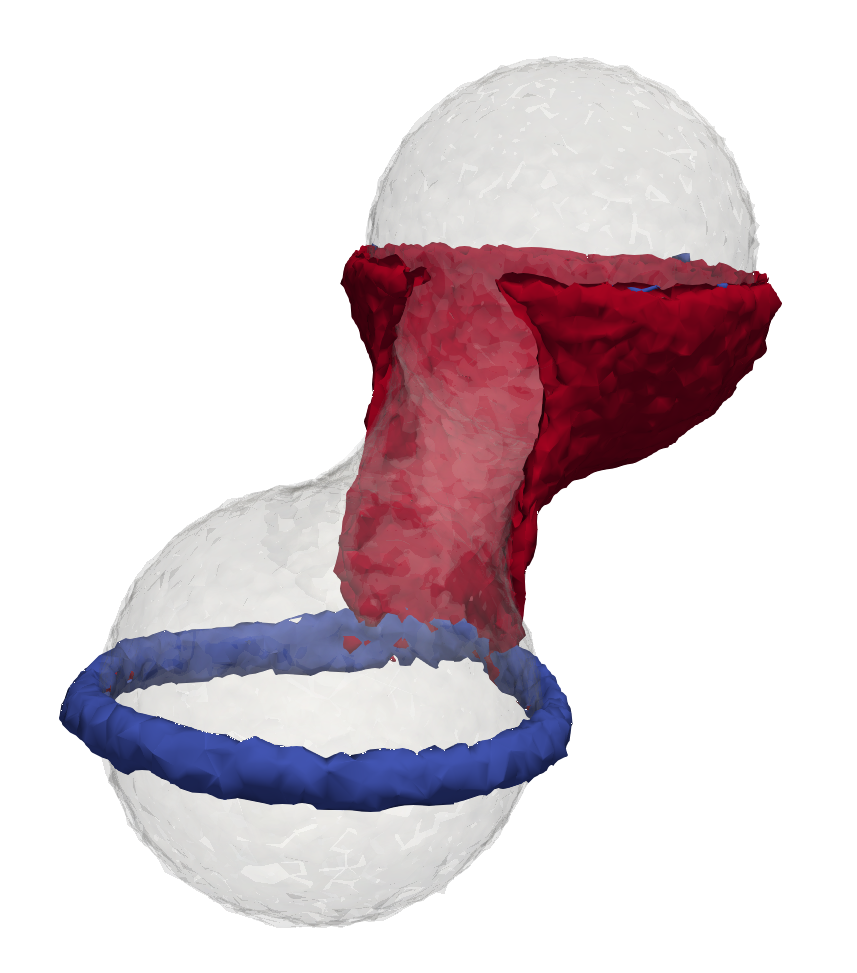}
\caption{} 
\label{fig:image_P_special_phipi4-b}
\end{subfigure}
\begin{subfigure}[c]{0.32\textwidth}
\centering
\includegraphics[scale=0.2]{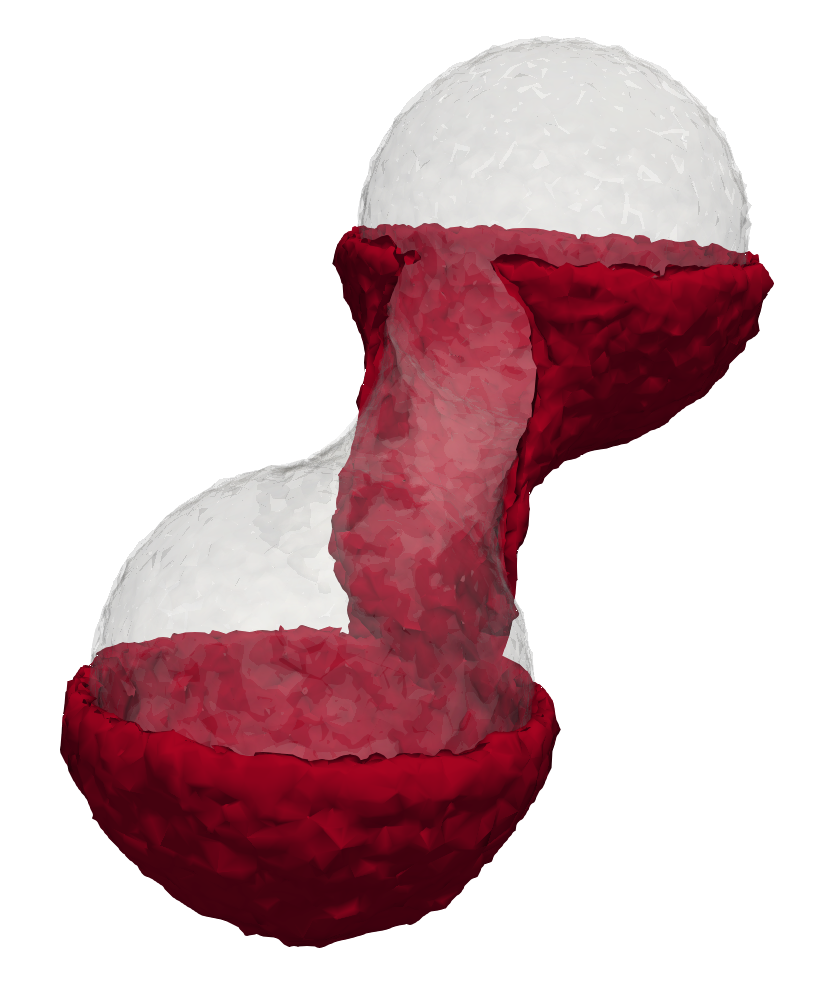}
\caption{} 
\label{fig:image_P_special_phipi4-c}
\end{subfigure}
\end{center}
\caption{Configurations obtained for $\phi = \frac\pi 4$ and $\beta=0.3, 0.4, 0.5\,$.}
\label{fig:image_P_special_phipi4}
\end{figure}

\begin{figure}
\begin{center}
\begin{subfigure}[c]{0.24\textwidth}
\centering
\includegraphics[scale=0.15]{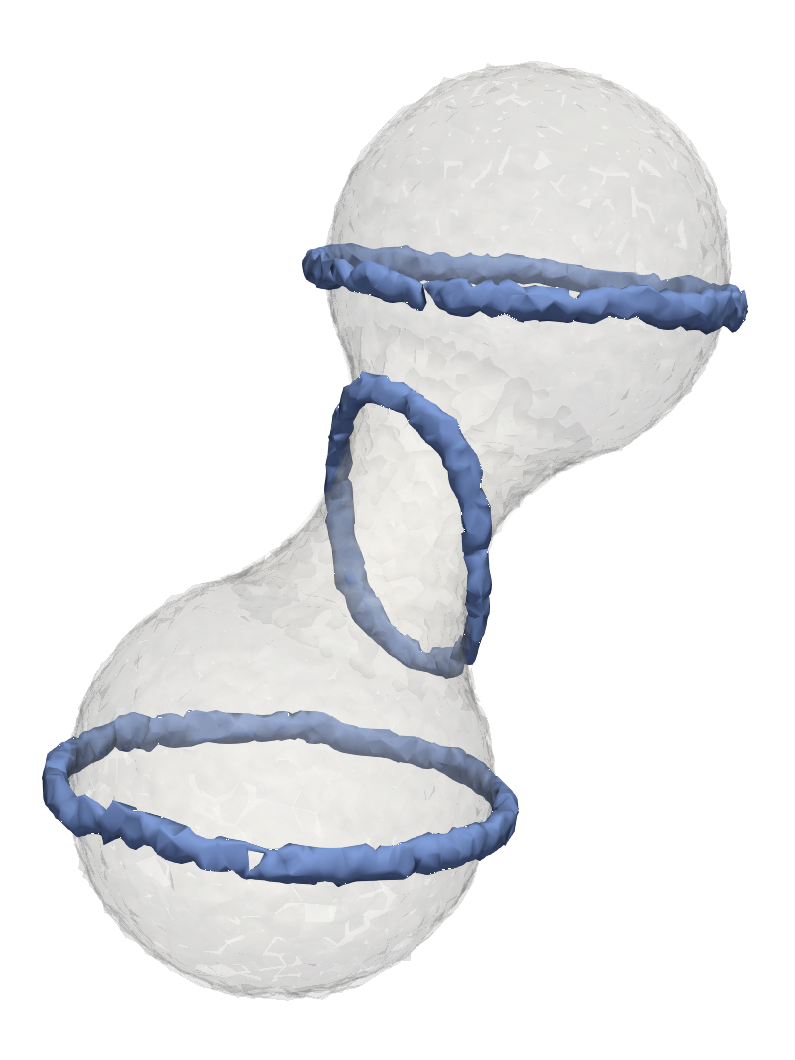}
\caption{} 
\label{fig:image_P_special_phipi6-a}
\end{subfigure}
\begin{subfigure}[c]{0.24\textwidth}
\centering
\includegraphics[scale=0.15]{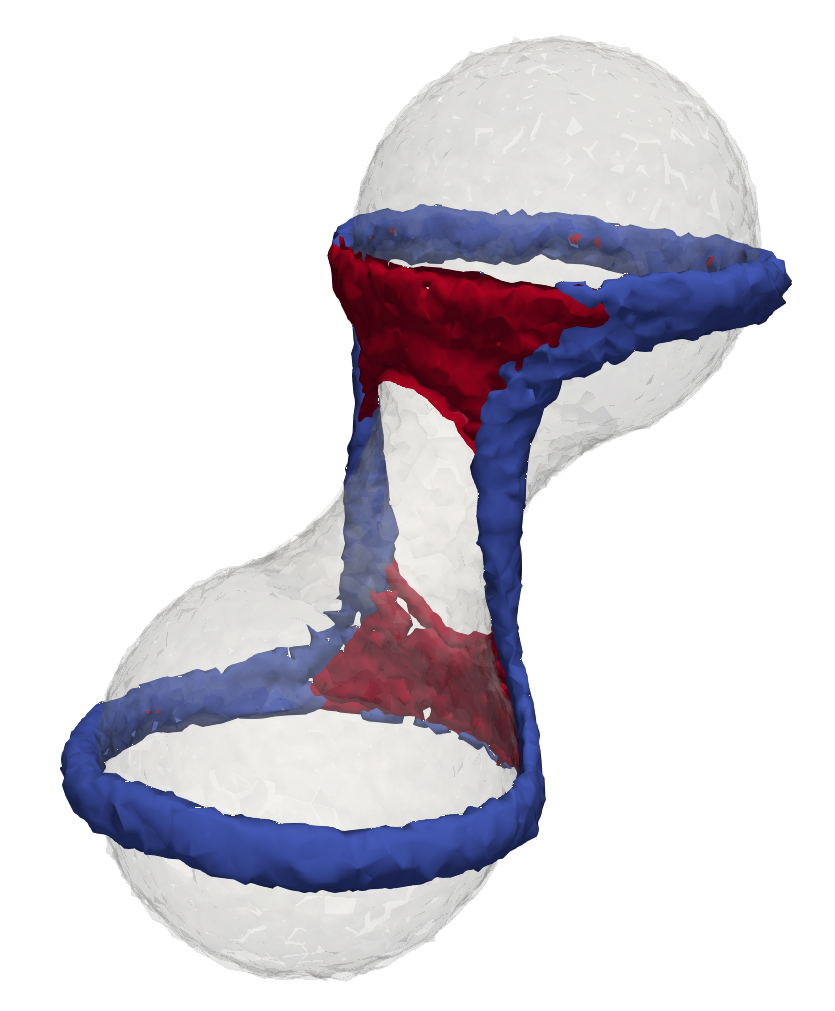}
\caption{} 
\label{fig:image_P_special_phipi6-b}
\end{subfigure}
\begin{subfigure}[c]{0.24\textwidth}
\centering
\includegraphics[scale=0.15]{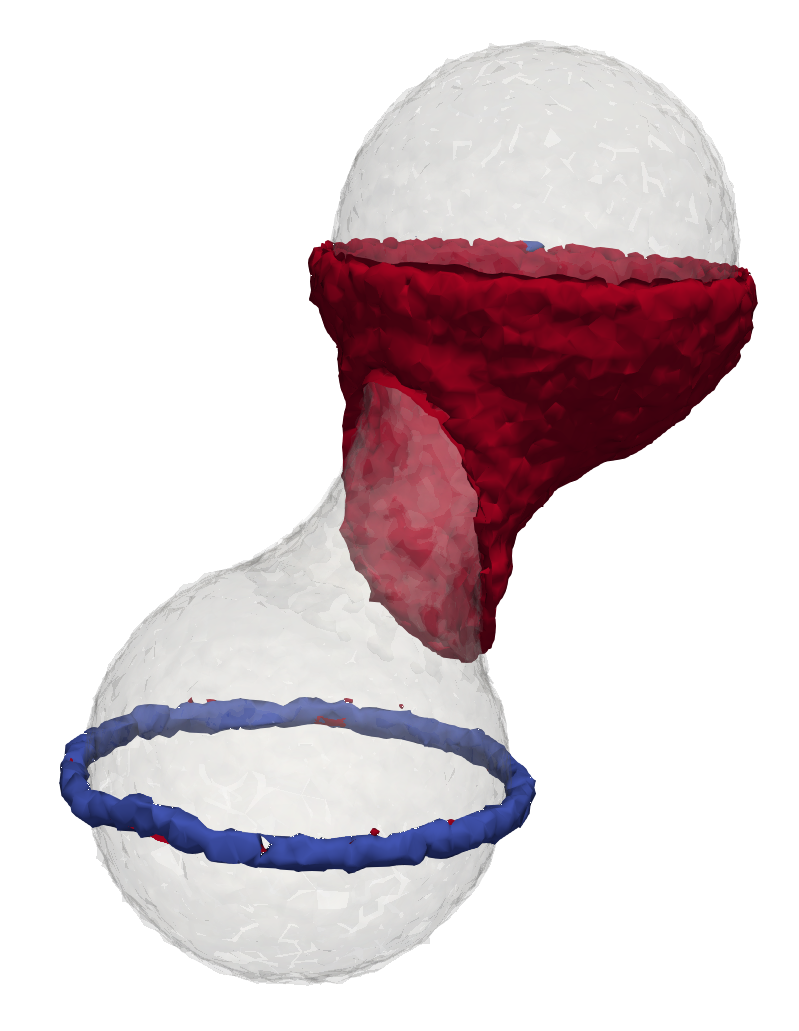}
\caption{} 
\label{fig:image_P_special_phipi6-c}
\end{subfigure}
\begin{subfigure}[c]{0.24\textwidth}
\centering
\includegraphics[scale=0.15]{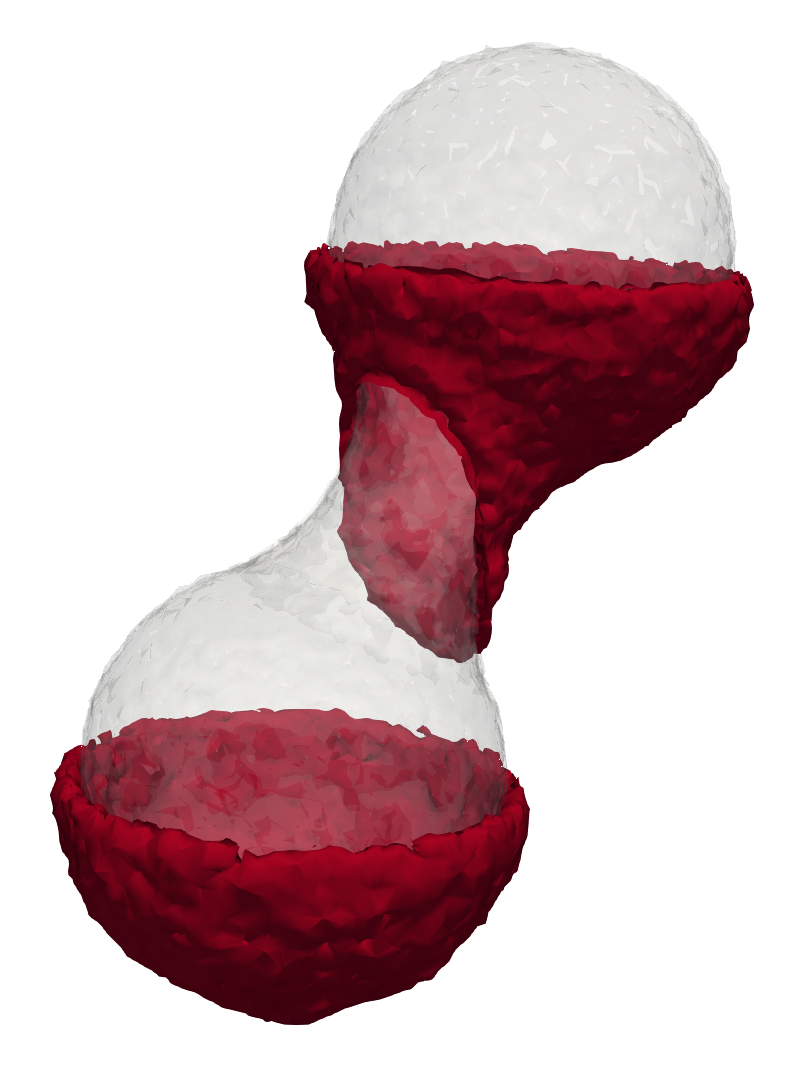}
\caption{} 
\label{fig:image_P_special_phipi6-d}
\end{subfigure}
\end{center}
\caption{Configurations obtained for $\phi = \frac\pi 6$ and $\beta=0.1, 0.2, 0.4, 0.5\,$.}
\label{fig:image_P_special_phipi6}
\end{figure}

As the peanut is non-convex, one could hope for the surface $T$ to form outside the boundary layer $\M_h$. 
We did not observe this in our simulations which can heuristically be explained by the fact that the line $S$ detaching from the surface at distance $d$ would only be shortened by a term of order $d^2$ while the additional $\MM(T\restr\Omega)$ would be of order $d$.
We therefore expect the line $S$ to always stay at the surface of $\M$ unless the dipole becomes energetically favourable.

\subsection{Donut-shaped particle}

To illustrate the case of a particle with non-trivial topology, we study the a donut-shaped colloid with inner radius $r>0$ and outer radius $R>r$.
We choose $R=0.7$ and $r=0.4$ (i.e.\ a ratio of $1.75$) to illustrate our findings, other ratios exhibit similar minimizing configurations, see for example Figure~\ref{fig:image_D105_phi15} in which a ratio of $R/r=5$ has been used.

The landscape of minimizers is more complex than in the case of the peanut-shaped particle, we give an overview of observed minimizers in Figure~\ref{fig:diagr_D_E_beta_phi_2}.
It is worth pointing out that from the plot of the energy as function of $\beta$ we can conclude that the angle $\phi\approx\frac{\pi}{2}$ is minimizing for any $\beta$ and thus if the particle is free to rotate inside the liquid crystal, one should expect to observe configurations as in Figure~\ref{fig:image_D419_phi15} and \ref{fig:image_D105_phi15} in the particular limit in which the model \eqref{intro:eq:E0} is valid, see \cite{ACS2021}.
This asymptotic model does not cover the whole range of physical parameters, as we are not able to justify the configuration corresponding to $\phi=0$ that has been observed in \cite{Senyuk2012} (with the magnetic field replaced by an electric field). \\
If the particle is not free to rotate, one typically observes a transition for increasing $\beta$ from two Saturn rings (one small interior and one big exterior) by first replacing the smaller ring defect by a piece of $T$, either on the surface $\M$ and/or inside $\Omega$.
Increasing $\beta$ even further, both rings are connected via $T$ on the particle surface and no $T\restr\Omega$ is observed, see Figure~\ref{fig:image_D419_phi0}, \ref{fig:image_D419_phipi4} and \ref{fig:image_D419_phi5pi12}.

\begin{figure}
\begin{center}
\begin{subfigure}[c]{0.32\textwidth}
\centering
\includegraphics[scale=0.15]{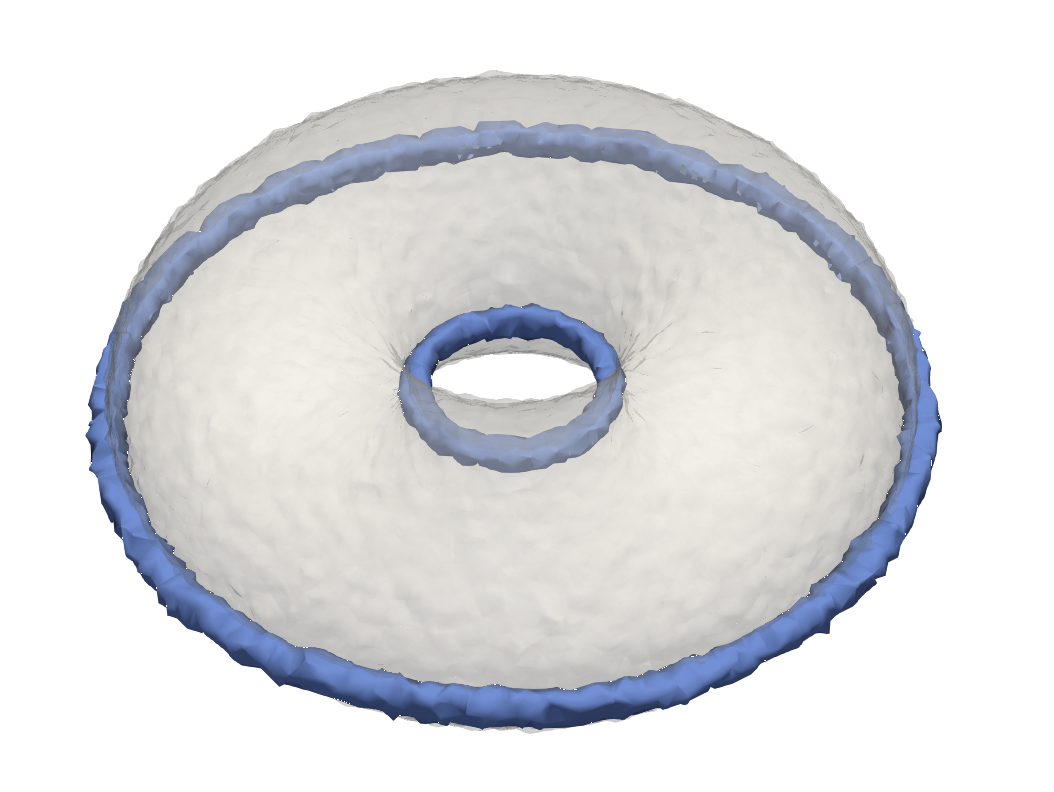}
\caption{} 
\label{fig:image_D419_phi0-a}
\end{subfigure}
\begin{subfigure}[c]{0.32\textwidth}
\centering
\includegraphics[scale=0.15]{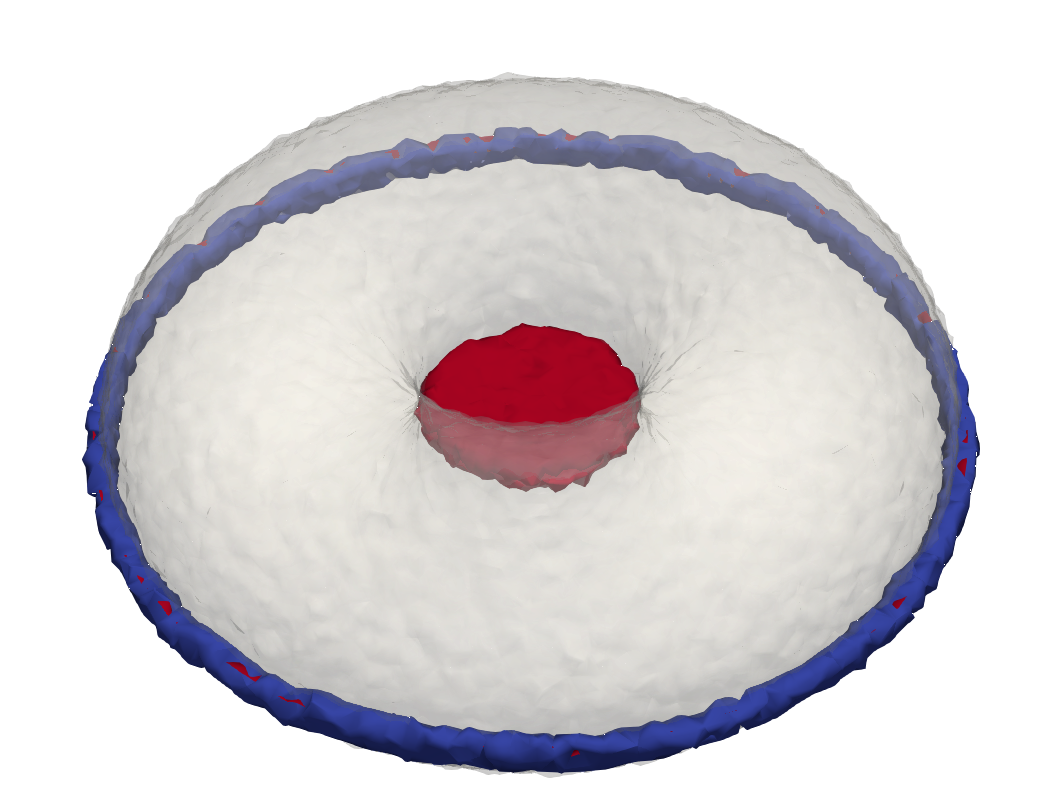}
\caption{} 
\label{fig:image_D419_phi0-b}
\end{subfigure}
\begin{subfigure}[c]{0.32\textwidth}
\centering
\includegraphics[scale=0.15]{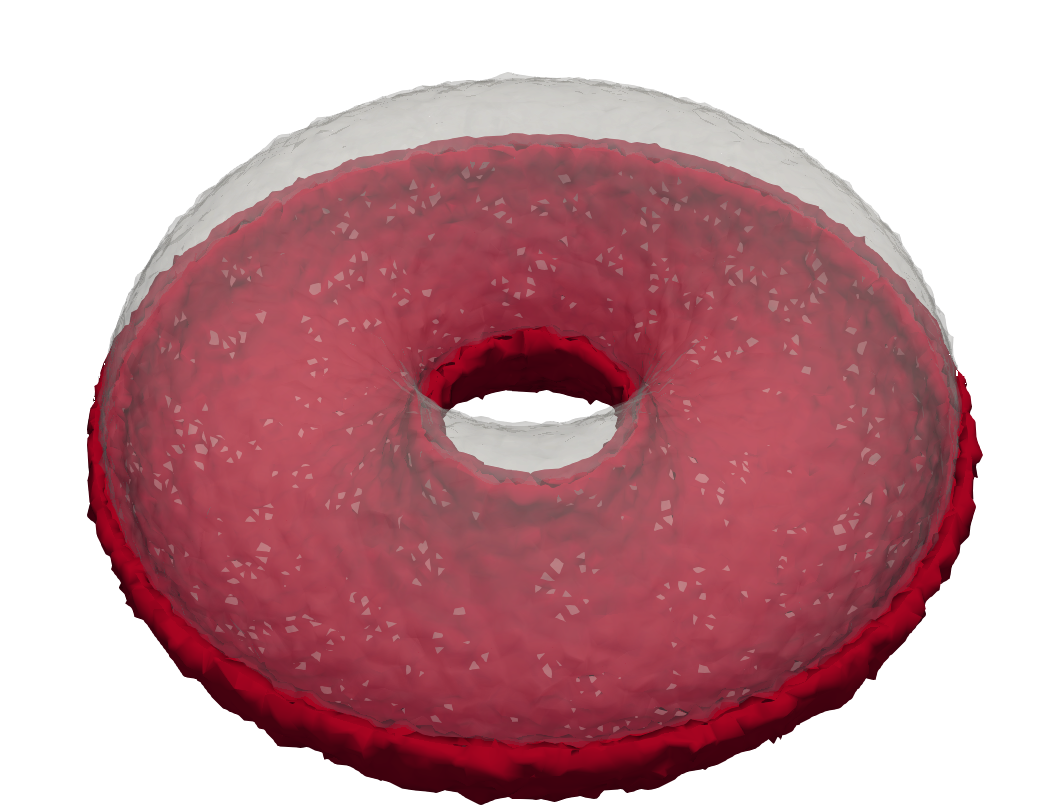}
\caption{} 
\label{fig:image_D419_phi0-c}
\end{subfigure}
\end{center}
\caption{Configurations obtained for $\phi = 0$ and $\beta=0.1, 0.2, 0.6\,$.}
\label{fig:image_D419_phi0}
\end{figure}

\begin{figure}
\begin{center}
\begin{subfigure}[c]{0.49\textwidth}
\centering
\includegraphics[scale=0.25]{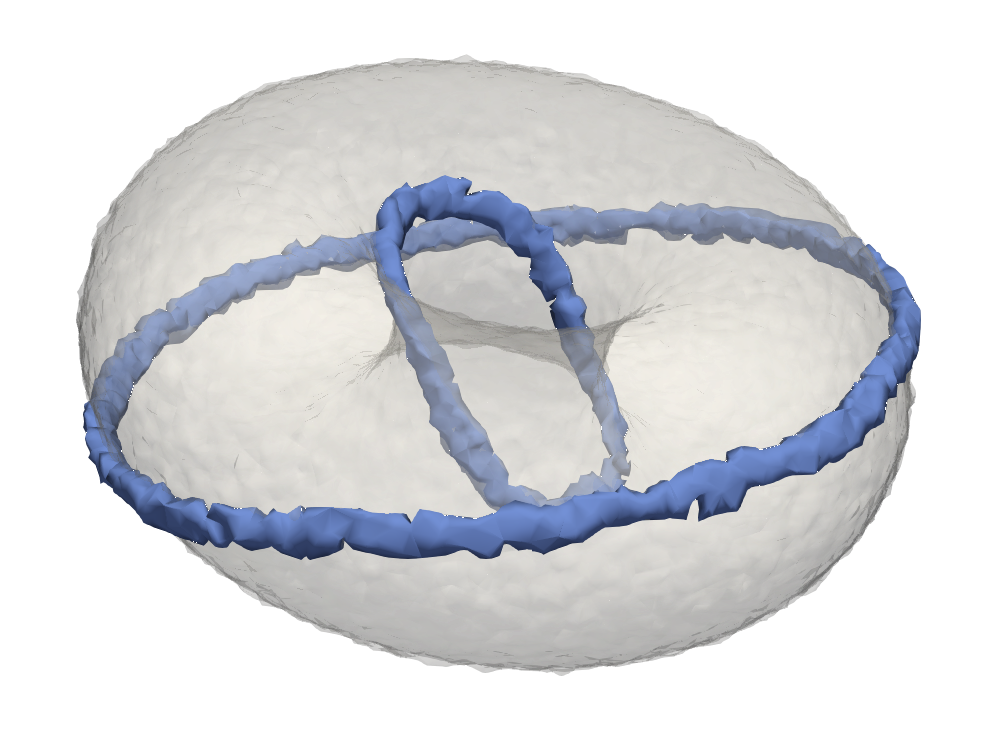}
\caption{} 
\label{fig:image_D419_phipi4-a}
\end{subfigure}
\begin{subfigure}[c]{0.49\textwidth}
\centering
\includegraphics[scale=0.25]{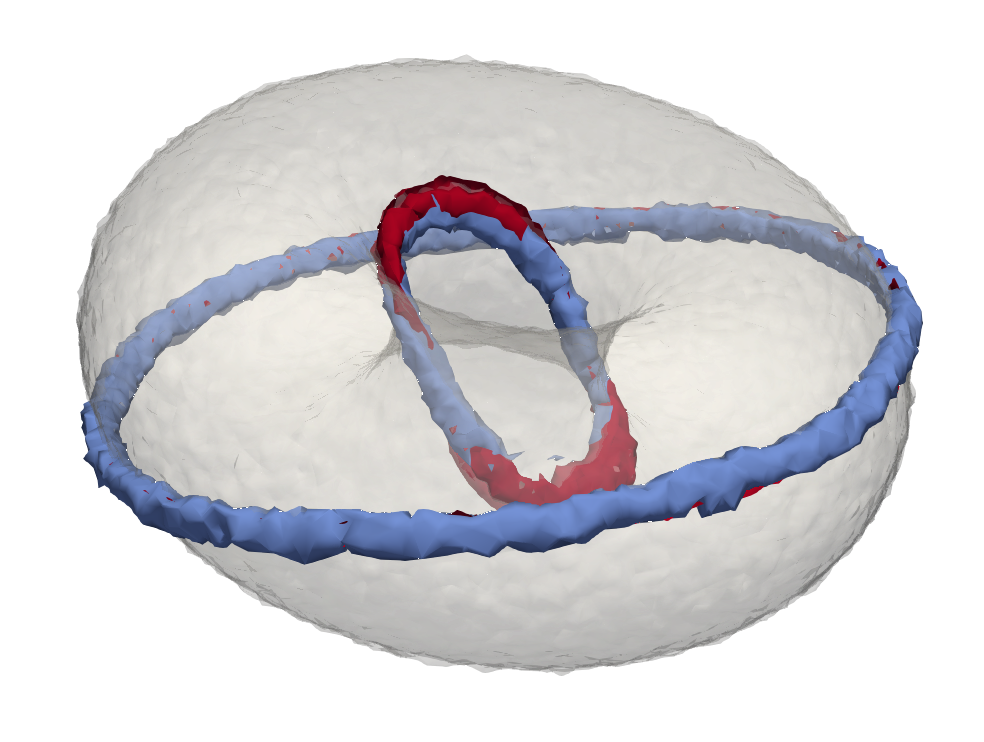}
\caption{} 
\label{fig:image_D419_phipi4-b}
\end{subfigure}
\begin{subfigure}[c]{0.49\textwidth}
\centering
\includegraphics[scale=0.25]{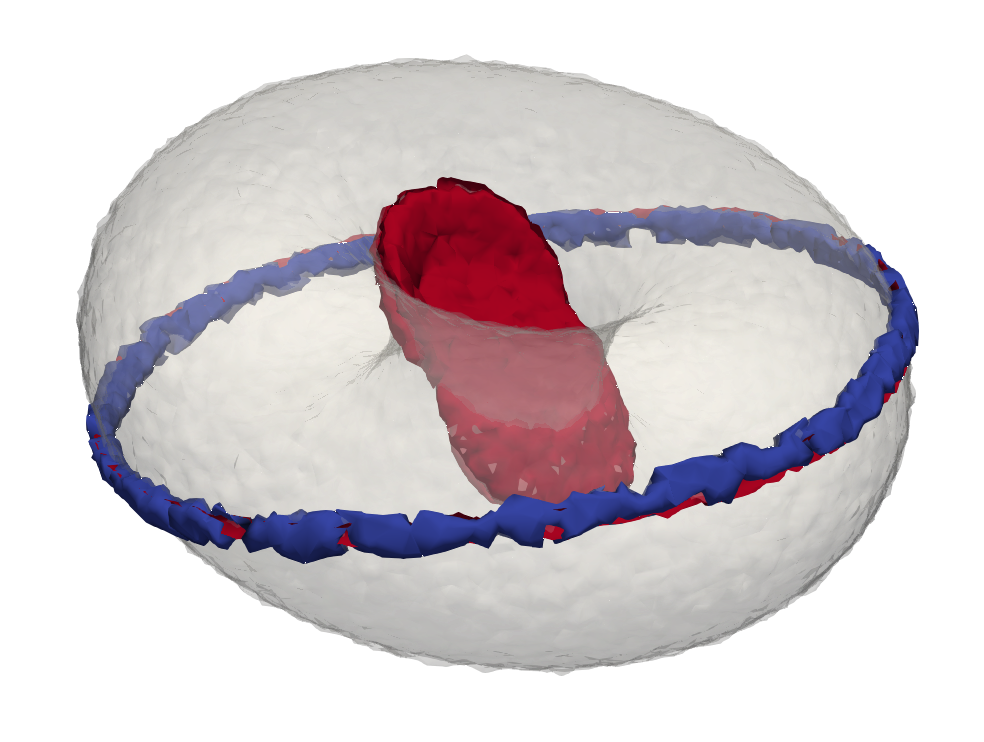}
\caption{} 
\label{fig:image_D419_phipi4-c}
\end{subfigure}
\begin{subfigure}[c]{0.49\textwidth}
\centering
\includegraphics[scale=0.25]{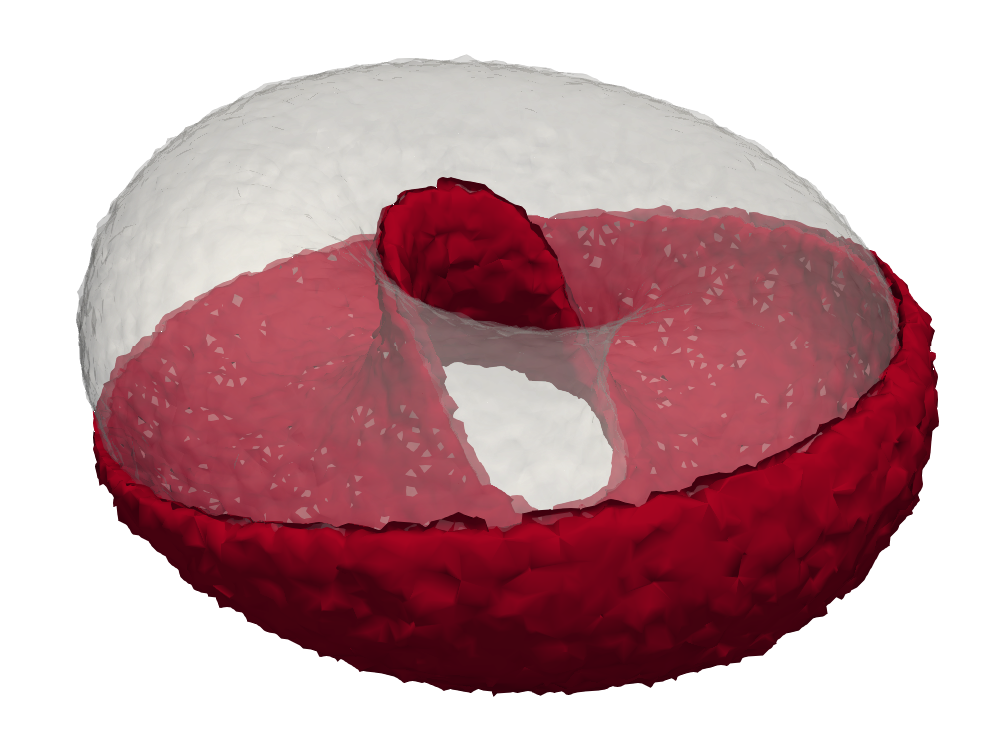}
\caption{} 
\label{fig:image_D419_phipi4-d}
\end{subfigure}
\end{center}
\caption{Configurations obtained for $\phi = \frac{\pi}{4}$ and $\beta=0.01, 0.1, 0.2, 0.5\,$.}
\label{fig:image_D419_phipi4}
\end{figure}

\begin{figure}
\begin{center}
\begin{subfigure}[c]{0.32\textwidth}
\centering
\includegraphics[scale=0.15]{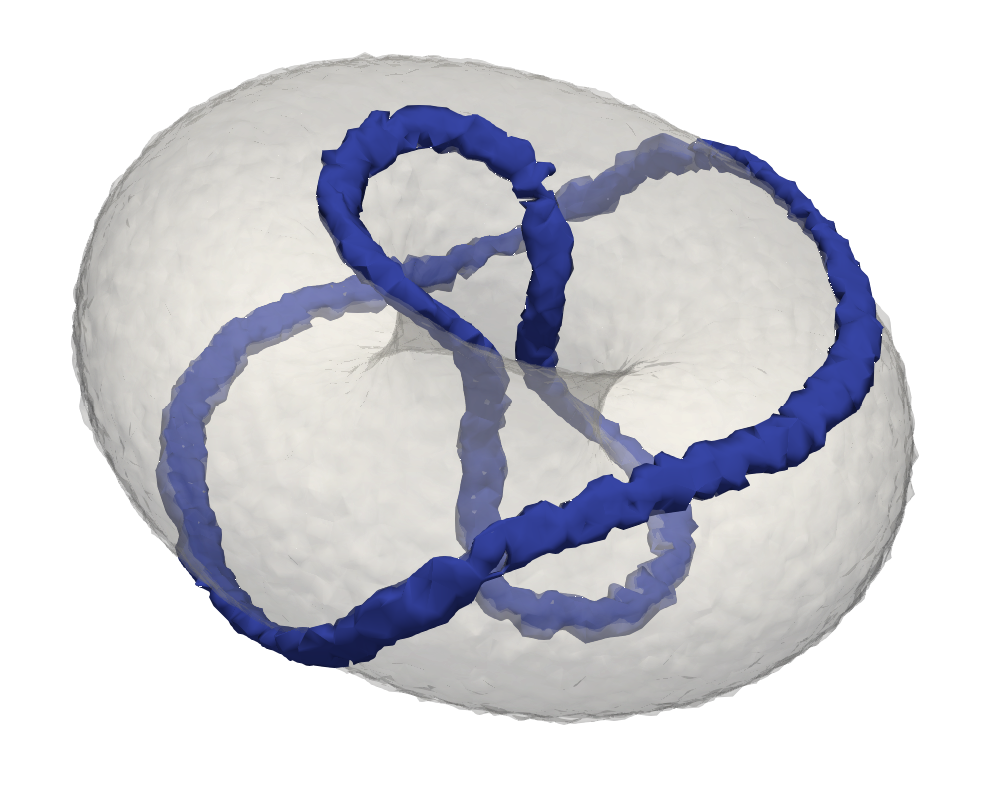}
\caption{} 
\label{fig:image_D419_phi5pi12-a}
\end{subfigure}
\begin{subfigure}[c]{0.32\textwidth}
\centering
\includegraphics[scale=0.15]{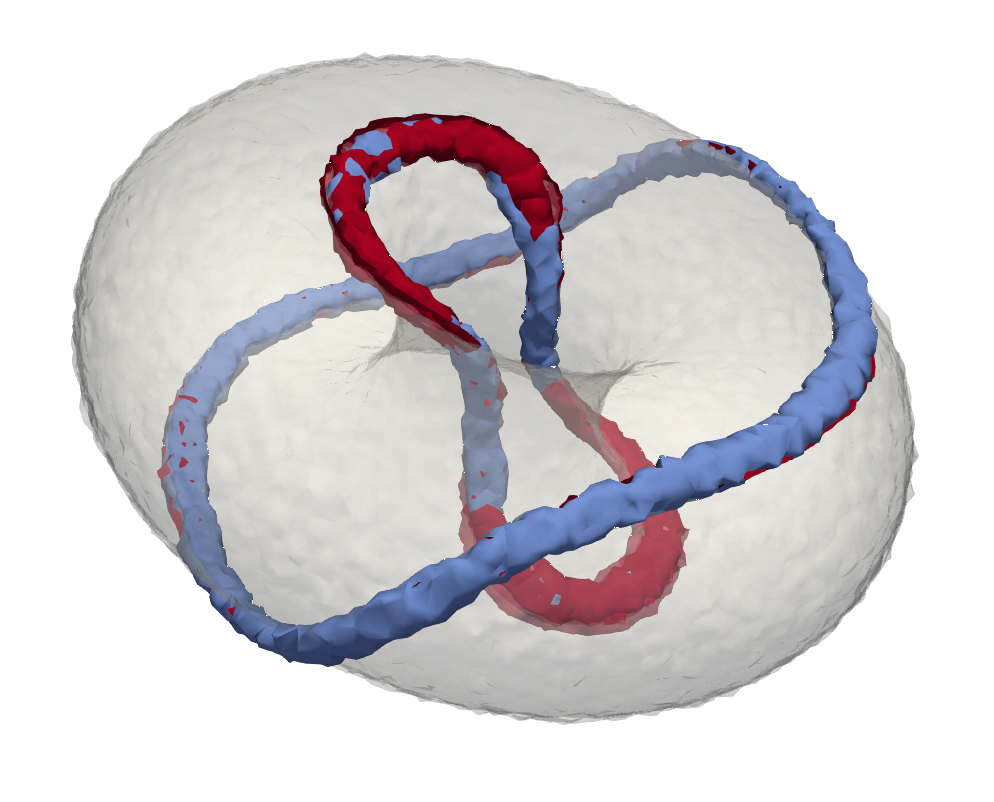}
\caption{} 
\label{fig:image_D419_phi5pi12-b}
\end{subfigure}
\begin{subfigure}[c]{0.32\textwidth}
\centering
\includegraphics[scale=0.15]{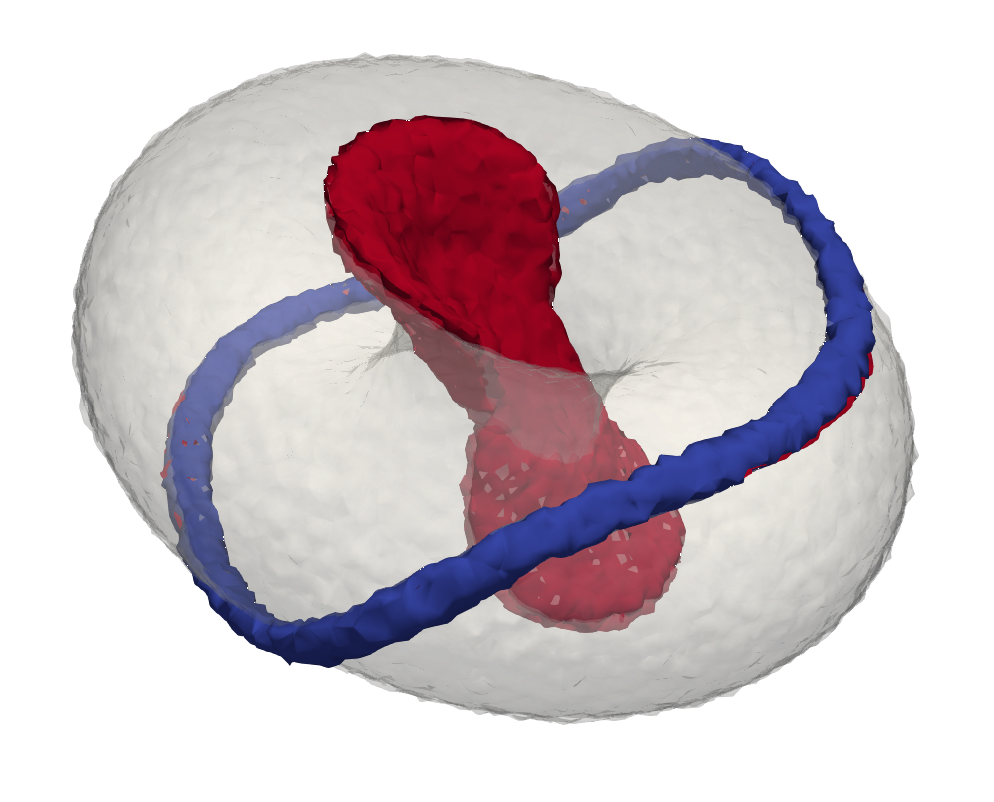}
\caption{} 
\label{fig:image_D419_phi5pi12-c}
\end{subfigure}
\begin{subfigure}[c]{0.32\textwidth}
\centering
\includegraphics[scale=0.15]{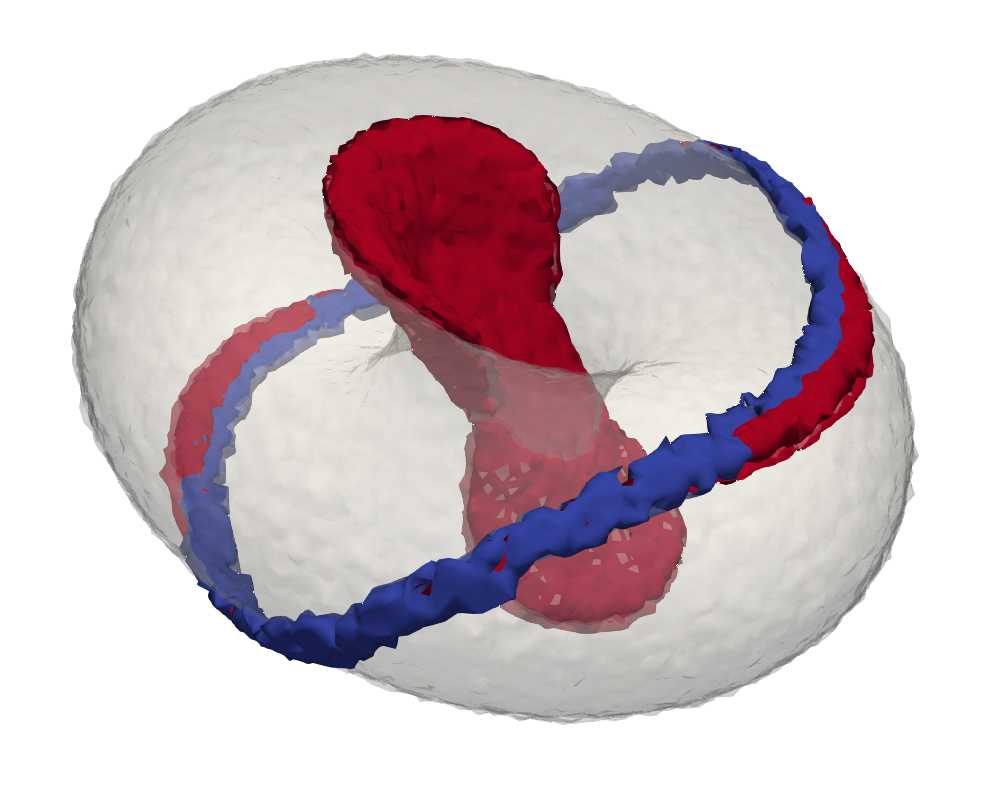}
\caption{} 
\label{fig:image_D419_phi5pi12-d}
\end{subfigure}
\begin{subfigure}[c]{0.32\textwidth}
\centering
\includegraphics[scale=0.15]{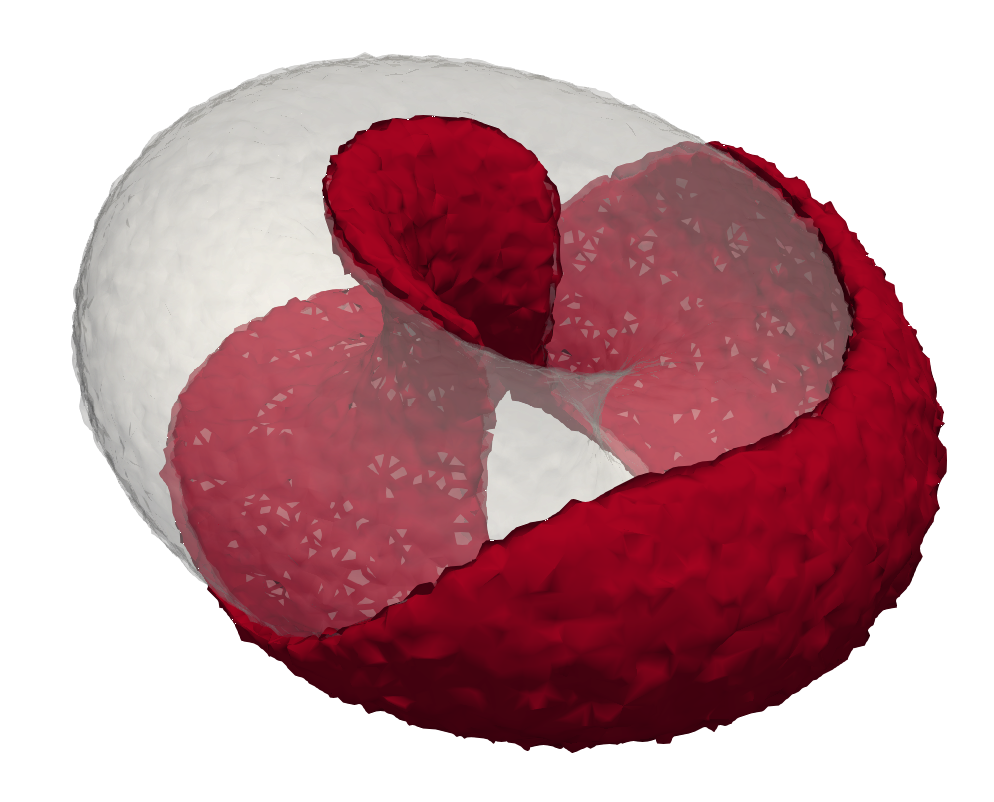}
\caption{} 
\label{fig:image_D419_phi5pi12-e}
\end{subfigure}
\end{center}
\caption{Configurations obtained for $\phi = \frac{5\pi}{12}$ and $\beta=0.01, 0.1, 0.2, 0.3, 0.4\,$.}
\label{fig:image_D419_phi5pi12}
\end{figure}

\begin{figure}
\begin{center}
\begin{subfigure}[c]{0.32\textwidth}
\centering
\includegraphics[scale=0.15]{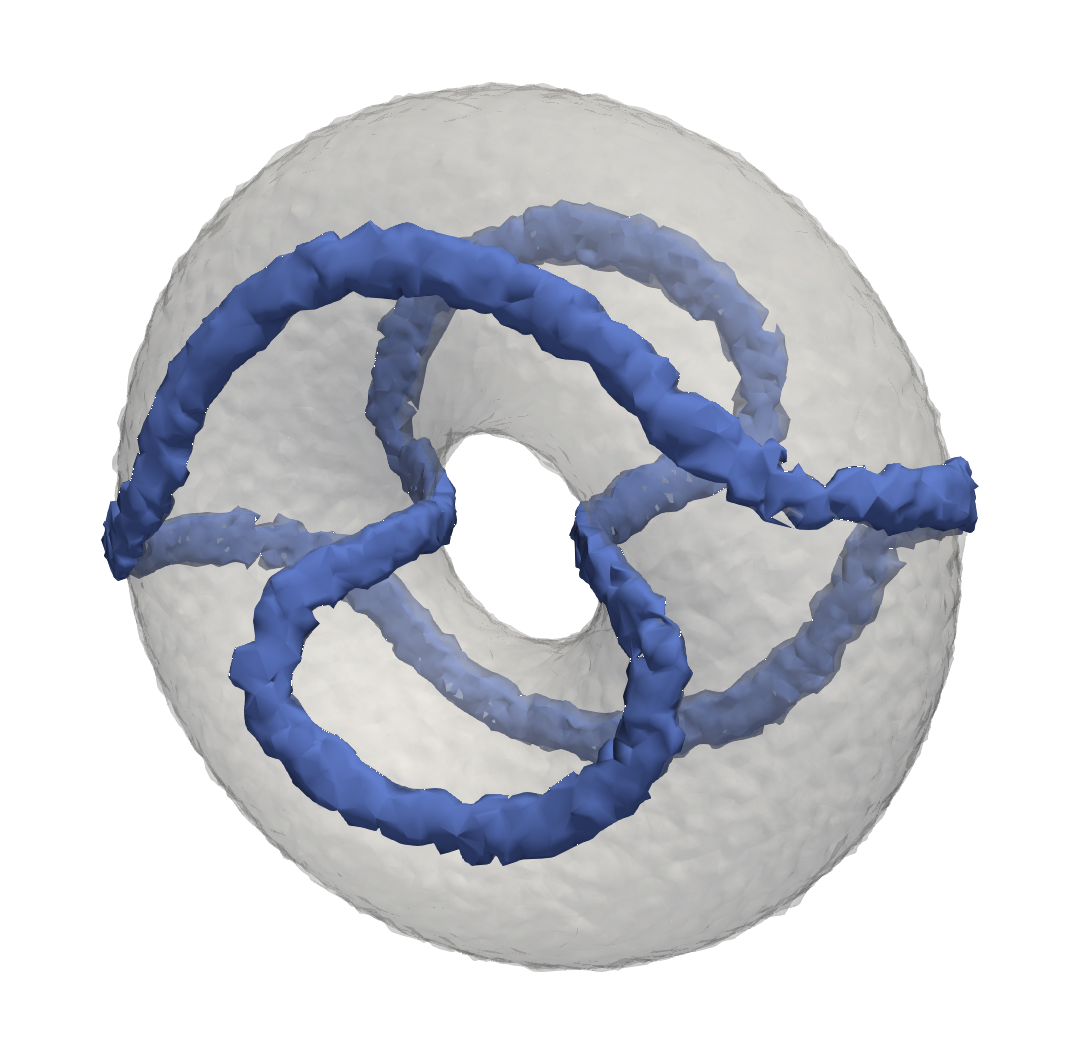}
\caption{} 
\label{fig:image_D419_phi15-a}
\end{subfigure}
\begin{subfigure}[c]{0.32\textwidth}
\centering
\includegraphics[scale=0.15]{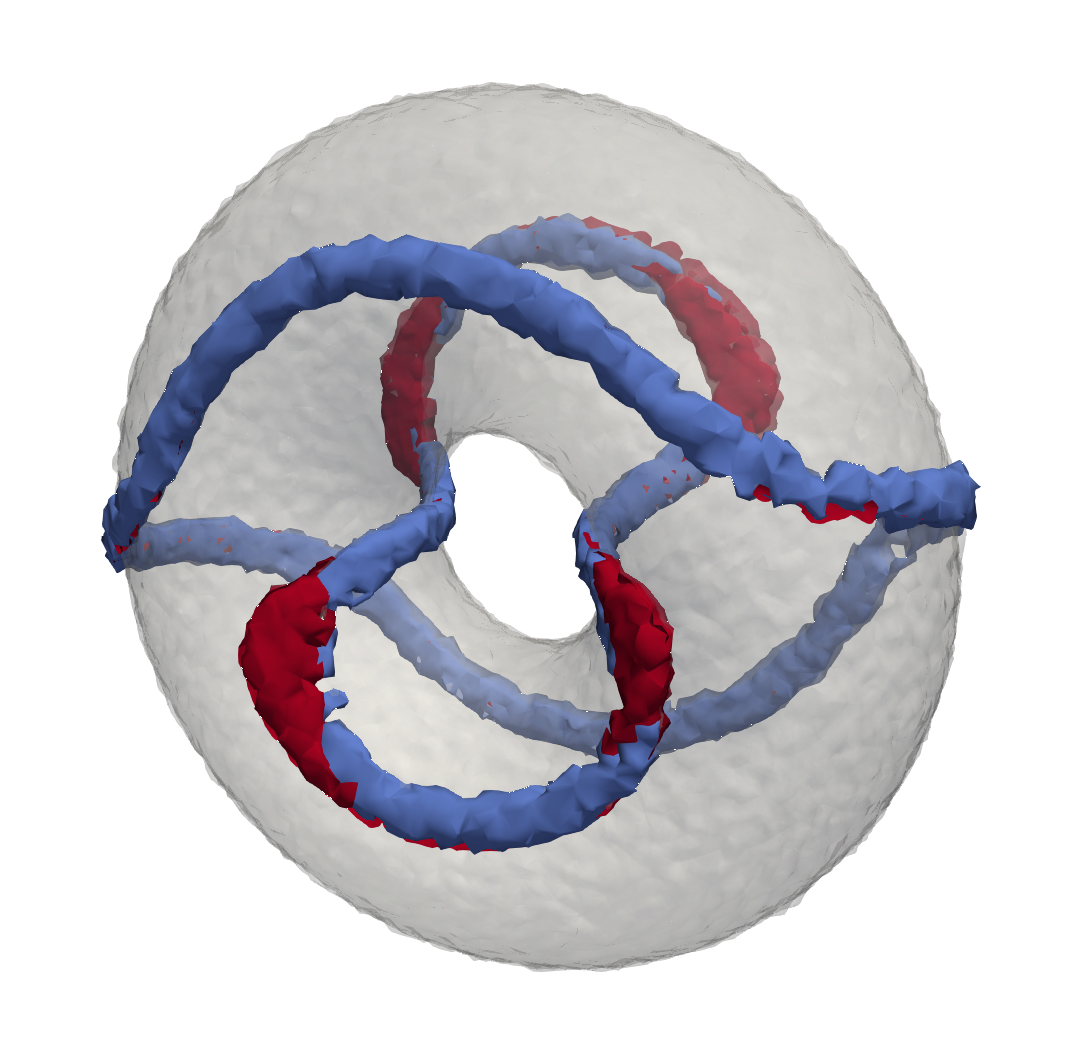}
\caption{} 
\label{fig:image_D419_phi15-b}
\end{subfigure}
\begin{subfigure}[c]{0.32\textwidth}
\centering
\includegraphics[scale=0.15]{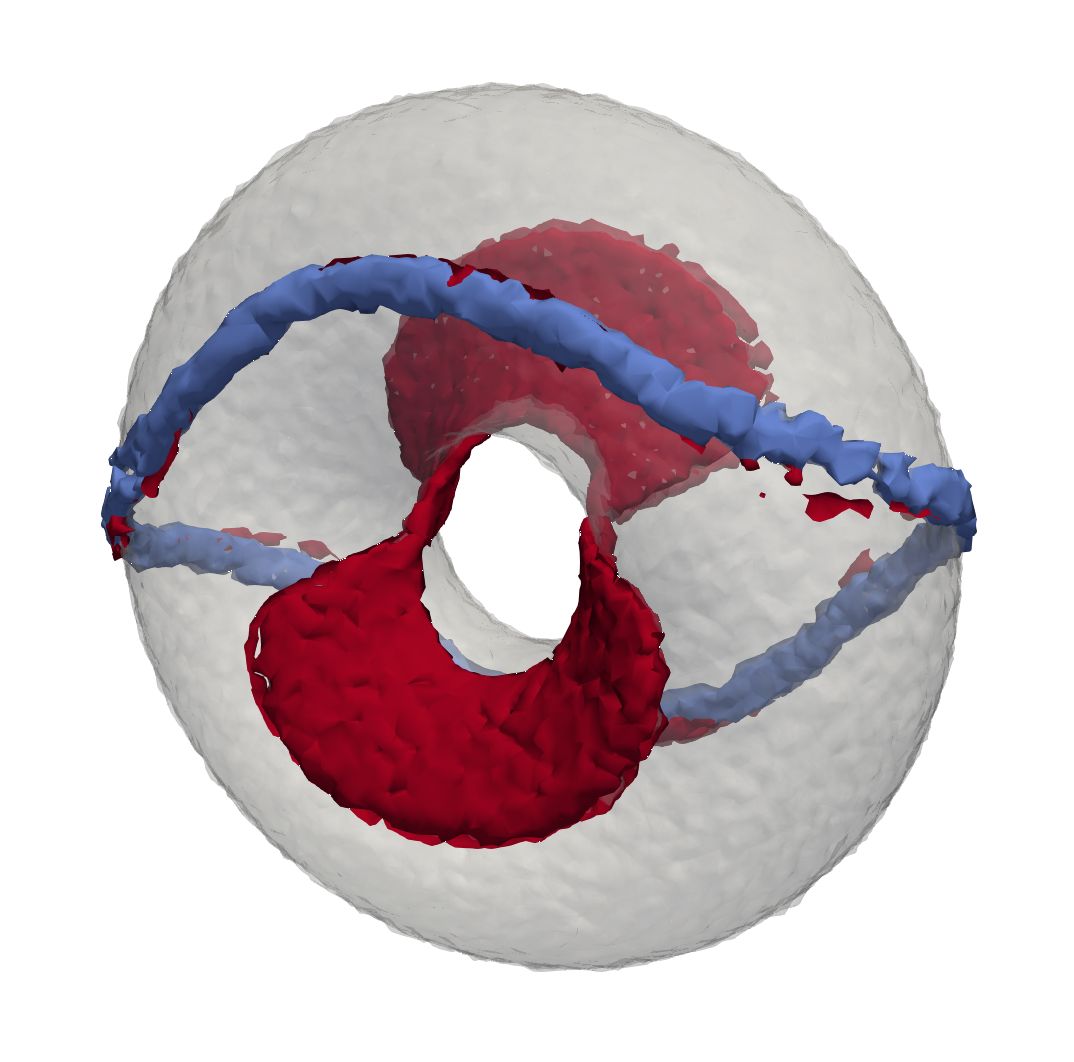}
\caption{} 
\label{fig:image_D419_phi15-c}
\end{subfigure}
\begin{subfigure}[c]{0.32\textwidth}
\centering
\includegraphics[scale=0.15]{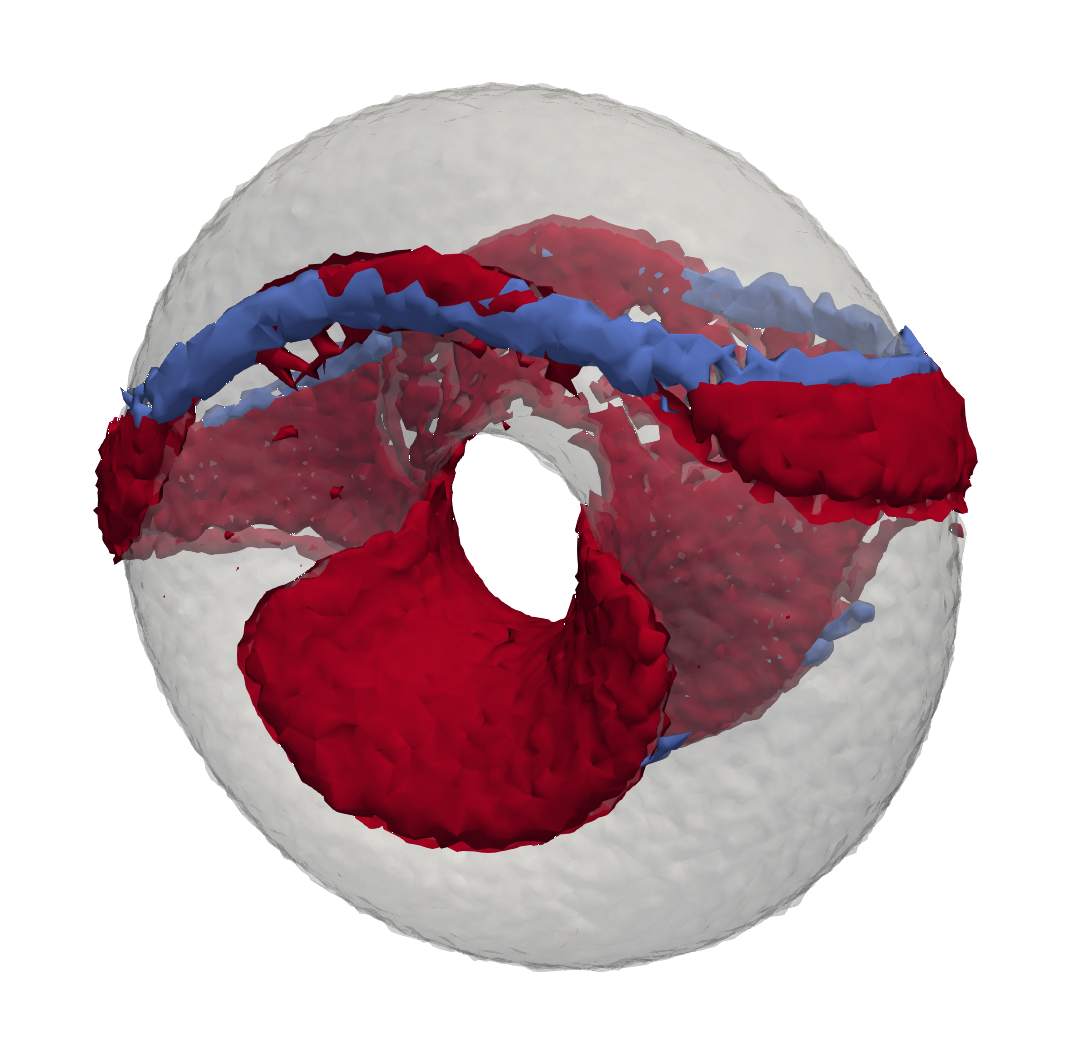}
\caption{} 
\label{fig:image_D419_phi15-d}
\end{subfigure}
\begin{subfigure}[c]{0.32\textwidth}
\centering
\includegraphics[scale=0.15]{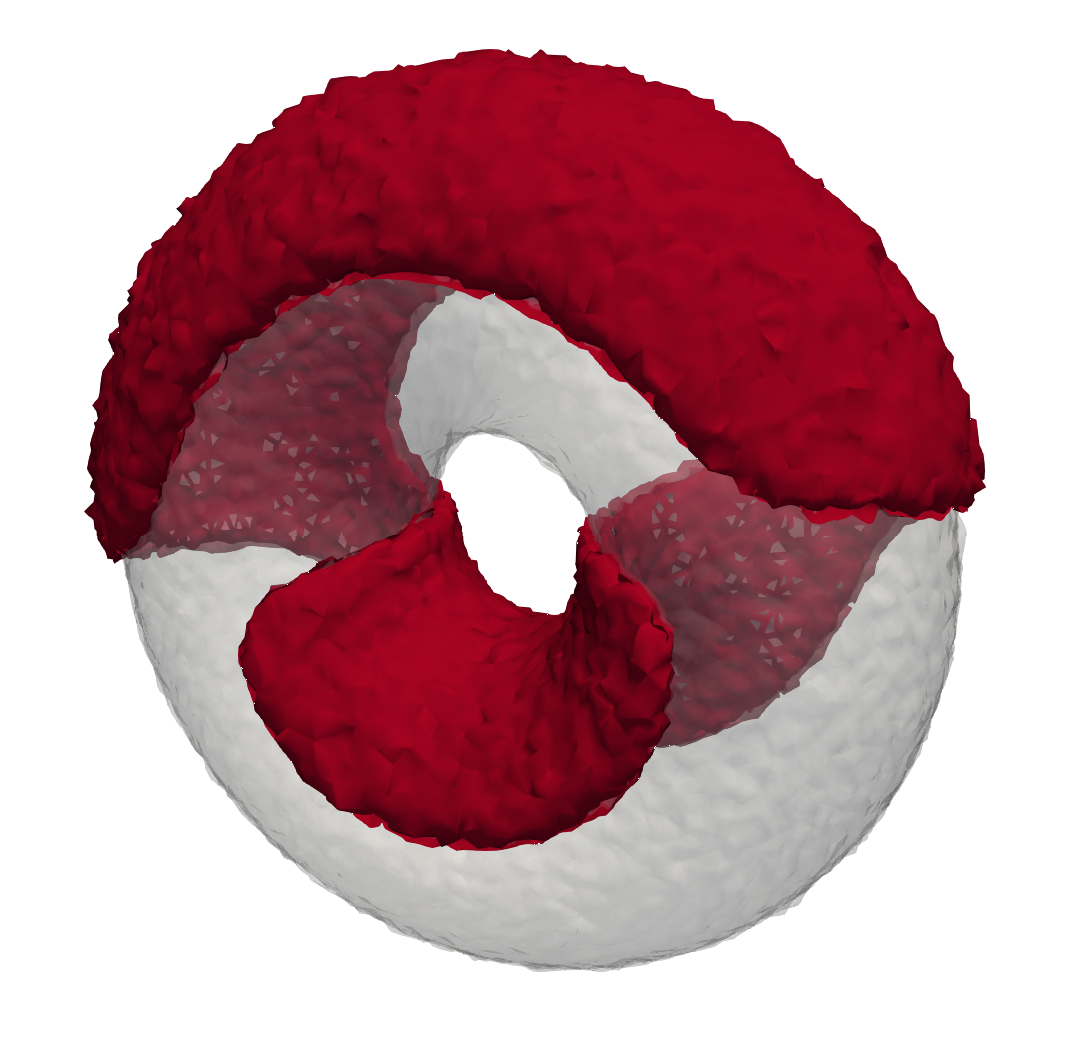}
\caption{} 
\label{fig:image_D419_phi15-e}
\end{subfigure}
\end{center}
\caption{Configurations obtained for $\phi = 1.5$ and $\beta=0.01, 0.1, 0.2, 0.3, 0.4\,$.}
\label{fig:image_D419_phi15}
\end{figure}

\begin{figure}[H]
\begin{center}
\includegraphics[scale=0.7]{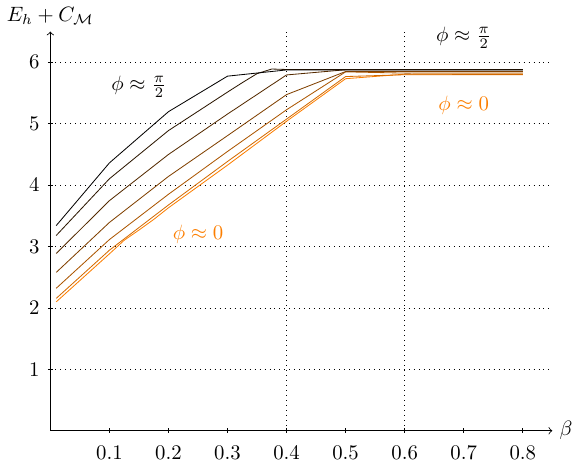}
\hspace*{0.5cm}
\includegraphics[scale=0.7]{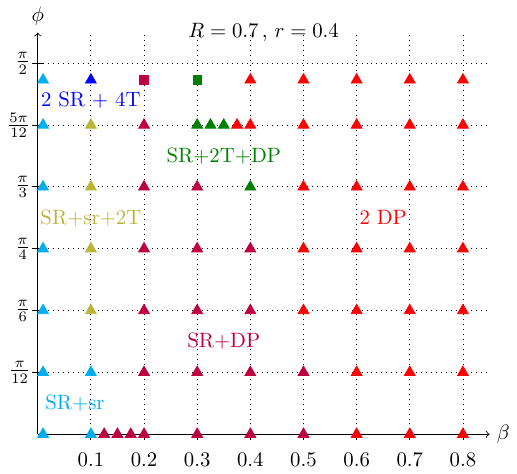}
\end{center}
\caption{Left: Energy of minimizers as function of $\beta$ around the donut shape ($R/r=1.75$) for values of $\phi\in\{0,\frac{1\pi}{12},\frac{1\pi}{6},\frac{3\pi}{12},\frac{1\pi}{3},\frac{5\pi}{12},\frac{11\pi}{24}\}$ between $0$ (black line) and $\frac{11\pi}{24}$ (orange line). 
The mesh consists of around $419\, 000$ cells of size $h=0.04$ around the particle surface.
Right: Defect configuration corresponding to the minimal energy for given $\phi$ and $\beta$. Triangles indicate simulations with $4\ 000$ iterations, squares with $8\ 000$. 
We observe configurations with two point defects (2 DP), two Saturn rings (one larger than the other, SR+sr) and Saturn rings with non-trivial surface $T$ of one (2SR+2T and 2SR+4T) as well as a single Saturn ring-point defect combination with (SR+2T+DP) or without surface $T$ (SR+DP). See Figures~\ref{fig:image_D419_phi0}, \ref{fig:image_D419_phipi4}, \ref{fig:image_D419_phi5pi12} and \ref{fig:image_D419_phi15} for images of these configurations.}
\label{fig:diagr_D_E_beta_phi_2}
\end{figure}

\begin{figure}
\begin{center}
\begin{subfigure}[c]{0.32\textwidth}
\centering
\includegraphics[scale=0.22]{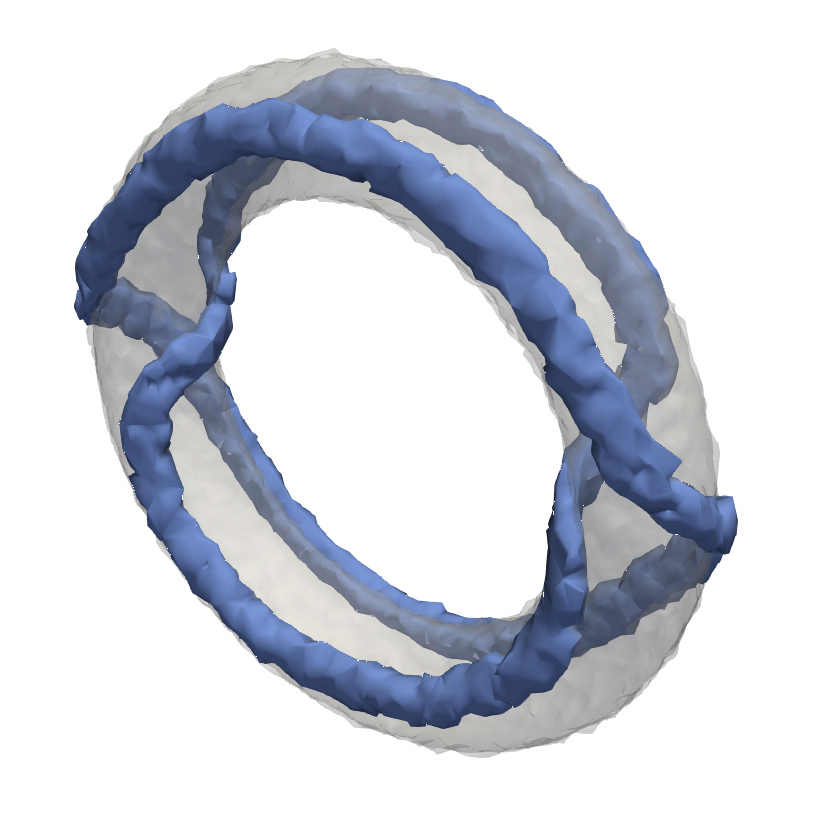}
\caption{} 
\label{fig:image_D105_phi15-a}
\end{subfigure}
\begin{subfigure}[c]{0.32\textwidth}
\centering
\includegraphics[scale=0.22]{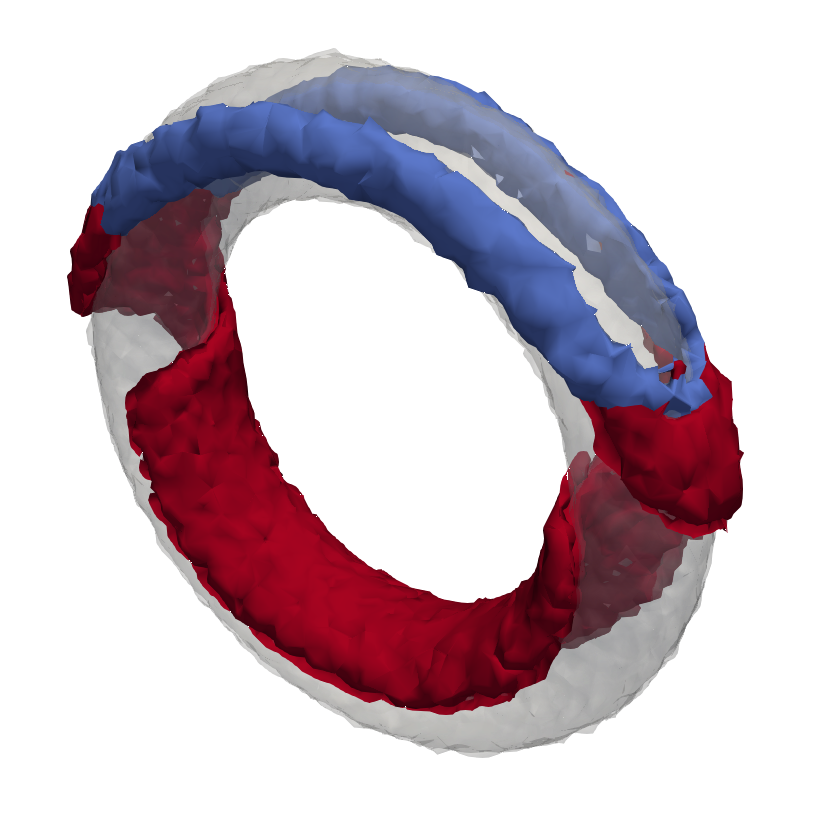}
\caption{} 
\label{fig:image_D105_phi15-b}
\end{subfigure}
\begin{subfigure}[c]{0.32\textwidth}
\centering
\includegraphics[scale=0.22]{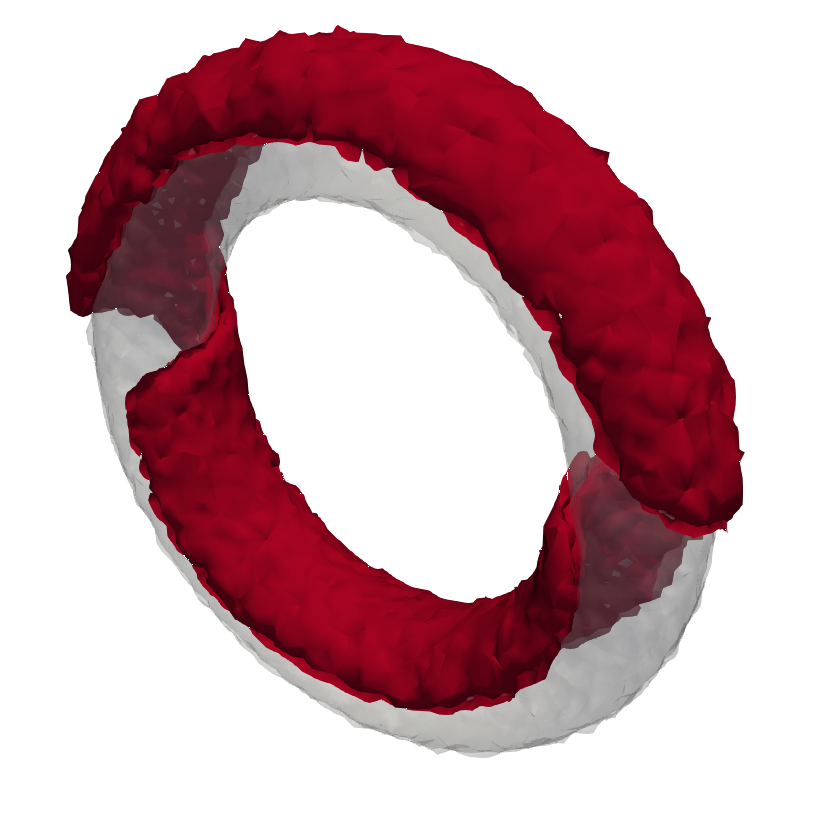}
\caption{} 
\label{fig:image_D105_phi15-c}
\end{subfigure}
\end{center}
\caption{Configurations obtained for $\phi = 1.5$ and $\beta=0.01, 0.1, 0.2\,$. The ratio of radii for this donut is $R/r=5$.}
\label{fig:image_D105_phi15}
\end{figure}


\subsection{Croissant-shaped particle}

Our interest in croissant-shaped particles is twofold:
\begin{enumerate}
\item Until now, both $S$ has been entirely included in the boundary layer $\M_h$. 
We would therefore like to give an example of when $S$ detaches from $\M$ and have parts inside $\Omega$.
The examples given here further illustrate the variety and complexity of minimizers $(T,S)$ of the limit problem.
\item From the point of view of applications it appears that particles similar to our croissant shape ("horse shoe") are interesting since they are able to self-assemble into two- and three-dimensional nematic colloidal crystals \cite{Aplinc2019}. 
The precise understanding of defect structure could therefore be valuable e.g.\ for tunable metamaterials.
\end{enumerate}

The particle we use in our simulations is made from five components. 
The central piece is a half-torus to which two cylindrical parts are attached.
The remaining open ends of the cylinders are closed using two half-spheres.
Since the croissant has less symmetries than the peanut, we describe its orientation relative to the external field by the two angles $\phi$ (rotation around the $x_1-$axis) and $\psi$ (rotation around the $x_2-$axis).
The two radii $R,r$ of the torus and the length $L$ of the cylinders are chosen as $R=0.7$, $r=0.4$ and $L=0.5$.
Those values are obtained heuristically as we expect for well chosen orientation and $\beta$ to observe a non-trivial surface $T\restr\Omega$.
Indeed, if the torus lies parallel to the $x_1 x_2-$plane, then in order to shorten the length of $S$, we expect the line $S$ to directly connect the two half-spheres and hence a surface $T$ outside the boundary layer $\M_h$ connects $\Gamma$ to $S\restr\Omega$.
Varying the parameter $\beta$ for $\phi=\psi=0$, we obtain the expected intermediate configuration, see Figure~\ref{fig:image_C_view}. 
The values of $\beta$ for which these transitions occur are contained in the two intervals $(0.25,0.29)$ and $(0.4,0.42)$. 
It is therefore possible to qualitatively study the optimality conditions from \cite[Prop.~7.2]{ACS2024}.
\begin{enumerate}
\item The curvature of $S$ should behave like $\beta^{-1}$ and indeed we find qualitatively that for increasing $\beta$ the curvature of $S\restr\Omega$ decreases, see Figure~\ref{fig:image_C_curv_S}.
\item We also observe a surface $T$ that has a non-trivial components $T\restr\M$ and $T\restr\Omega$ and detaches from $\M$ in a line other than $\Gamma$, see Figure~\ref{fig:image_C_surf_T_Ome_M}.
We observe that the angle formed by the normal vectors of $\M$ and $T$ do not form a right angle as predicted by Young's law.
This can also be observed in Figure~\ref{fig:image_C_curv_T} (a).
\item The surface $T\restr\Omega$ in Figure~\ref{fig:image_C_view} and \ref{fig:image_C_surf_T_Ome_M} appears to be flat, verifying the optimality condition of vanishing mean curvature.
In Figure~\ref{fig:image_C_curv_T} we give an example where $T\restr\Omega$ has non-vanishing curvature but the curvature in the two depicted slices have opposite sign.
\end{enumerate}

\begin{figure}
\begin{center}
\begin{subfigure}[c]{0.32\textwidth}
\centering
\includegraphics[scale=0.15]{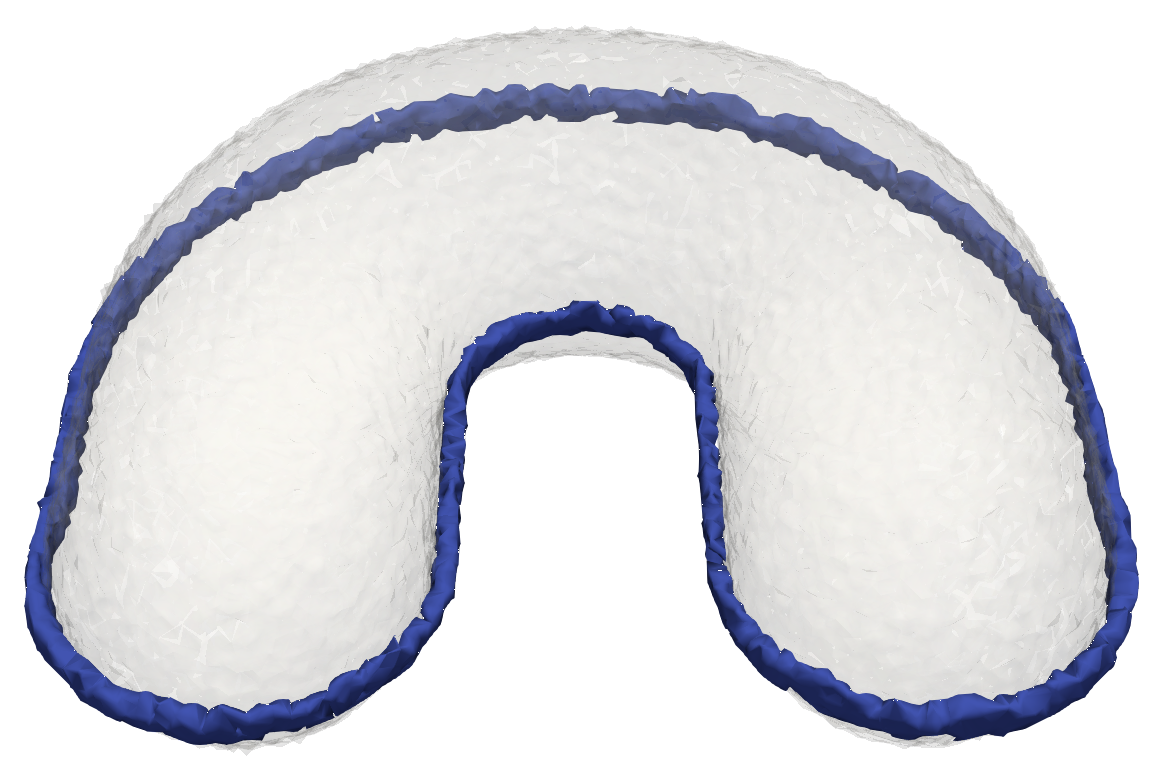}
\caption{} 
\label{fig:image_C_view-a}
\end{subfigure}
\begin{subfigure}[c]{0.32\textwidth}
\centering
\includegraphics[scale=0.15]{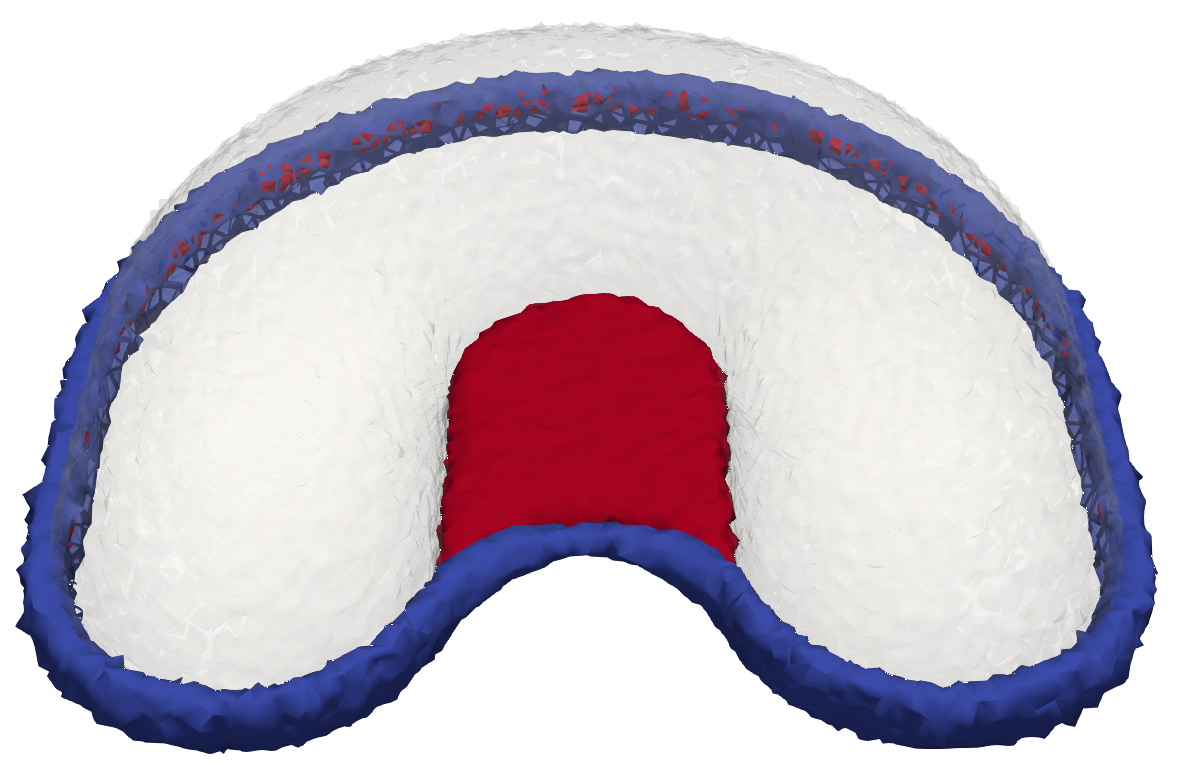}
\caption{} 
\label{fig:image_C_view-b}
\end{subfigure}
\begin{subfigure}[c]{0.32\textwidth}
\centering
\includegraphics[scale=0.15]{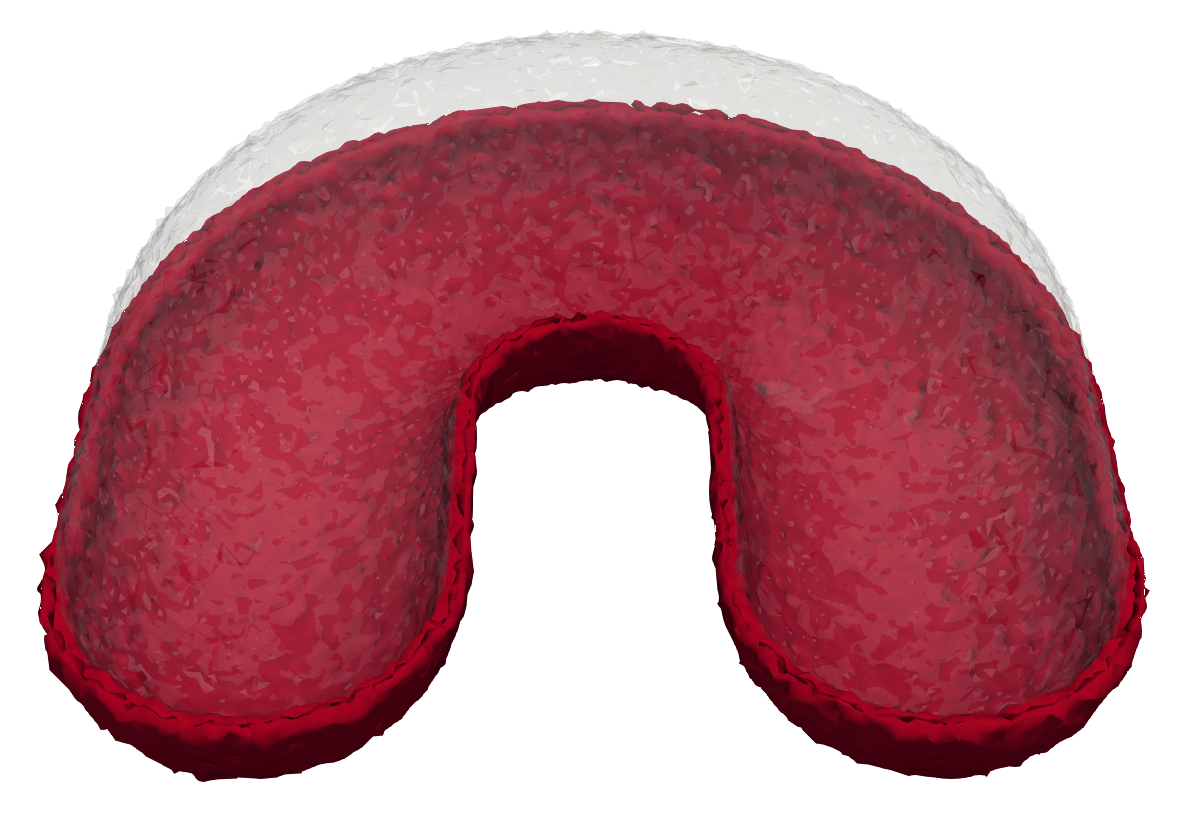}
\caption{} 
\label{fig:image_C_view-c}
\end{subfigure}
\end{center}
\caption{Configurations obtained for $\beta=0.01,0.33,0.67$ and $\phi = \psi = 0$ after $4\,000$ iterations.}
\label{fig:image_C_view}
\end{figure}

\begin{figure}
\begin{center}
\begin{subfigure}[c]{0.24\textwidth}
\centering
\includegraphics[scale=0.17]{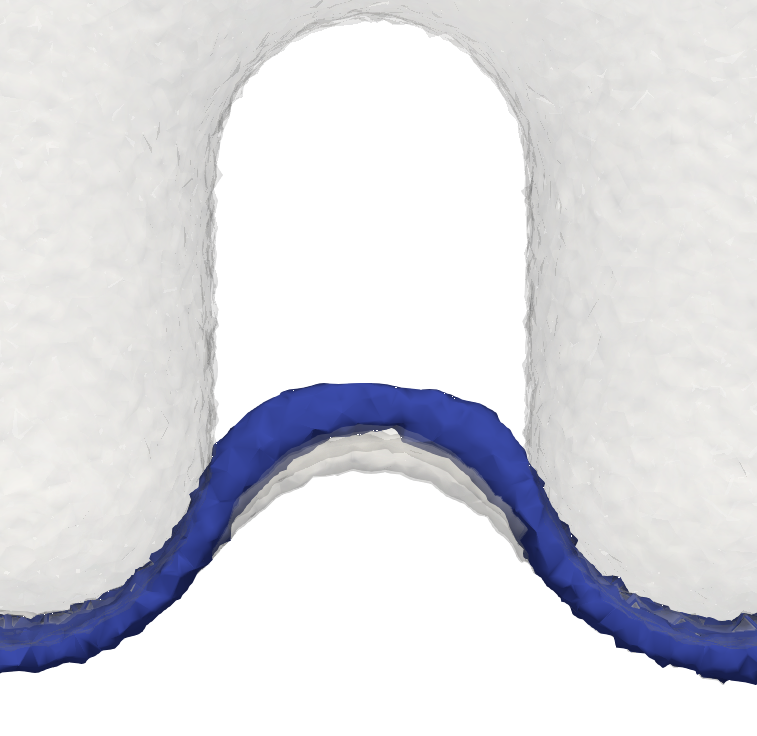}
\caption{} 
\label{fig:image_C_curv_S-a}
\end{subfigure}
\begin{subfigure}[c]{0.24\textwidth}
\centering
\includegraphics[scale=0.17]{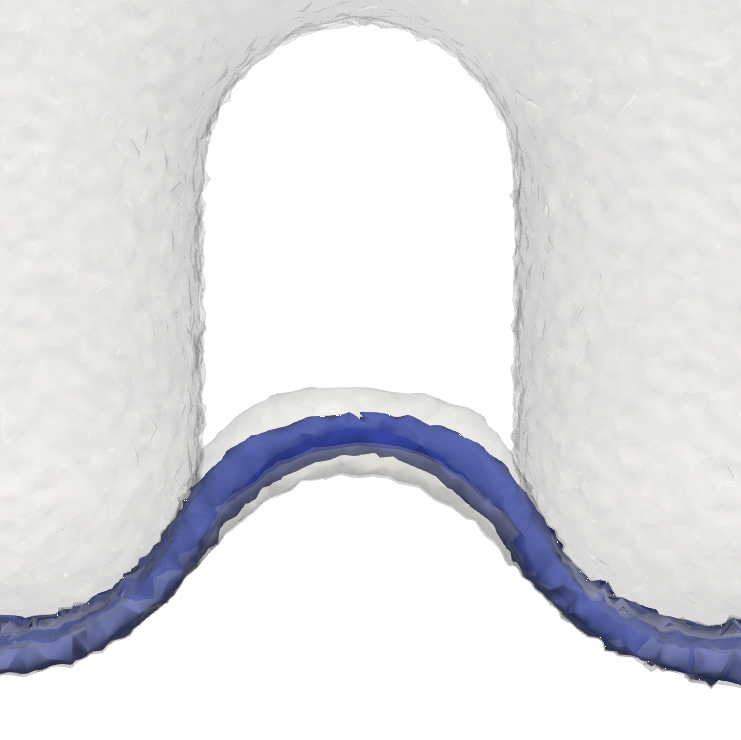}
\caption{} 
\label{fig:image_C_curv_S-b}
\end{subfigure}
\begin{subfigure}[c]{0.24\textwidth}
\centering
\includegraphics[scale=0.17]{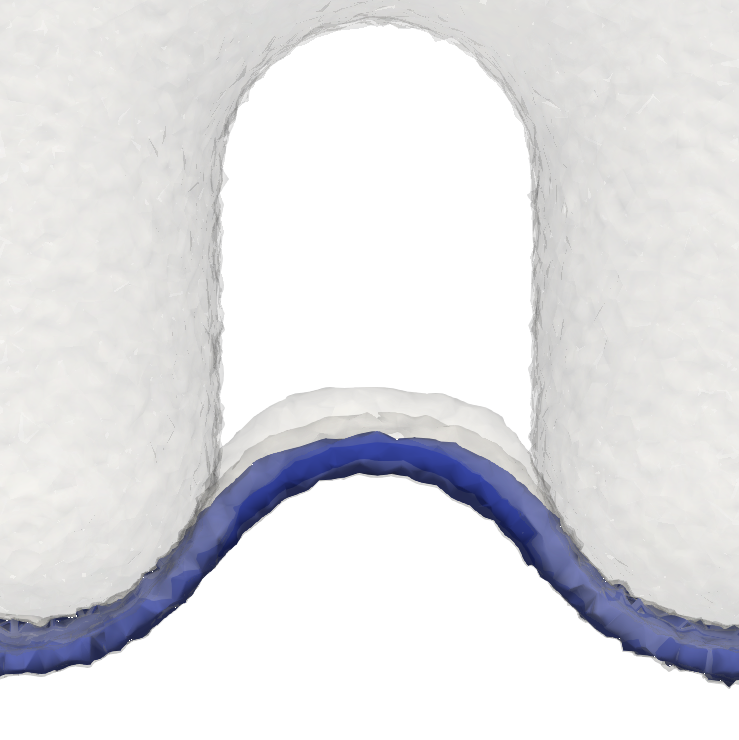}
\caption{} 
\label{fig:image_C_curv_S-c}
\end{subfigure}
\begin{subfigure}[c]{0.24\textwidth}
\centering
\includegraphics[scale=0.17]{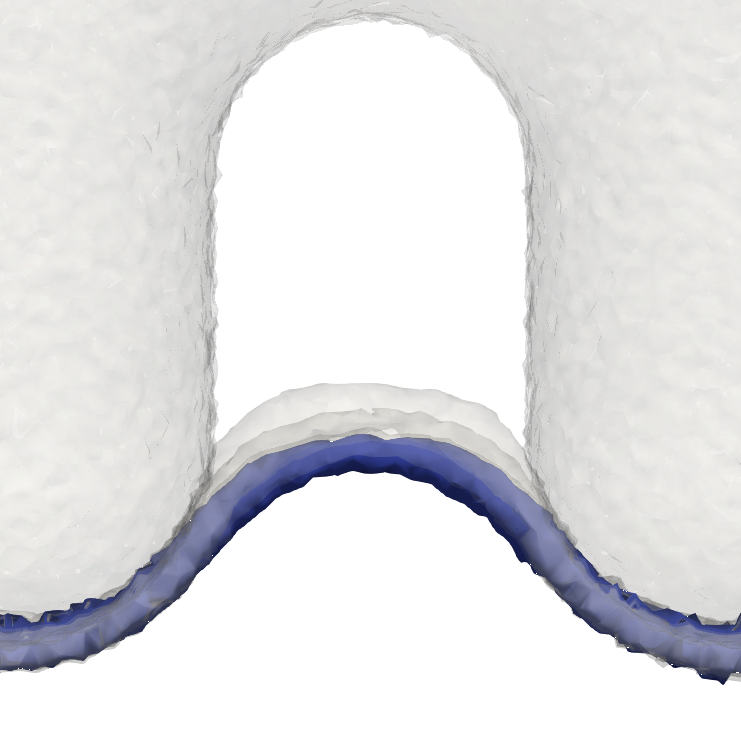}
\caption{} 
\label{fig:image_C_curv_S-d}
\end{subfigure}
\end{center}
\caption{Evolution of the singularity line $S\restr\Omega$ as $\beta$ increases, from left to right $\beta=0.33,0.375,0.39,0.4$. 
For comparison all four lines are indicated in all four images.
Configurations obtained for $\phi = \psi = 0$ and $4\,000$ iterations.}
\label{fig:image_C_curv_S}
\end{figure}

\begin{figure}
\begin{center}
\begin{subfigure}[c]{0.49\textwidth}
\centering
\includegraphics[scale=0.20]{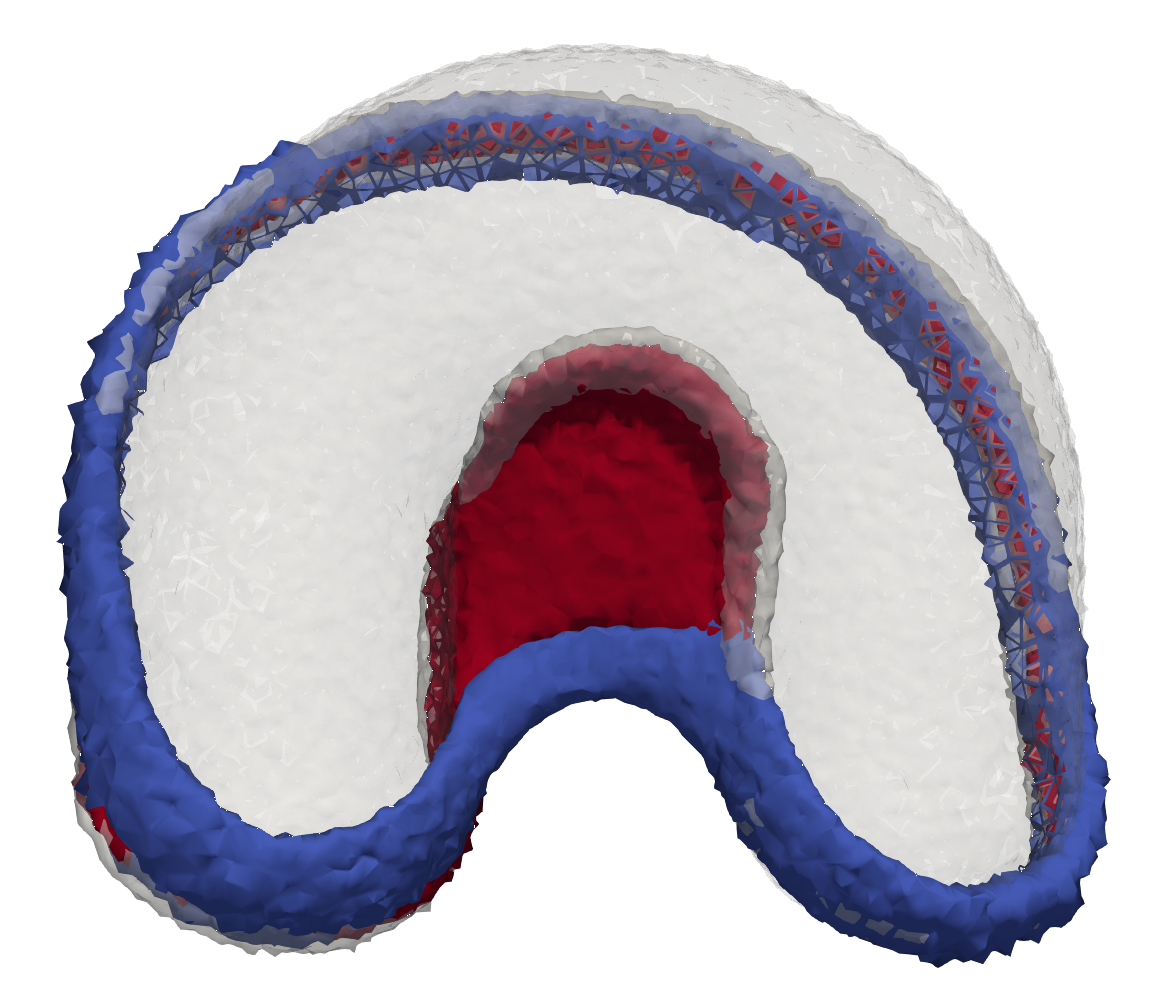}
\caption{} 
\label{fig:image_C_surf_T_Ome_M-a}
\end{subfigure}
\begin{subfigure}[c]{0.49\textwidth}
\centering
\includegraphics[scale=0.17]{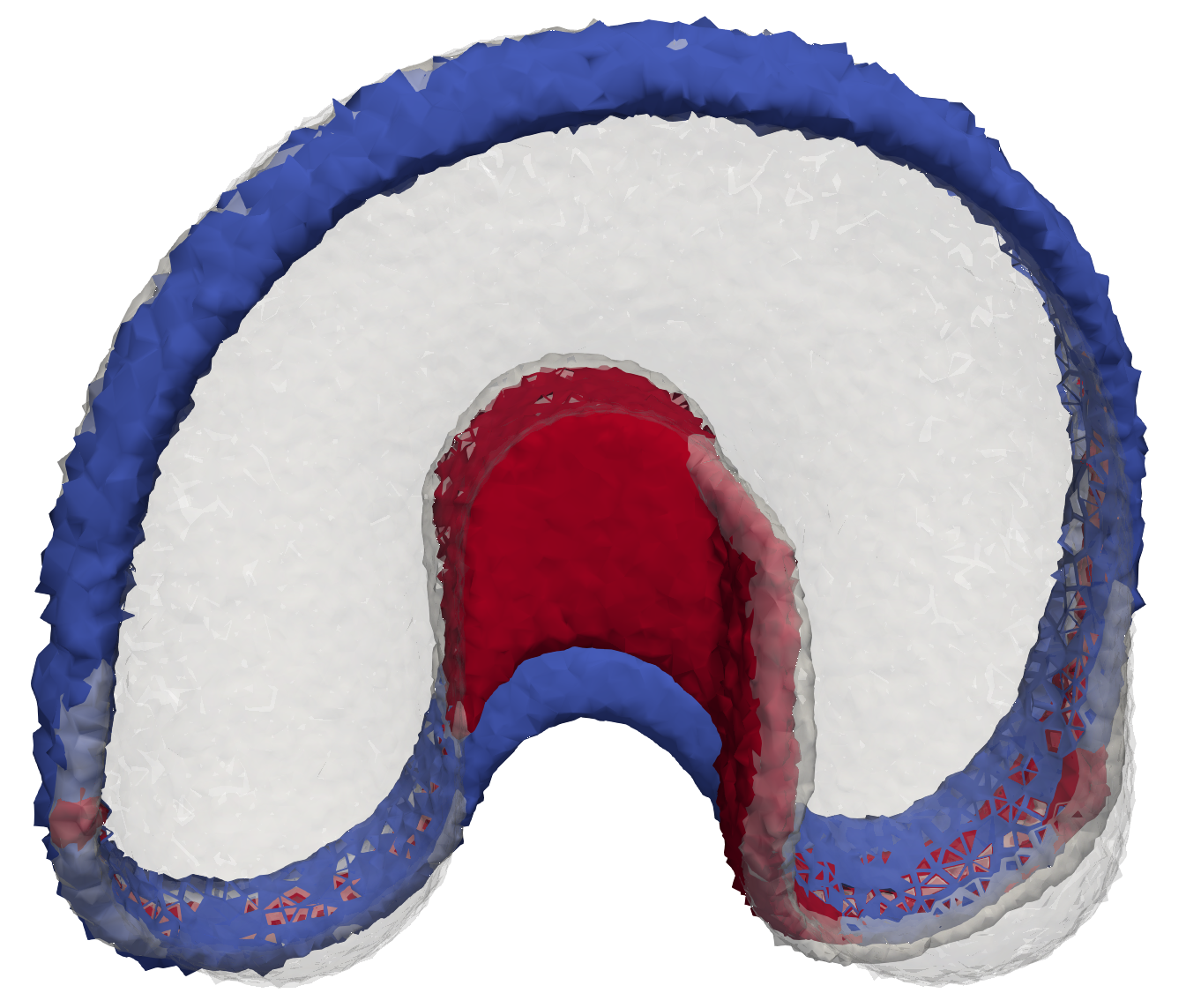}
\caption{} 
\label{fig:image_C_surf_T_Ome_M-b}
\end{subfigure}
\vspace*{0.5cm}
\begin{subfigure}[c]{0.49\textwidth}
\centering
\includegraphics[scale=0.14]{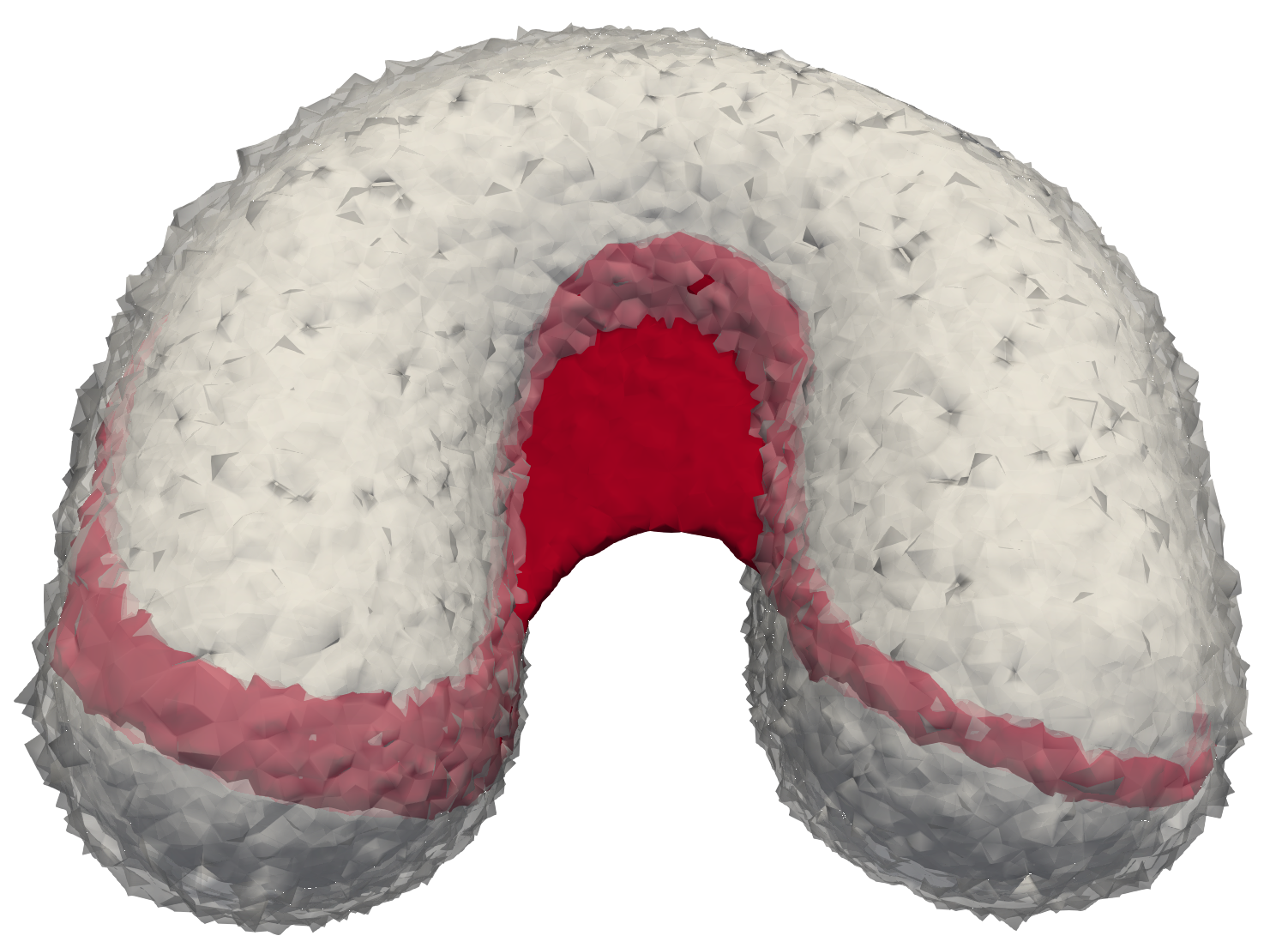}
\caption{} 
\label{fig:image_C_surf_T_Ome_M-c}
\end{subfigure}
\begin{subfigure}[c]{0.49\textwidth}
\centering
\includegraphics[scale=0.14]{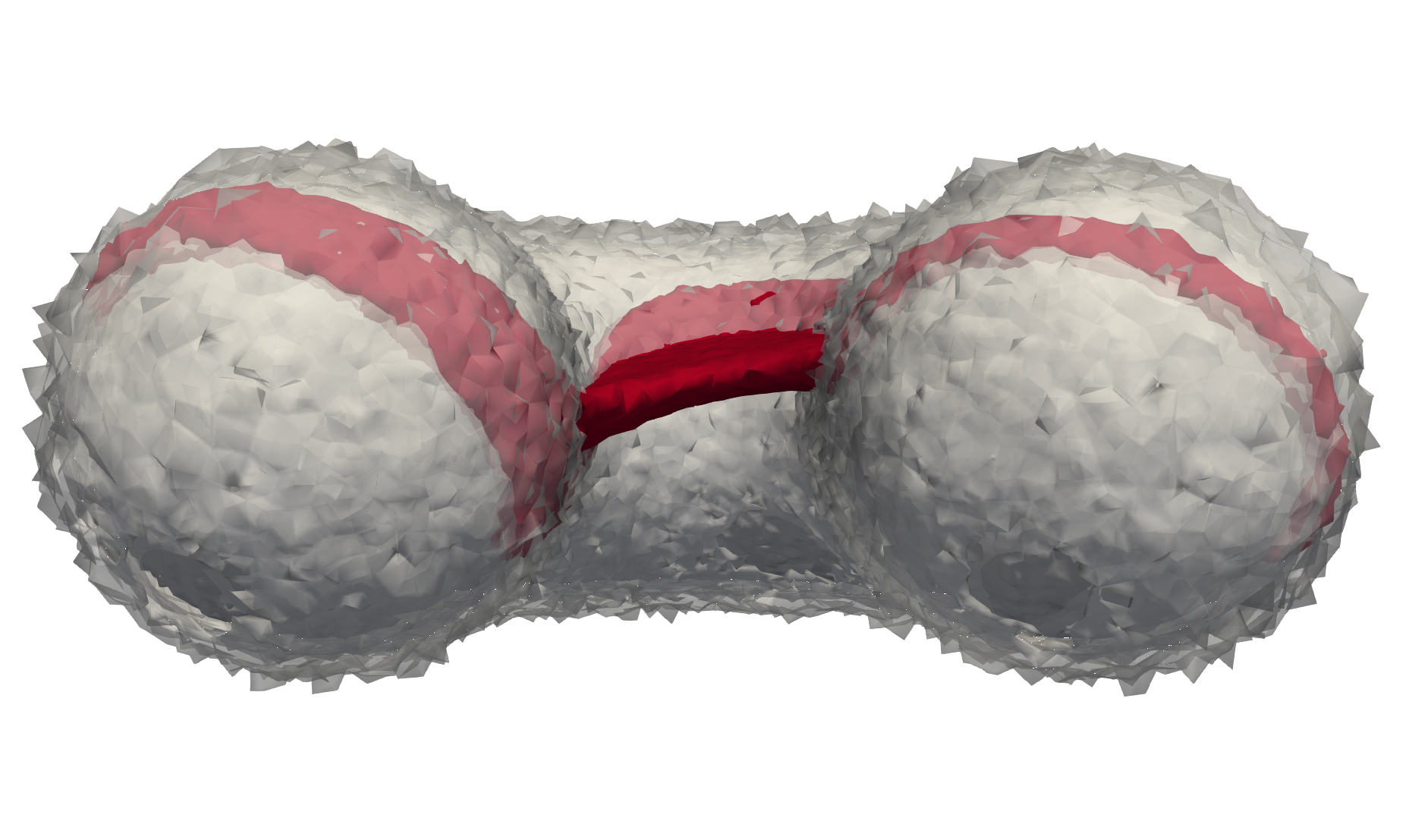}
\caption{} 
\label{fig:image_C_surf_T_Ome_M-d}
\end{subfigure}
\end{center}
\caption{Four views on the configuration after $8\,000$ iterations for $\phi = \frac{\pi}{4}$, $\psi = \frac{\pi}{8}$ and $\beta = 0.31$.
The line $\Gamma$ is indicated  transparently in the images of the upper row.
The bottom row shows the transparent boundary layer $\M_h$ around the solid particle $E$, allowing to distinguish $T\restr\M$ and $T\restr\Omega$.}
\label{fig:image_C_surf_T_Ome_M}
\end{figure}

\begin{figure}
\begin{center}
	\begin{subfigure}[c]{0.43\textwidth}
	\centering
	\includegraphics[scale=0.23]{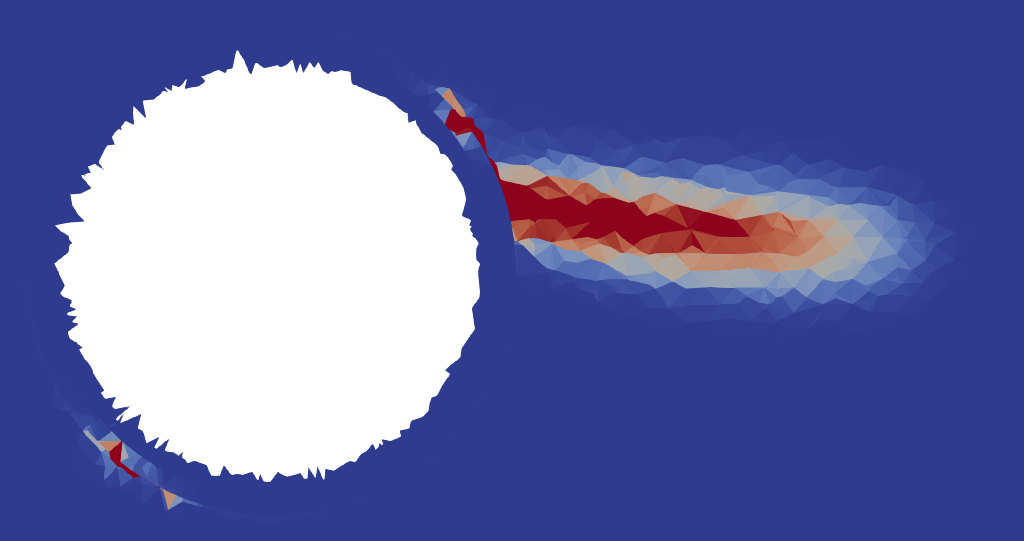}
    \caption{} 
    \label{fig:image_C_curv_T-a}
    \end{subfigure}
	\begin{subfigure}[c]{0.56\textwidth}
	\centering
	\includegraphics[scale=0.23]{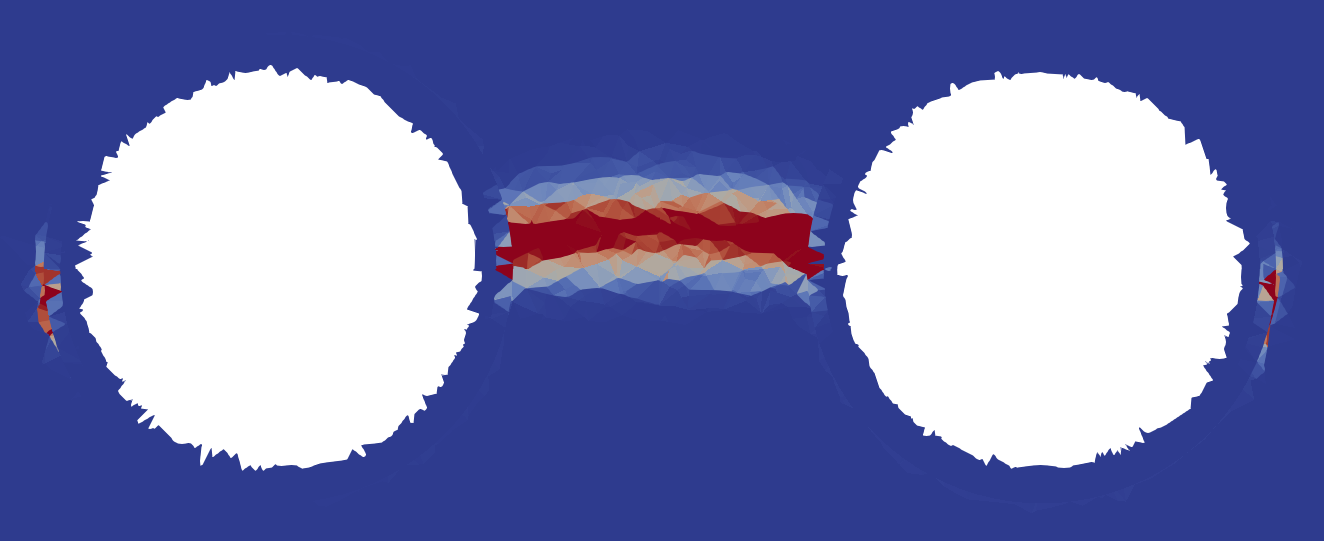}
    \caption{} 
    \label{fig:image_C_curv_T-b}
    \end{subfigure}
\end{center}
\caption{Section though the $x_2x_3-$plane (left) and $x_1x_3-$plane(right) of $T$ after $4\,000$ iterations for $\phi = \frac{\pi}{4}$, $\psi = 0$ and $\beta = 0.3$.
The part of $T$ inside $\Omega$ is curved into opposite directions in the two images.}
\label{fig:image_C_curv_T}
\end{figure}



\paragraph{Acknowledgment.} The author would like to thank Christophe Geuzaine for his help in the mesh generation with GMSH and Fran\c{c}ois Alouges and Antonin Chambolle for their support and the discussions as part of the author's PhD thesis at École Polytechnique, where the main part of this work was carried out.




\addcontentsline{toc}{section}{References}
\bibliography{LC_references}{}

\end{document}